\documentclass[reqno, 12pt]{amsart} 

\usepackage[expansion=false]{microtype} 
\usepackage{amsfonts,amsthm,amsmath,amssymb,amscd,mathrsfs} 
\allowdisplaybreaks
\usepackage{latexsym} 
\usepackage[colorlinks=true,linkcolor=blue,citecolor=red,urlcolor=black]{hyperref} 
\usepackage{graphicx}
\usepackage{indentfirst} 
\usepackage{cite} 
\usepackage{color} 
\usepackage{lmodern} 

\theoremstyle{plain} 
\newtheorem{Thm}{Theorem}[section] 
\newtheorem{Lem}[Thm]{Lemma}     
\newtheorem{Prop}[Thm]{Proposition}

\theoremstyle{definition}

\theoremstyle{remark}
\newtheorem{Rem}[Thm]{Remark}

\numberwithin{equation}{section} 

\newcommand{\beq}{\begin{equation}}
	\newcommand{\eeq}{\end{equation}}
\newcommand{\ben}{\begin{eqnarray}}
	\newcommand{\een}{\end{eqnarray}}
\newcommand{\beno}{\begin{eqnarray*}}
	\newcommand{\eeno}{\end{eqnarray*}}
\newcommand{\no}{\nonumber}
\newcommand{\lt}{\left}
\newcommand{\rt}{\right}

\newcommand{\px}{\partial_x}
\newcommand{\py}{\partial_y}

\newcommand{\pt}{\partial_t}

\newcommand{\qqquad}{\qquad\quad}

\newcommand{\alabel}{\stepcounter{equation}\tag{\theequation}\label}

\usepackage[
    letterpaper,                 
    textheight=8.35in,           
    headsep=20pt,                
    footskip=36pt,               
    marginparwidth=0pt,          
    marginparsep=0pt,            
    left=0.75in, 
    right=0.75in
]{geometry}

\begin{document}
	
	\title[Quantitative blow-up suppression for the PKS(NS) system]{Quantitative blow-up suppression for the Patlak--Keller--Segel(--Navier--Stokes) system via Couette flow on $\mathbb{R}^2$}	
    
	\author{Yubo~Chen}
	\address[Yubo~Chen]{School of Mathematical Sciences, Dalian University of Technology, Dalian, 116024,  China}
	\email{1220823215@mail.dlut.edu.cn}
	
	\author{Wendong~Wang}
	\address[Wendong~Wang]{School of Mathematical Sciences, Dalian University of Technology, Dalian, 116024,  China}
	\email{wendong@dlut.edu.cn}
	
	\author{Guoxu~Yang}
	\address[Guoxu~Yang]{School of Mathematical Sciences, Dalian University of Technology, Dalian, 116024,  China}
	\email{guoxu\_dlut@outlook.com}
    
	\begin{abstract}
        It is well known that solutions to the Patlak--Keller--Segel system on $\mathbb{R}^2$ blow up in finite time if the initial mass exceeds $8\pi$. In this paper, we investigate the mixing effect induced by a Couette flow $(Ay, 0)$ with a quantitatively determined amplitude $A$, which suppresses bacterial aggregation. For the Patlak--Keller--Segel system advected by such a flow on $\mathbb{R}^2$, we prove that the solutions remain global in time even for large initial mass, provided the amplitude $A$ is sufficiently large. Specifically, global well-posedness holds if $A$ satisfies a lower bound of the form $C_* \left(\| \langle D_x\rangle^{m} \langle {D_x}^{-1}\rangle^\epsilon n_{\mathrm{in}} \|_{L^2}^2+1\right)^{9/2}$. A notable feature of our result is the explicit estimate of the sufficient constant, given by $C_* = 2,058,614$. Furthermore, for the coupled Patlak--Keller--Segel--Navier--Stokes system near the Couette flow, we establish an analogous global existence result, provided the amplitude is sufficiently large in form of $C_*\|(n_{\rm in}, |D_x|^{1/3} n_{\rm in},\omega_{\rm in})\|_{Y_{m,\epsilon}}^9$.
	\end{abstract}

	\maketitle
	
	
	\tableofcontents

	\section{Introduction}  
	Consider the following Patlak--Keller--Segel (PKS) system coupled with the Navier--Stokes (NS) equations on $\mathbb{R}^2$:
	\begin{equation}\begin{aligned} \label{eq:main0}
			\left\{\begin{array}{l}
				\partial_t n+v \cdot \nabla n=\Delta n-\nabla \cdot(n \nabla c), \\
				\beta (\pt c  + v\cdot \nabla c) = \Delta c+n- \alpha c , \\
				\partial_t v+v \cdot \nabla v+\nabla p=\Delta v+n \nabla \phi, \\
				\nabla \cdot v=0,
			\end{array}\right.
	\end{aligned}\end{equation}
    where $n$ denotes the microorganism density, $c$ the chemoattractant density and $v$ the velocity of the fluid. Moreover, $p$ is the scalar pressure, $\alpha \geq 0$, $\beta \geq0 $, and $\phi$ is a given potential function accounting the effects of external forces. For simplicity, we focus on the parabolic-elliptic case $\beta =0$ and $\phi(x, y)=y$ with initial conditions 
    $$
    \left.(n, \, v)\right|_{t=0}= (n_{\mathrm{in}}, \, v_{\mathrm{in}} ).
    $$

	\subsection{Developments}
	If $v=0$ and $\phi=0$, the system \eqref{eq:main0} is reduced to the classical Patlak--Keller--Segel system, and the origin of the PKS model is a classic narrative of interdisciplinary science derived by Patlak \cite{P1953} and Keller--Segel \cite{KS1970}. There exists a vast body of literature on the classical PKS system, and we briefly summarize some relevant developments. The parabolic-elliptic PKS system in 2D is globally well-posed if and only if the initial mass $M:=||n_{\rm in}||_{L^1}\leq8\pi$; see Blanchet--Dolbeault--Perthame \cite{BDP2006} or Wei \cite{Wei2018} for the details. When $M=8\pi$, the solution may blow up at infinity, see Blanchet--Carrillo--Masmoudi \cite{BCM2008}; see also Davila--del Pino--Dolbeault--Musso--Wei \cite{DDDMW2024} for the blow-up rate at infinity. The 2D parabolic-parabolic PKS model also has a critical mass of $8\pi$; see Calvez--Corrias \cite{CC2008} and Schweyer \cite{S2014}. In higher dimension ($d \geq 3$), the solution with any initial mass may blow up in finite time; see Nagai \cite{Na2000}, Souplet--Winkler \cite{SW2019}, and Winkler \cite{W2013}.
    
	Since the chemotactic processes often take place in a moving fluid, it is interesting to consider the mixing effect of the fluid. 
	
	{\bf Without coupling the fluid equations.}
	With an additional advection term modeling
	ambient fluid flow, 
	\begin{equation*}\begin{aligned}  
			\left\{\begin{array}{l}
				\partial_t n+v \cdot \nabla n=\Delta n-\nabla \cdot(n \nabla c), \\
				-\Delta c=n- \bar{n} ,
			\end{array}\right.
	\end{aligned}\end{equation*}
	Kiselev--Xu \cite{KX2016} showed the $8\pi$ critical mass threshold vanishes, as the blow-up can be suppressed by stationary relaxation enhancing flows and by Yao--Zlatos flows in $\mathbb{T}^d$ ($d=2~\text{or} ~3$). Also, Shi--Wang \cite{SW2025} considered the suppression effect of the Couette-Poiseuille flow $(z,z^2,0)$ in $\mathbb{T}^2\times\mathbb{R}$. For the parabolic-parabolic PKS system, He \cite{H2025-1} introduced a family of time-dependent alternating shear flows in $\mathbb{T}^3$, and proved the solution remains globally regular as long as the flow is sufficiently strong. As for the shear flow $v$, Bedrossian--He \cite{BH2017} obtained the suppression effect of blow-up by non-degenerate shear flows in $\mathbb{T}^2$ for the parabolic-elliptic case. He \cite{H2018} investigated the suppression of blow-up for the parabolic-parabolic PKS model near the large strictly monotone shear flow in $\mathbb{T}\times\mathbb{R}$. Deng--Shi--Wang \cite{DSW2025} proved the Couette flow with a sufficiently large amplitude prevents the blow-up of solutions in the whole space $\mathbb{R}^3$ for exponential decay data. For the parabolic-parabolic PKS system near a time-dependent shear flow, He \cite{H2025-2} demonstrated  the total
	mass of the cell density is below a specific threshold ($8\pi|\mathbb{T}|$) in $\mathbb{T}^3$.

	{\bf With coupling the fluid equations.}
	For the coupled PKS--NS system, Zeng--Zhang--Zi \cite{ZZZ2021} firstly considered the 2D PKS--NS system near Couette flow in $\mathbb{T}\times\mathbb{R}$, and proved that the solution stays globally regular provided that the Couette flow is sufficiently strong. He \cite{H2023} considered the blow-up suppression for the parabolic-elliptic PKS--NS system in $\mathbb{T}\times\mathbb{R}$ with the coupling of buoyancy effects for a class of initial data with small vorticity.
	Wang--Wang--Zhang \cite{WWZ2024}  studied the blow-up suppression of 2D supercritical parabolic-elliptic PKS--NS system near the Couette flow in $\mathbb{T}\times\mathbb{R}.$ 
	Li--Xiang--Xu \cite{LXX2025} suppressed the blow-up for the PKS--NS system via the Poiseuille flow  in $\mathbb{T}\times\mathbb{R}.$
	Furthermore, Cui--Wang \cite{CW2024} considered the blow-up suppression for the PKS--NS system in $\mathbb{T}\times \mathbb{I}$ with Navier-slip boundary condition.
	For the 3D PKS--NS system, Cui--Wang--Wang \cite{CWW2025} considered the PKS system coupled with
	the linearized NS equations near the Couette flow in $\mathbb{T}\times\mathbb{I}\times\mathbb{T}.$ 
	Then, they \cite{CWW2024} also studied the blow-up suppression and the nonlinear stability of the PKS--NS system via the Couette flow in an unbounded domain $\mathbb{T}\times\mathbb{R}\times\mathbb{T},$ and proved that the solutions are global in time as long as the initial cell mass $M< \tfrac{24}{5}\pi^2$. Recently, when the total mass of the cell density is below a specific threshold ($8\pi|\mathbb{T}|$), Cui--Wang--Wang--Wei \cite{CWWW2025} obtained the solution remains globally regular in $\mathbb{T}^3$. Moreover, by combining the effects of the Couette flow and the logistic source, Cui--Wang--Wang--Wei--Yang \cite{CWWWY2025} obtained the solution remains bounded globally  in $\mathbb{T}\times\mathbb{I}\times\mathbb{T}$ with no-slip boundary condition. 
    
    It is worth noting that the above results primarily concern periodic domains, where the enhanced dissipation plays a crucial role. A natural question then arises:
    \begin{center}
    \textit{Do similar conclusions hold for non-periodic domains?}
    \end{center}
	
	
	Next we investigate these issues. Motivated by the stability threshold established in $\mathbb{R}^2$ by Arbon--Bedrossian \cite{AB2024} and Li--Liu--Zhao \cite{LLZ2025}, we consider the Patlak--Keller--Segel system coupled with the incompressible Navier--Stokes equations on $\mathbb{R}^2$. It is easy to verify that $(\tilde{n}, \tilde{c}, \tilde{v})=(0,0, U)$ with the Couette flow $U=(A y, 0)$ is a stationary solution to the system \eqref{eq:main0}. Denote the perturbation $u(t, x, y)=v(t, x, y)-U(y)$, then $(n, c, u)$ satisfies
	\begin{equation}\begin{aligned} \label{eq:main1}
			\left\{\begin{array}{l}
				\partial_t n+A y \partial_x n+u \cdot \nabla n-\Delta n=-\nabla \cdot(n \nabla c), \\
				-\Delta c = n - \alpha c, \\
				\partial_t u+A y \partial_x u+\binom{A u_2}{0}+u \cdot \nabla u-\Delta u+\nabla p=\binom{0}{n}, \\
				\nabla \cdot u=0 ,
			\end{array}\right.
	\end{aligned}\end{equation}
    with initial conditions 
    $$
    \left.(n, \, u)\right|_{t=0}= (n_{\mathrm{in}}, \, u_{\mathrm{in}} ).
    $$
	Define
	$$
	\omega = \partial_y u_1-\partial_x u_2, \quad u=\nabla^{\perp} \Phi=\left(\partial_y \Phi,-\partial_x \Phi\right),
	$$
	then
	$
	\Delta \Phi=\omega
	$
	and $\omega$ satisfies
	$$
	\partial_t \omega+A y \partial_x \omega-\Delta \omega+u \cdot \nabla \omega=-\partial_x n .
	$$
	After the time rescaling $t \mapsto A^{-1} t$, one can rewrite the system \eqref{eq:main1} by
	\begin{equation}\begin{aligned} \label{eq:main}
			\left\{\begin{array}{l}
				\partial_t n+y \partial_x n-\frac{1}{A} \Delta n=-\frac{1}{A}(\nabla \cdot(n \nabla c)+u \cdot \nabla n), \\
				- \Delta c=  n - \alpha  c, \\
				\partial_t \omega+y \partial_x \omega-\frac{1}{A} \Delta \omega=-\frac{1}{A}\left(\partial_x n+u \cdot \nabla \omega\right), \\
				u=\nabla^{\perp} \Delta^{-1} \omega .
			\end{array}\right.
	\end{aligned}\end{equation}
	Similarly, without coupling the fluid equations, we have 
	\begin{equation}\begin{aligned} \label{eq:main2}
			\left\{\begin{array}{l}
				\partial_t n+y \partial_x n-\frac{1}{A} \Delta n=-\frac{1}{A}\nabla \cdot(n \nabla c), \\
				- \Delta c=  n - \alpha  c.
			\end{array}\right.
	\end{aligned}\end{equation}

	\subsection{Stability mechanisms}

	Note that in a sense, the global well-posedness of the systems above is similar to the stability threshold analysis of the Navier--Stokes equations around the Couette flow:
	\begin{equation*}
		\left\{
		\begin{array}{lr}
			\partial_{t}u+y\partial_x u-\frac{1}{Re}\Delta u+u\cdot\nabla u+\nabla p=0, \\
			\nabla\cdot u=0,\\
			u|_{t=0}=u_{\rm in}(x,y).
		\end{array}
		\right.
	\end{equation*}
	The transition threshold problem was formulated by Bedrossian--Germain--Masmoudi \cite{BGM2017} as follows:
	Given a norm $\|\cdot\|_X$, determine a $\gamma=\gamma(X)$ so that
	$$
	\begin{aligned}
		& \left\|u_{\rm in}\right\|_X \leq R e^{-\gamma} \Rightarrow \text { stability, } \\
		& \left\|u_{\rm in}\right\|_X \gg R e^{-\gamma} \Rightarrow \text { instability. }
	\end{aligned}
	$$
	The exponent $\gamma$ is referred to as the transition threshold. It is worth pointing out that the enhanced dissipation effect and the inviscid damping effect of the Couette flow are of great importance to stabilize the NS equations at high Reynolds numbers (i.e., $Re = \nu^{-1}$, where $\nu$ is the fluid viscosity coefficient).
	The stability of the 2D Couette flow for the NS equation has been widely studied. Let us briefly recall some related results. When $\Omega=\mathbb{T} \times \mathbb{R}$, if the perturbation for vorticity is in Gevrey-$s$ with $s < 2$, Bedrossian--Masmoudi--Vicol \cite{BMV2016} showed that the solution is stable for $\gamma =0$. Bedrossian--Vicol--Wang \cite{BVW2018} proved that Couette flow is stable in Sobolev norms for $\gamma \leq \tfrac12$ by Fourier multiplier method.  The sharp case of $\gamma \leq \tfrac13$ was proved by Masmoudi--Zhao \cite{MZ2022} and Wei--Zhang \cite{WZ2023}, respectively. In $ \mathbb{T}\times[-1,1]$, $\gamma \leq \tfrac12$ was proved by Chen--Li--Wei--Zhang \cite{CLWZ2020} with no-slip boundary condition by establishing some sharp resolvent estimates. When $\Omega=\mathbb{R} \times \mathbb{R}$, Arbon--Bedrossian \cite{AB2024} obtained $\gamma \leq \tfrac{1}{2}+$ and  Li--Liu--Zhao \cite{LLZ2025} improved their result to $\gamma \leq \tfrac{1}{3}+$, which was generalized to the Boussinesq setting by the authors in \cite{CWY2025}. For further references,
    we refer the reader to \cite{WZ2021,CWZ2024,DWZ2021} and the references therein.


	Different from the domain $\mathbb{T}\times\mathbb{R}$, one needs to be careful when dealing with the horizontal frequency in the Fourier space, because of the singularity as the horizontal frequency tends to zero. We will briefly introduce these two stability mechanisms. In the enhanced dissipation estimate, the dissipation time scale is $O\left(\nu^{-{1}/{3}}\right)$, which is much faster than the standard heat dissipation time scale $O\left(\nu^{-1}\right)$ for small $\nu$. Clearly the dissipation rate is inhomogeneous and depends on the frequencies $k$. The phenomenon of enhanced dissipation has been widely observed and studied in physics literature (see, e.g., \cite{BL1994,RY1983,T1887}). It has recently attracted enormous attention from the mathematics community and significant progress has been made. One of the earliest rigorous results on the enhanced dissipation is obtained by Constantin--Kiselev--Ryzhik--Zlatos \cite{CKRZ2008} on the enhancement of diffusive mixing. Inviscid damping is analogous to Landau damping in plasma physics found by Landau \cite{L1946}. For general shear flows, establishing linear inviscid damping remains a challenging problem, while the nonlinear inviscid damping is even more difficult to prove. Inspired by the groundbreaking work of Mouhot--Villani \cite{MV2011} on nonlinear Landau damping, Bedrossian--Masmoudi \cite{BM2015} established nonlinear inviscid damping for the Couette flow in the domain $ \mathbb{T} \times \mathbb{R}$. We refer to \cite{O1907, MZ2024, IJ2023,DM2023} and the references therein for more progress.

	\subsection{Main results}
	In the following, we focus on the stability analysis of the system \eqref{eq:main} and \eqref{eq:main2} on $\mathbb{R}^2$. To state the main theorems, we first introduce some notation. Denote $D_x := -i \partial_x$, understood as a Fourier multiplier operator. All norms involving $D_x$ are defined on the Fourier side and $\langle \cdot \rangle := \sqrt{1 + (\cdot)^2}$. For $t \geq 0$ and certain smooth enough $f$, define the space-time norm
		\begin{align*}
			\|f\|_{X}^2&=\|f\|_{X(\Xi,\theta_1,\theta_2,a)}^2\\
			&:=  \|  e^{a A^{-\frac{1}{3}}\left|D_x\right|^{\frac{2}{3} }t}\langle D_x\rangle^m\langle\frac{1}{D_x}\rangle^\epsilon f\|_{L^\infty L^2}^2      + \lt(2 - \frac{\theta_2\Xi}{2\pi}\rt)\frac{1}{A} \|e^{a A^{-\frac{1}{3}}\left|D_x\right|^{\frac{2}{3} }t}\langle D_x\rangle^m\langle\frac{1}{D_x}\rangle^\epsilon \py f\|_{L^2 L^2}^2   \\
			& \quad + \lt(\frac{\theta_1 \Xi}{2\pi}  - 2 a\rt)\frac{1}{ A^{\frac{1}{3}}}\|e^{a A^{-\frac{1}{3}}\left|D_x\right|^{\frac{2}{3} } t }\langle D_x\rangle^m\langle\frac{1}{D_x}\rangle^\epsilon\left|D_x\right|^{\frac{1}{3}} f\|_{L^2 L^2}^2 \\
			&\quad +  \frac{2}{A} \|e^{a A^{-\frac{1}{3}}\left|D_x\right|^{\frac{2}{3} }t}\langle D_x\rangle^m\langle\frac{1}{D_x}\rangle^\epsilon \px f\|_{L^2 L^2}^2+ \frac{\Xi}{2\pi} \|e^{a A^{-\frac{1}{3}}\left|D_x\right|^{\frac{2}{3} }t }\langle D_x\rangle^m\langle\frac{1}{D_x}\rangle^\epsilon \partial_x \nabla  \Delta^{-1} f \|_{L^2 L^2}^2,\alabel{eq:X norm}
		\end{align*}
	where 
	\begin{equation} \label{eq:positivity conditions}
		\left\{
		\begin{array}{lr}
			0<\theta_1, \theta_2 \leq 1, \quad \theta_1 \leq 2\sqrt{\theta_2} -\theta_2, \\
			0<\theta_2 \Xi<4\pi,\quad 0<4\pi a<\theta_1\Xi.
		\end{array}
		\right.
	\end{equation}\

	For the PKS system near a large Couette flow \eqref{eq:main2}, our main result is stated as follows.
\begin{Thm} \label{thm:PKS}
For $0<\epsilon< 1/2<m$, assume that the initial data $n_{\mathrm{in}} \in L^\infty(\mathbb{R}^2) \cap L^1(\mathbb{R}^2)$ satisfying
		\begin{equation*}\begin{aligned}
			  \|n_{\rm in}\|_{Y_{m, \epsilon}} := \|  \langle D_x\rangle^{m}  \langle \frac{1}{D_x}\rangle^\epsilon   n_{\mathrm{in}} \|_{L^2} < \infty.
		\end{aligned}\end{equation*}
		There exists a positive constant $\Lambda$ depending on $\epsilon, m,\alpha, \theta_1, \theta_2, \Xi$ and $a$ 
        such that if  
        $$
        A>\Lambda\left( \|n_{\rm in}\|_{Y_{m, \epsilon}}^2+1\right)^{\frac92},
        $$ 
        the solutions to \eqref{eq:main2} are global in time. Additionally, $\epsilon>1/3$ is required when $\alpha = 0.$
        
   In fact, the value of $\Lambda$ can be calculated rigorously.  For example, consider $a=  (2000\pi)^{-1}$ and $\Xi=\theta_1 = \theta_2=1$, which satisfy the positivity conditions \eqref{eq:positivity conditions}.
    \begin{itemize}
        \item \textbf{Case of $\alpha>0$}: one takes $\epsilon=1/4$, $m=9/{10}$
        and has
        \begin{equation*}
            \begin{aligned}
                \Lambda \leq  \max\lt\{389256 \alpha^{-\frac12},~ 239197\alpha^{-\frac34}\rt\}.
            \end{aligned}
        \end{equation*}
        \item \textbf{Case of $\alpha=0$}: one takes $\epsilon=5/{12}$, $m=7/{10}$ and has
        \begin{equation*}
            \begin{aligned}
                \Lambda \leq  2,058,614.
            \end{aligned}
        \end{equation*}
    \end{itemize}
\end{Thm}
	\begin{Rem} \label{rem:PKS1}
   It is well-known that Couette flow with large amplitude $A$ can suppress the blow-up of the PKS system in a periodic domain, as shown in \cite{BH2017} and \cite{ZZZ2021}. However, it remains unknown whether there exists a critical threshold value for $A$. Here, we consider the whole plane and carefully calculate the upper bound of $A$ for the first time. 
   The value of $\Lambda$ in the above theorem may not be sharp. Determining the optimal value for the amplitude $A$ is an interesting problem, which will be investigated in our future work.
    \end{Rem}

\begin{Rem} \label{rem:PKS2}
    The parameter $\alpha$ denotes the rate constant for the acrasin-acrinase reaction in \cite{KS1970}. The above results show an asymptotic behavior of $ \Lambda$ as $\alpha$ approaches infinity, which implies that if the decomposition rate of  the acrasin-acrinase reaction is fast enough, it is difficult for bacteria to concentrate through signals. 
    \end{Rem}
	Since most chemotactic processes  often take place in a moving fluid,
	we consider the  Patlak--Keller--Segel--Navier--Stokes system near the Couette flow \eqref{eq:main}, and have the following results.
	\begin{Thm} \label{Thm}
		For $0<\epsilon<1/2<m$, assume that the initial data $u_{\rm in} \in H^2(\mathbb{R}^2)$, $n_{\mathrm{in}} \in L^\infty(\mathbb{R}^2) \cap L^1(\mathbb{R}^2)$, and $\omega_{\mathrm{in}} = \nabla^{\perp} \cdot u_{\mathrm{in}}$ satisfying
		\begin{equation*}\begin{aligned} 
				&\|(\omega_{\rm in}, n_{\rm in}, |D_x|^\frac13 n_{\rm in} )\|_{Y_{m, \epsilon}}^2\\
                &:=\|  \langle D_x\rangle^m  \langle \frac{1}{D_x}\rangle^\epsilon \omega_{\mathrm{in}}\|_{L^2}^2 + \|  \langle D_x\rangle^{m  }  \langle \frac{1}{D_x}\rangle^\epsilon (1, |D_x|^\frac13)  n_{\mathrm{in}} \|_{L^2}^2  < \infty.
		\end{aligned}\end{equation*}
		Let $\theta_1$, $\theta_2$, $\Xi$ and $a$ satisfy the positivity conditions \eqref{eq:positivity conditions}. Then there exists a positive constant $\bar{\Lambda}$  depending on $\theta_1$, $\theta_2$, $\Xi$, $a$, $\epsilon$, $m$, and $\alpha$
		such that if $$A>\bar{\Lambda} \left(\|(\omega_{\rm in}, n_{\rm in}, |D_x|^{1/3} n_{\rm in})\|_{Y_{m, \epsilon}}^2+1\right)^{9/2},$$ the solutions to \eqref{eq:main} are global in time satisfying
		\begin{equation*}\begin{aligned}
				&\quad \|  \langle D_x\rangle^m  \langle \frac{1}{D_x}\rangle^\epsilon   \omega \|_{X(\Xi,\theta_1,\theta_2,a)}  +  \|  \langle D_x\rangle^{m  }  \langle \frac{1}{D_x}\rangle^\epsilon   n \|_{X(\Xi,\theta_1,\theta_2,a)} + \|  \langle D_x\rangle^{m  }  \langle \frac{1}{D_x}\rangle^\epsilon |D_x|^\frac13  n \|_{X(\Xi,\theta_1,\theta_2,a)}  \\
				&\leq  C\lt(\|  \langle D_x\rangle^m  \langle \frac{1}{D_x}\rangle^\epsilon   \omega_{\mathrm{in}} \|_{L^2}  +  \|  \langle D_x\rangle^{m  }  \langle \frac{1}{D_x}\rangle^\epsilon   n_{\mathrm{in}} \|_{L^2}  + \|  \langle D_x\rangle^{m  }  \langle \frac{1}{D_x}\rangle^\epsilon |D_x|^\frac13   n_{\mathrm{in}} \|_{L^2} \rt) + 1
		\end{aligned}\end{equation*}
		for all $t\geq0$. Especially,   $\epsilon>1/3$ is required for the case of $\alpha = 0.$
	\end{Thm}

	\begin{Rem} \label{rem1.2}
        For given parameters $a,\Xi,\theta_1,\theta_2,\epsilon$ and $m$, one can calculate an explicit relation between amplitude $A$ and the initial data.
		Similarly as Theorem \ref{thm:PKS}, set $a=  (2000\pi)^{-1}$ and $\Xi=\theta_1 = \theta_2=1$. When $\alpha>0$, let $\epsilon=1/4$ and $m=9/{10}$. The conclusion of Theorem \ref{Thm} still holds with the constant $\bar{\Lambda}$ given explicitly by 
        \begin{align*}
           \bar{\Lambda} =   1.013\times 10^6 \max\lt\{1,\alpha^{-\frac34}\rt\}.
        \end{align*}
		When $\alpha=0$, set $\epsilon=5/{12}$ and $m=7/{10}$. The constant $\bar{\Lambda}$ is  given explicitly by
        \begin{align*}
            \bar{\Lambda} =  4.673\times 10^6.
        \end{align*}
        The values above follow from a direct numerical calculation, whose details are provided in Appendix \ref{appB} for completeness.
	\end{Rem}

	\begin{Rem} \label{rem1.2-} 
    One of the unsolved problems of fluid mechanics is the theoretical description of the inception and growth of instabilities in laminar shear flows that lead to transition to turbulence. Generally, the plane Couette flow is linearly stable for any Reynolds number; see \cite{R1973}, for example. It is shown that finite-amplitude disturbances can drive transition to turbulence in plane Couette flow at Reynolds numbers of order $1000$ in 3D \cite{OK1980}. Here, we show that the plane Couette flow may be stable for the system \eqref{eq:main} under the  bacterial concentration and fluid mixing, but it requires a sufficiently large quantity similar to the Reynolds number. 
	\end{Rem}

	\begin{Rem}
		In Theorem \ref{thm:PKS} or Theorem \ref{Thm},  
		$\epsilon$ needs an additional lower bound ${1}/{3}$ when $\alpha =0$, which allows us to gain additional integrability
		\begin{equation*}
			\||k|^{-\frac{5}{6}}\langle k\rangle^{-m}\langle\frac{1}{k}\rangle^{-\epsilon}\|_{L_k^2} < \infty
		\end{equation*}
		(see \eqref{eq:pks13} in Lemma \ref{lem:est of n'}, for example) and thereby formally control $ \| \nabla^2 c \|_{L^2}$ with the help of  $ \| n\|_{L^2}$. This insightful idea was motivated by \cite{AB2024,LLZ2025,CWY2025}. There are additional advantages in using this derivative.
	\end{Rem}

	\begin{Rem}
		The above theorems contain an anisotropic norm
		\[
		\|f\|_{Y_{m,\epsilon}}
		=  \|\langle D_x\rangle^{m}\langle \frac{1}{D_x}\rangle^{\epsilon} f \|_{L^2(\mathbb{R}^2)},
		\quad  0  <\epsilon<\tfrac{1}{2}<m.
		\]
		Expressed in Fourier variables, the weight function is defined as
		\[
		W(k) = (1+|k|^2)^{m}(1+|k|^{-2})^{\epsilon}.
		\]
		At high frequencies ($|k|\gg1$), it behaves asymptotically as $W(k)\sim |k|^{2m}$, thereby enforcing $m$-th order regularity in the $x$-variable.
		Conversely, at low frequencies ($|k|\ll1$), the asymptotic form $W(k)\sim |k|^{-2\epsilon}$ penalizes spectral concentration near $k=0$, thus suppressing large-scale modes in $x$. Since $2\epsilon<1$, the singularity is integrable and does not exclude broad classes of $L^2$ functions. In contrast to the isotropic Sobolev space $H^s(\mathbb{R}^2)$ with its symmetric weight $(1+|k|^2+|\xi|^2)^s$, the $Y_{m,\epsilon}$-norm is anisotropic: it specifically targets the $x$-direction while simultaneously controlling both high and low frequencies in $x$. In fact, the embeddings hold:
		\[
		\|f\|_{H_x^s L_y^2} \lesssim \|f\|_{Y_{m,\epsilon}}, \quad -\epsilon \leq s \leq m.
		\]
		Thus, the norm controls a full interval of regularities in the $x$-direction. Some typical classes contained in $Y_{m,\epsilon}$ include:
		\begin{itemize}
			\item Schwartz functions $\mathcal{S}(\mathbb{R}^2)$ and $C_c^\infty(\mathbb{R}^2)$,
			\item Gaussian functions and other rapidly decaying profiles,
			\item derivatives $f=\partial_x g$ with $g\in L^2(\mathbb{R}^2)$ (vanishing of $\widehat f$ at $k=0$ offsets the low-frequency weight),
			\item $L^2 (\mathbb{R}^2)$ functions with $x$-zero-average. 
		\end{itemize}
		Nonzero constants, however, are excluded (they are also not in $L^2 (\mathbb{R}^2)$). 

	\end{Rem}

	\begin{Rem}
		By comparing the two theorems, it shows that the weight exponent of $D_x$ in the initial value of $n$ is exactly ${1}/{3}$ greater than that in $\omega$. This originates from transferring a $|D_x|^{1/3}$ derivative to $\omega$ in \eqref{eq:px n omega}, which is inspired by the handling of the buoyancy term $\partial_x \theta$ in \cite{CWY2025}. In fact, it is essential to employ a $|D_x|^{1/3}$ derivative  for the enhanced dissipation generated by the Couette flow in the whole space $\mathbb{R}^2$ or $\mathbb{R}^3$.
	\end{Rem}
	
	\begin{Rem} \label{rem1}
		One difficulty lies in estimating the Fourier frequencies. When handling the nonlinear terms, one encounters convolutions in the horizontal frequency, which require careful scaling to overcome singularities. Inspired by \cite{LLZ2025,CWY2025}, one can adopt their approach for the transformations among $k$, $l$, and $k-l$. The integration domain is divided into three regions based on the relative sizes of $|k|$ and $|k-l|$: 
		\begin{itemize}
			\item Near-resonant region ($\frac{|k-l|}{2} \leq |k| \leq 2|k-l|$), where $|l|$ is ``smallest'' and $|k| \sim |k-l|$;
			\item High-low interaction region ($|k| > 2|k-l|$), where $|k-l|$ is ``smallest'' and $|k| \sim |l|$; 
			\item Low-high interaction region ($2|k| < |k-l|$), where $|k|$ is ``smallest'' and $|k-l| \sim |l|$.
		\end{itemize}
		This classification allows one to handle each case efficiently by exploiting the relative sizes of the frequencies.
	\end{Rem}
	

	\subsection{Some notations and outlines}
	
	Here are some notations used in this paper.
	
	\noindent\textbf{Notations}:
	\begin{itemize}
		\item For a given function $f(x,y)$ on $\mathbb{R}^2$, its $k$-th horizontal Fourier modes can be defined by
		\begin{equation} \label{eq:def of fk}
			f_k(y)=\mathcal{F}_{x \rightarrow k}(f)(k, y) = (2\pi)^{-\frac12}\int_{\mathbb{R}} f(x, y) \mathrm{e}^{- ikx} dx,
		\end{equation}
		In addition, denote 
		$$
		\hat{f}_k(\xi)=\hat{f}(k, \xi) = \mathcal{F}_{y \rightarrow \xi}(f_k)(y) = (2\pi)^{-1}  \int_{\mathbb{R}^2} f(x, y) \mathrm{e}^{- i(kx + \xi y)} dxdy.
		$$
		For another given function $g(x,y)$ on $\mathbb{R}^2$, note that
		\begin{equation*}
			\mathcal{F}_{x \rightarrow k}(f\cdot g)(k, y) = (2\pi)^{-\frac12}  f_k(y)*g_k(y).
		\end{equation*}
		\item Denote $C$ by  a positive constant independent of $A$, $t$ and the initial data, and it may be different from line to line. $A \lesssim B$ means there exists an absolute constant $C$, such that $A \leq C B$.
		\item Denote $\nabla_l :=(l,\, \py)$ and $\Delta_l := \py^2 - l^2 $.
		\item The space norm $\|f\|_{L^{p}}$ is defined by	
		$\|f\|_{L^{p}(\mathbb{R}^2)}=\left(\int_{\mathbb{R}^2}|f|^p dxdy\right)^{{1}/{p}}  $. For $t\geq0$, the time-space norm $\|f\|_{L^{q}L^{p}}$ is defined by	
		$\|f\|_{L^qL^p}=\|\|f\|_{L^p(\mathbb{R}^2)}\|_{L^q(0,t)}$.
		For simplicity, we write $\|f\|_{L^p(\mathbb{R}^2)}$ as $\|f\|_{L^p}$ and write $\| (f,g)\|_{L^p} = \sqrt{\|f\|_{L^p}^2 + \|g\|_{L^p}^2 }$. 
		\item Denote $(f | g)$ by the $L^2\left(\mathbb{R}^2\right)$ inner product of $f$ and $g$.
	\end{itemize}
	
	\noindent\textbf{Outlines}:

	The paper is organized as follows. In Section \ref{Sec.2}, we outline the main ideas and the proofs of the mian theorems. Section \ref{sec.3} is devoted to the PKS system, and we complete the proof of Proposition \ref{main prop3}, which implies Theorem \ref{thm:PKS}. In Section \ref{sec.4}, we consider the PKSNS system and prove Proposition \ref{main prop}, then the proof of Theorem \ref{Thm} is complete.

	\section{Preliminaries and proofs of the main theorems} \label{Sec.2}
	
	\subsection{Construction of the multipliers}
	For $k,\xi \in \mathbb{R}$ and $A >0$, denote that 
	\begin{equation*}
		\begin{aligned}
			\mathcal{M}_1(k, \xi) & := \arctan \left(A^{-\frac{1}{3}}|k|^{-\frac{1}{3}} \operatorname{sgn}(k) \xi\right)+ \frac{\pi}{2},\\
			\mathcal{M}_2(k, \xi) & := \arctan \left(\frac{\xi}{k}\right)+\frac{\pi}{2}
			.
		\end{aligned}
	\end{equation*}
	Then, these two self-adjoint Fourier multipliers satisfy
	\begin{equation} \label{eq:bound of M}
		1\leq \mathcal{M}:= \frac{\mathcal{M}_1+\mathcal{M}_2}{2\pi\Xi^{-1}}
		+1 \leq 1+\Xi,
	\end{equation}
	where $\Xi>0$, to be decided.
	The enhanced dissipation multiplier $\mathcal{M}_1$ and the inviscid damping multiplier $\mathcal{M}_2$ are  motivated by \cite{DWZ2021,WZ2023}. 
	On the Fourier side, these multipliers can be understood as some kind of `ghost weight', which are bounded weights providing additional dissipation properties. One crucial feature of the two multipliers mentioned above is that
	\begin{equation} \label{eq:crucial M}
		2 \Re\big(\left(\partial_t+y \partial_x\right) f \mid \mathcal{M}_i f\big)=\frac{d}{d t}\|\sqrt{\mathcal{M}_i} f\|_{L^2}^2+\int_{\mathbb{R}^2}  k \partial_{\xi}  \mathcal{M}_i(k, \xi)|\hat{f}|^2 d k d \xi,
	\end{equation}
	for a sufficiently smooth function $f= f(t, x ,y)$ on $(0, +\infty) \times \mathbb{R}^2$.
	\begin{Lem} \label{lem:M12}
		For a sufficiently smooth function $f= f(t, x ,y)$ on $(0, +\infty) \times \mathbb{R}^2$ and the Fourier multiplier $\mathcal{M}$ defined in \eqref{eq:bound of M}, it holds that
		\begin{equation}  \label{eq:m}
			2\pi \Xi^{-1}  \int_{\mathbb{R}^2}  k \partial_{\xi}   \mathcal{M}(k, \xi)|\hat{f}|^2 d k d \xi 
			\geq \frac{\theta_1}{ A^{\frac{1}{3}}}\|\left|D_x\right|^{\frac{1}{3}} f\|_{L^2}^2-\frac{\theta_2}{A}\|\partial_y f\|_{L^2}^2 + \|\partial_x \nabla \Delta^{-1} f\|_{L^2}^2,
		\end{equation}
		where $0< \theta_1,\theta_2\leq 1$  satisfying $\theta_1\leq 2\sqrt{\theta_2}-\theta_2.$
	\end{Lem}
	\begin{proof}
		Let $0<\theta_1, \theta_2\leq 1 $, then we have 
		\begin{equation*}
			\frac{1}{1+x} \geq \theta_1 - \theta_2 x, \qquad \forall\, x \ge 0,
		\end{equation*}
        provided that $\theta_1\leq 2\sqrt{\theta_2}-\theta_2.$
		Direct calculation shows that
		$$
		k \partial_{\xi} \mathcal{M}_1(k, \xi)=\frac{A^{-\frac{1}{3}}|k|^{\frac{2}{3}}}{1+A^{-\frac{2}{3}}|k|^{-\frac{2}{3}}|\xi|^2} \geq \frac{\theta_1}{ A^{\frac{1}{3}}  }  |k|^{\frac{2}{3}}-\frac{\theta_2}{A}|\xi|^2
		$$
		and
		$$
		k \partial_{\xi} \mathcal{M}_2(k, \xi)=\frac{k^2}{k^2+\xi^2},
		$$
		which imply \eqref{eq:m}.
	\end{proof}
	
	\subsection{Elliptic estimates of $c$}
	\begin{Lem} \label{lem:Elliptic estimate}
		Let $\mathcal{N} = \mathcal{N}(D_x, D_y, t)$ be a Fourier multiplier. Then, for $c$ and $n$ satisfying $\eqref{eq:main}_2$ and $t \geq 0$, the following estimates hold.

 \textbf{Case of $\alpha = 0$:}
\begin{equation*}
\| \nabla^2 \mathcal{N} c \|_{L^2} \leq \| \mathcal{N} n \|_{L^2}, \quad
\| \nabla c \|_{L^4} \leq \frac{2}{\pi}(3 + 2\sqrt{2}) \big( \| n \|_{L^2} + M \big);
\end{equation*}

 \textbf{Case of $\alpha > 0$:}
\begin{equation*}
\| \nabla \mathcal{N} c \|_{L^2} \leq \frac{1}{\sqrt{2\alpha}} \| \mathcal{N} n \|_{L^2}, \quad 
\| \nabla c \|_{L^4} \leq \left( \frac{1}{2\pi\alpha} \right)^{\frac14} \| n \|_{L^2}.
\end{equation*}
	\end{Lem}
	\begin{proof}
		\underline{\bf Case of $\alpha =0$.} Recalling $\eqref{eq:main}_2$, one gets 
		\begin{equation*}\begin{aligned}
				-\Delta \mathcal{N}c   =  \mathcal{N}n .
		\end{aligned}\end{equation*}
		Then, by
		\begin{equation*}\begin{aligned}
				\| \nabla^2 \mathcal{N}c \|_{L^2}^2 = \| \Delta\mathcal{N}c \|_{L^2}^2,
		\end{aligned}\end{equation*}
		we have
		\begin{equation*}\begin{aligned}
				\| \nabla^2 \mathcal{N}c \|_{L^2}^2 =	\|  \mathcal{N}n \|_{L^2}^2.
		\end{aligned}\end{equation*}
		From \cite{T1976} (see the formula (1)), we get
		\begin{equation*}
			\|\nabla c \|_{L^4} \leq \frac1\pi \|\nabla^2 c \|_{L^\frac43}.
		\end{equation*}
		Then, using Lemma \ref{lem:R} and the interpolation inequality, we obtain
		\begin{equation*}\begin{aligned}  
				\|\nabla c \|_{L^4}   &\leq \frac1\pi \| \mathbf{R}_i \mathbf{R}_j n \|_{L^\frac43} \leq \frac4\pi (3 + 2\sqrt{2})  \| n \|_{L^\frac43} \\
				&\leq \frac4\pi (3 + 2\sqrt{2}) \| n \|_{L^1}^\frac12  \| n \|_{L^2}^\frac12 \leq  \frac2\pi (3 + 2\sqrt{2})  (\| n_{\mathrm{in}} \|_{L^1} +  \| n \|_{L^2}  ).
		\end{aligned}\end{equation*}

		\underline{\bf Case of $\alpha > 0$.} It follows from $\eqref{eq:main}_2$ that
		\begin{equation*}\begin{aligned}
				\alpha\mathcal{N}c - \Delta\mathcal{N}c  =  \mathcal{N}n.
		\end{aligned}\end{equation*}
		Therefore, $L^2$ energy estimate implies
		\begin{equation*}\begin{aligned}
				\| \Delta\mathcal{N}c \|_{L^2}^2 +  	2 \alpha \| \nabla \mathcal{N}c \|_{L^2}^2 + 	\|  \alpha \mathcal{N}c \|_{L^2}^2 = 	\|  \mathcal{N}n \|_{L^2}^2
		\end{aligned}\end{equation*}
		and then (see the formula (1.11) in \cite{L1980})
		$$  \|\nabla c\|_{  L^4} \leq \pi^{-\frac14} \|\nabla^2 c\|_{  L^2}^\frac12 \|  \nabla c\|_{  L^2 }^\frac12   \leq (\frac1{2\pi\alpha})^\frac14 \|n\|_{  L^2}.  $$
		The proof is complete.
	\end{proof}
	
	\subsection{Space-time estimate}
	Let $u$ be determined by $\eqref{eq:main}_4$, and $f$ satisfy
	\begin{equation}\begin{aligned} \label{eq:f}
			\left\{\begin{array}{l}
				\partial_t f+y \partial_x f-\frac{1}{A} \Delta f=-\frac{1}{A}u\cdot\nabla f+g,\quad (t,x,y) \in [0, T]\times \mathbb{R}^2, \\
				\left. f\right|_{t=0} = f_{\rm in}, \quad (x,y) \in  \mathbb{R}^2,
			\end{array}\right.
	\end{aligned}\end{equation}
	where $g$ represents an additional term.
	For \eqref{eq:f}, we prove the following space-time estimate, which plays an important role in the energy estimates stated in Section \ref{sec.3}.
	\begin{Prop} \label{lem:est of f}
		Let $u$ be determined by $\eqref{eq:main}_4$, $f$ satisfy \eqref{eq:f}, and  $\theta_1$, $\theta_2$, $\Xi$ and $a$ satisfy the positivity conditions \eqref{eq:positivity conditions}.
		Then, for $0<\epsilon<1/2<m$ and $0\leq t\leq T$, it holds that 
		\begin{equation*}\begin{aligned}
                \|  \langle D_x\rangle^m  \langle \frac{1}{D_x}\rangle^\epsilon   f  \|_{X}^2 
                &\leq    
                \|  \langle D_x\rangle^m  \langle \frac{1}{D_x}\rangle^\epsilon   f_{\mathrm{in}} \|_{L^2}^2 +  2 \int_{0}^{t} \Re\left(g \left\lvert\, \mathcal{M} e^{2 a A^{-\frac{1}{3}}\left|D_x\right|^{\frac{2}{3} }t }\langle D_x\rangle^{2 m}\langle\frac{1}{D_x}\rangle^{2 \epsilon} f\right.\right) dt  \\
				&\quad +2 C_{st}\cdot \frac{  2^{  m} }{A^\frac12}   \lt(\frac32\rt)^\epsilon\max\lt\{ \sqrt{B(\frac12 - \epsilon, m -\frac12)} , \sqrt{B(\epsilon, m -\frac12)}\rt\} \\
                &\qquad \times \|  \langle D_x\rangle^m  \langle \frac{1}{D_x}\rangle^\epsilon   \omega \|_{X}   \|  \langle D_x\rangle^m  \langle \frac{1}{D_x}\rangle^\epsilon   f \|_{X}^2 ,
		\end{aligned}\end{equation*}
		where $$C_{st}= (2\pi)^{-\frac12} (1 + \Xi)\lt[  2^\frac13 \sqrt{\frac43} \lt(\frac{\theta_1\Xi}{2\pi} - 2a \rt)^{-\frac34}  2^{-\frac14}               +   (3 + \sqrt{2} + \sqrt{\frac43})  \left(\frac{\Xi}{2\pi}\right)^{-\frac12}    \left(2-\frac{\theta_2\Xi}{2\pi}\right)^{-\frac12}   \rt].$$
	\end{Prop}	
	\begin{proof}
		By taking the inner product of $\eqref{eq:f}_1$ with $2 \mathcal{M} e^{2 a A^{-{1}/{3}}\left|D_x\right|^{{2}/{3}}  t } \langle D_x \rangle^{2 m} \langle{D_x}^{-1} \rangle^{2 \epsilon} f$ and using \eqref{eq:crucial M}, one has
		$$
		\begin{aligned}
			&\quad \frac{d}{d t}\|\sqrt{\mathcal{M}} e^{a A^{-\frac{1}{3}}\left|D_x\right|^{\frac{2}{3} } t}\langle D_x\rangle^m\langle\frac{1}{D_x}\rangle^\epsilon f\|_{L^2}^2 +\frac2A \|\sqrt{\mathcal{M}} e^{a A^{-\frac{1}{3}} \left|D_x\right|^{\frac{2}{3} } t}\langle D_x\rangle^m\langle\frac{1}{D_x}\rangle^\epsilon \nabla f\|_{L^2}^2    \\
			&\quad \quad -    \frac{2a} { A^{\frac{1}{3}} }    \|\sqrt{\mathcal{M}} e^{a A^{- \frac{1}{3}}\left|D_x\right|^{\frac{2}{3} } t}\langle D_x\rangle^m\langle\frac{1}{D_x}\rangle^\epsilon\left|D_x\right|^{\frac{1}{3}} f\|_{L^2}^2 \\
			&\quad\quad+\int_{\mathbb{R}^2} k \partial_{\xi} \mathcal{M}(k, \xi) e^{2 a A^{-\frac{1}{3}}|k|^{\frac{2}{3} } t }\langle k\rangle^{2 m}\langle\frac{1}{k}\rangle^{2 \epsilon}|\hat{f}(k, \xi)|^2 d k d \xi \\
			&  =     - \frac2A  \Re\left(u \cdot \nabla f \left\lvert\, \mathcal{M} e^{2 a A^{-\frac{1}{3}}\left|D_x\right|^{\frac{2}{3} } t}\langle D_x\rangle^{2 m}\langle\frac{1}{D_x}\rangle^{2 \epsilon} f\right.\right) \\
			&\quad   + 2  \Re\left( g \left\lvert\, \mathcal{M} e^{2 a A^{-\frac{1}{3}}\left|D_x\right|^{\frac{2}{3} } t}\langle D_x\rangle^{2 m}\langle\frac{1}{D_x}\rangle^{2 \epsilon} f\right.\right).
		\end{aligned}
		$$
		Then, by Lemma \ref{lem:M12} and \eqref{eq:bound of M}, there holds that
		\begin{align*}
			&\quad \frac{d}{d t}\|\sqrt{\mathcal{M}} e^{a A^{-\frac{1}{3}}\left|D_x\right|^{\frac{2}{3} }t}\langle D_x\rangle^m\langle\frac{1}{D_x}\rangle^\epsilon f\|_{L^2}^2   + \frac{\Xi}{2\pi} \|e^{a A^{-\frac{1}{3}}\left|D_x\right|^{\frac{2}{3} }t }\langle D_x\rangle^m\langle\frac{1}{D_x}\rangle^\epsilon \partial_x \nabla  \Delta^{-1} f \|_{L^2}^2    \\
			& \quad\quad + \lt(\frac{\theta_1 \Xi}{2\pi}  - 2 a\rt)\frac{1}{ A^{\frac{1}{3}}}\|e^{a A^{-\frac{1}{3}}\left|D_x\right|^{\frac{2}{3} } t }\langle D_x\rangle^m\langle\frac{1}{D_x}\rangle^\epsilon\left|D_x\right|^{\frac{1}{3}} f\|_{L^2}^2 \\
			&\quad\quad + \lt(2 - \frac{\theta_2\Xi}{2\pi}\rt)\frac{1}{A} \|e^{a A^{-\frac{1}{3}}\left|D_x\right|^{\frac{2}{3} }t}\langle D_x\rangle^m\langle\frac{1}{D_x}\rangle^\epsilon \py f\|_{L^2}^2 + \frac{2}{A} \|e^{a A^{-\frac{1}{3}}\left|D_x\right|^{\frac{2}{3} }t}\langle D_x\rangle^m\langle\frac{1}{D_x}\rangle^\epsilon \px f\|_{L^2}^2 \\
			&   \leq \frac2A \left|\Re\left(u \cdot \nabla f \left\lvert\, \mathcal{M} e^{2 a A^{-\frac{1}{3}}\left|D_x\right|^{\frac{2}{3} }t }\langle D_x\rangle^{2 m}\langle\frac{1}{D_x}\rangle^{2 \epsilon} f\right.\right)\right| \\
			&\quad  + 2 \Re\left(g \left\lvert\, \mathcal{M} e^{2 a A^{-\frac{1}{3}}\left|D_x\right|^{\frac{2}{3} }t }\langle D_x\rangle^{2 m}\langle\frac{1}{D_x}\rangle^{2 \epsilon} f\right.\right) .
		\end{align*}

		Due to the different properties of $u_1 \px$ and $u_2 \py$, the estimates for $u\cdot\nabla f$ are divided into two parts.
		
		\underline{\bf Estimate of $u_1 \px f$.}
		The Fourier transform implies
		\begin{equation}\begin{aligned} \label{eq:est0 of u1 px omega}
				&\quad\left|\Re\left(u_1 \px  f \left\lvert\, \mathcal{M} e^{2 a A^{-\frac{1}{3}}\left|D_x\right|^{\frac{2}{3} }t }\langle D_x\rangle^{2 m}\langle\frac{1}{D_x}\rangle^{2 \epsilon} f\right.\right)\right| \\
				&\leq (2\pi)^{-\frac12}  \int_{\mathbb{R}^2} e^{2 a A^{-\frac{1}{3}}|k|^{\frac{2}{3} } t }\langle k\rangle^{2 m}\langle\frac{1}{k}\rangle^{2 \epsilon}\left|\int_{\mathbb{R}} \mathcal{M}\left(k, D_y\right) f_k(y) \partial_y \Delta_l^{-1} \omega_l(y)  (k-l) f_{k-l}(y) d y\right| d k d l.
		\end{aligned}\end{equation}
		In the following, we divide the integration domain in the above equality into three cases, as described in Remark \ref{rem1}. 
		
		$\bullet$~Case 1:  $\frac{|k-l|}{2} \leq |k| \leq 2|k-l|$.
		Note that
		\begin{equation}\begin{aligned} \label{eq:trick7}
				\langle k\rangle^m\langle\frac{1}{k}\rangle^\epsilon \leq   2^{m + \epsilon} \langle k-l\rangle^m\langle\frac{1}{k-l}\rangle^\epsilon, \quad |k-l|^{\frac{1}{3}}  \leq 2^{\frac{1}{3}} |k|^{\frac{1}{3}}
		\end{aligned}\end{equation} 
		and
		$$
		\begin{aligned}
			&\quad \|e^{a A^{-\frac{1}{3}}|l|^{\frac{2}{3} } t }\| \partial_y \Delta^{-1}_l \omega_l \|_{L_y^{\infty}}\|_{L_l^1} \leq \|e^{a A^{-\frac{1}{3}}|l|^{\frac{2}{3}} t}|l|^{-\frac{1}{2}}\| \partial_y \nabla_l \Delta^{-1}_l \omega_l \|_{L_y^2}\|_{L_l^1} \\
			& \leq \|e^{a A^{-\frac{1}{3}}|l|^{\frac{2}{3} }t}\langle l\rangle^m\langle\frac{1}{l}\rangle^\epsilon\|    \partial_y \nabla_l \Delta^{-1}_l \omega_l \|_{L_y^2}\|_{L_l^2}\||l|^{-\frac{1}{2}}\langle l\rangle^{-m}\langle\frac{1}{l}\rangle^{-\epsilon}\|_{L_l^2} \\
			&\leq  \sqrt{ B(\epsilon, m)}\|e^{a A^{-\frac{1}{3}}\left|D_x\right|^{\frac{2}{3} }t}\langle D_x\rangle^m\langle\frac{1}{D_x}\rangle^\epsilon \omega\|_{L^2},
		\end{aligned}
		$$
		which followed from \eqref{eq:GN2}, and where $B(\epsilon, m)$ is the usual Beta function. Then we have
		\begin{equation}\begin{aligned} \label{eq:est1 of u1 px omega}
				&\quad \int_{\frac{|k-l|}{2} \leq|k| \leq 2|k-l|} e^{2 a A^{-\frac{1}{3}}|k|^{\frac{2}{3} }t }\langle k\rangle^{2 m}\langle\frac{1}{k}\rangle^{2 \epsilon}   \\
				&\qquad\qquad\qquad \times  \left| \int_{\mathbb{R}} \mathcal{M}\left(k, D_y\right) f_k (y) \partial_y \Delta^{-1}_l \omega_l(y)  (k-l) f_{k-l}(y) d y \right| d k d l \\
				& \leq 2^{m + \epsilon +\frac13}  (1 + \Xi) \|e^{a A^{-\frac{1}{3}}|k|^{\frac{2}{3} } t }\langle k\rangle^m\langle\frac{1}{k}\rangle^\epsilon|k|^{\frac{1}{3}}\| f_k\|_{L_y^2}\|_{L_k^2}            \|e^{a A^{-\frac{1}{3}}|l|^{\frac{2}{3} } t }\| \partial_y \Delta^{-1}_l \omega_l \|_{L_y^{\infty}}\|_{L_l^1} \\
				&\quad  \times \|e^{a A^{-\frac{1}{3}}|k-l|^{\frac{2}{3} }t }\langle k-l\rangle^m\langle\frac{1}{k-l}\rangle^\epsilon|k-l|^{\frac{2}{3}}\| f_{k-l}\|_{L_y^2}\|_{L_{k-l}^2} \\
				& \leq   2^{m + \epsilon +\frac13} (1 + \Xi)  \sqrt{ B(\epsilon, m)}   \|e^{a A^{-\frac{1}{3}}\left|D_x\right|^{\frac{2}{3} } t }\langle D_x\rangle^m\langle\frac{1}{D_x}\rangle^\epsilon\left|D_x\right|^{\frac{1}{3}} f\|_{L^2} ^{\frac{3}{2}}          \\
				&\quad \times  \|e^{a A^{-\frac{1}{3}}\left|D_x\right|^{\frac{2}{3} } t }\langle D_x\rangle^m\langle\frac{1}{D_x}\rangle^\epsilon \partial_x f\|_{L^2} ^{\frac{1}{2}}\|e^{a A^{-\frac{1}{3}}\left|D_x\right|^{\frac{2}{3} } t }\langle D_x\rangle^m\langle\frac{1}{D_x}\rangle^\epsilon \omega\|_{L^2}.
		\end{aligned}\end{equation}

		$\bullet$~Case 2:   $2|k-l|<|k|$.
		There hold
		\begin{equation}\begin{aligned} \label{eq:trick8}
				\langle k\rangle^m\langle\frac{1}{k}\rangle^\epsilon \leq 2^{m  } \lt(\frac32\rt)^\epsilon \langle l\rangle^m\langle\frac{1}{l}\rangle^\epsilon\quad \text{and} \quad |k-l| \leq |l|.
		\end{aligned}\end{equation}
		Hence, we get
		\begin{align} \label{eq:est2 of u1 px omega}
				&\no \quad\int_{2|k-l|<|k|} e^{2 a A^{-\frac{1}{3}}|k|^{\frac{2}{3} }t }\langle k\rangle^{2 m}\langle\frac{1}{k}\rangle^{2 \epsilon}    \\
				&\no \qquad \qquad\times \left|  \int_{\mathbb{R}} \mathcal{M}\left(k, D_y\right) f_k(y)  \partial_y \Delta^{-1}_l \omega_l(y)  (k-l) f_{k-l}(y) d y\right|   d k d l \\
				& \no\leq2^{m  }  \lt(\frac32\rt)^\epsilon (1 + \Xi)\|e^{a A^{-\frac{1}{3}}|k|^{\frac{2}{3} } t }\langle k\rangle^m\langle\frac{1}{k}\rangle^\epsilon\| f_k\|_{L_y^2}\|_{L_k^2}                \|e^{a A^{-\frac{1}{3}}|k-l|^{\frac{2}{3} } t }\| f_{k-l}\|_{L_y^{\infty}}\|_{L_{k-l}^1} \\
				&\no \quad \times  \|e^{a A^{-\frac{1}{3}}|l|^{\frac{2}{3} } t }\langle l\rangle^m\langle\frac{1}{l}\rangle^\epsilon|l|\| \partial_y \Delta^{-1}_l \omega_l \|_{L_y^2}\|_{L_l^2}   \\
				&\no \leq  2^{m   } \lt(\frac32\rt)^\epsilon (1 + \Xi)  \sqrt{ B(\epsilon, m)}  \|e^{a A^{-\frac{1}{3}}\left|D_x\right|^{\frac{2}{3} } t }\langle D_x\rangle^m\langle\frac{1}{D_x}\rangle^\epsilon f\|_{L^2}     \|e^{a A^{-\frac{1}{3}}\left|D_x\right|^{\frac{2}{3} } t }\langle D_x\rangle^m\langle\frac{1}{D_x}\rangle^\epsilon \nabla f\|_{L^2} \\
				&\quad \times  \|e^{a A^{-\frac{1}{3}}\left|D_x\right|^{\frac{2}{3} } t }\langle D_x\rangle^m\langle\frac{1}{D_x}\rangle^\epsilon \px \py \Delta^{-1} \omega \|_{L^2},
		\end{align}
		where we used 
		$$
		\begin{aligned}
			& \quad \|e^{a A^{-\frac{1}{3}}|k-l|^{\frac{2}{3}} t}\| f_{k-l}\|_{L_y^{\infty}}\|_{L_{k-l}^1} \leq \|e^{a A^{- \frac{1}{3}}|k-l|^{\frac{2}{3}}}|k-l|^{-\frac{1}{2}}\| \nabla_{k-l} f_{k-l}\|_{L_y^2}\|_{L_{k-l}^1} \\
			& \leq \|e^{a A^{-\frac{1}{3}}|k-l|^{\frac{3}{3}} t}\langle k-l\rangle^m\langle\frac{1}{k-l}\rangle^\epsilon\| \nabla_{k-l} f_{k-l}\|_{L_y^2}\|_{L_{k-l}^2}\||k-l|^{-\frac{1}{2}}\langle k-l\rangle^{-m}\langle\frac{1}{k-l}\rangle^{-\epsilon}\|_{L_{k-l}^2} \\
			& \leq \sqrt{ B(\epsilon, m)} \|e^{a A^{-\frac{1}{3}}\left|D_x\right|^{\frac{2}{3}} t}\langle D_x\rangle^m\langle\frac{1}{D_x}\rangle^\epsilon \nabla f\|_{L^2} .
		\end{aligned}
		$$
		
		$\bullet$~Case 3:   $2|k|<|k-l|$.
		From 
		\begin{equation}\begin{aligned} \label{eq:trick9}
				\langle k \rangle^{2m} \leq  \langle l \rangle^m \langle k-l \rangle^m,\quad |k-l|\leq 2|l|,
		\end{aligned}\end{equation}
        $$
		\begin{aligned}
			&\quad \|e^{a A^{-\frac{1}{3}}|k-l|^{\frac{2}{3} } t }\langle k-l\rangle^m\langle\frac{1}{k-l}\rangle^\epsilon|k-l|^{\frac{1}{2}}\| f_{k-l}\|_{L_y^{\infty}}\|_{L_{k-l}^2}  \\
			& \leq   \|e^{a A^{-\frac{1}{3}}|k-l|^{\frac{2}{3} } t }\langle k-l\rangle^m\langle\frac{1}{k-l}\rangle^\epsilon  \| \nabla_{k-l} f_{k-l}\|_{L_y^{2}}\|_{L_{k-l}^2}   \\
			& \leq \|e^{a A^{-\frac{1}{3}}\left|D_x\right|^{\frac{2}{3}} t}\langle D_x\rangle^m\langle\frac{1}{D_x}\rangle^\epsilon \nabla f\|_{L^2},
		\end{aligned}
		$$
		and
		\begin{equation} \label{eq:trick4}
			1 \leq \langle \frac{1}{k-l} \rangle^\frac12 |k-l|^{\frac12} \leq \lt(\frac32\rt)^\epsilon \langle\frac{1}{l}\rangle^\epsilon\langle\frac{1}{k-l}\rangle^\epsilon\left(\chi_{\{\frac14\leq\epsilon<\frac12 \}}+\langle\frac{1}{k}\rangle^{\frac{1}{2}-2 \epsilon}\chi_{\{0<\epsilon<\frac14\}}\right)|k-l|^{\frac12},
		\end{equation}
		we obtain
		\begin{align*} 
				& \quad\int_{2|k|<|k-l|} e^{2 a A^{-\frac{1}{3}}|k|^{\frac{2}{3} } t }\langle k\rangle^{2 m}\langle\frac{1}{k}\rangle^{2 \epsilon}\left|  \int_{\mathbb{R}} \mathcal{M}\left(k, D_y\right) f_k(y)  \partial_y \Delta^{-1}_l \omega_l(y) (k-l) f_{k-l}(y)  d y \right| d k d l \\
				& \leq 2 \lt(\frac32\rt)^\epsilon (1 + \Xi) \|e^{a A^{-\frac{1}{3}}|k|^{\frac{2}{3} } t }\langle k\rangle^m\langle\frac{1}{k}\rangle^\epsilon\| f_k\|_{L_y^2}\|_{L_k^2}      \|e^{a A^{-\frac{1}{3}}|l|^{\frac{2}{3} } t }\langle l\rangle^m\langle\frac{1}{l}\rangle^\epsilon|l|\| \partial_y \Delta^{-1}_l \omega_l\|_{L_y^2}\|_{L_l^2} \\
				& \quad \times  \|e^{a A^{-\frac{1}{3}}|k-l|^{\frac{2}{3} } t }\langle k-l\rangle^m\langle\frac{1}{k-l}\rangle^\epsilon|k-l|^{\frac{1}{2}}\| f_{k-l}\|_{L_y^{\infty}}\|_{L_{k-l}^2} \\
                &\quad \times \|\left(\langle\frac{1}{k}\rangle^\epsilon \chi_{\{\frac14\leq \epsilon<\frac12 \}} +\langle\frac{1}{k}\rangle^{\frac{1}{2}-\epsilon} \chi_{\{0<\epsilon<\frac14\}}\right)\langle k\rangle^{-m}\|_{L_k^2}\\
				& \leq 2 \lt(\frac32\rt)^\epsilon (1 + \Xi) \max\lt\{ \sqrt{B(\frac12 - \epsilon, m -\frac12)}, \sqrt{B(\epsilon, m -\frac12)}\rt\} \|e^{a A^{- \frac{1}{3}}\left|D_x\right|^{\frac{2}{3} } t }\langle D_x\rangle^m\langle\frac{1}{D_x}\rangle^\epsilon f\|_{L^2}      \\
				&\quad \times  \|e^{a A^{-\frac{1}{3}}\left|D_x\right|^{\frac{2}{3} } t }\langle D_x\rangle^m\langle\frac{1}{D_x}\rangle^\epsilon \px \py \Delta^{-1} \omega \|_{L^2} \|e^{a A^{-\frac{1}{3}}\left|D_x\right|^{\frac{2}{3} } t }\langle D_x\rangle^m\langle\frac{1}{D_x}\rangle^\epsilon \nabla f\|_{L^2} \alabel{eq:est3 of u1 px omega}
		\end{align*}
		provided that $0<\epsilon<1/2<m$.

		Combining \eqref{eq:est1 of u1 px omega}, \eqref{eq:est2 of u1 px omega} and \eqref{eq:est3 of u1 px omega} with \eqref{eq:est0 of u1 px omega}, we infer that
		\begin{equation}\begin{aligned} \label{eq:est of u1 px omega}
				& \left|\Re\left(u_1 \px  f \left\lvert\, \mathcal{M} e^{2 a A^{-\frac{1}{3}}\left|D_x\right|^{\frac{2}{3} }t }\langle D_x\rangle^{2 m}\langle\frac{1}{D_x}\rangle^{2 \epsilon} f\right.\right)\right|  \\
				\leq & (2\pi)^{-\frac12} 2^{ m} \lt(\frac32\rt)^\epsilon (1 + \Xi) \max\lt\{ \sqrt{B(\frac12 - \epsilon, m -\frac12)}, \sqrt{B(\epsilon, m -\frac12)}\rt\} \\
                &\times \lt[    2^{\frac13} \lt(\frac43\rt)^\frac12\|e^{a A^{-\frac{1}{3}}\left|D_x\right|^{\frac{2}{3} } t }\langle D_x\rangle^m\langle\frac{1}{D_x}\rangle^\epsilon \omega\|_{L^2} \rt.   \\
				& \lt. \qquad \times       \|e^{a A^{-\frac{1}{3}}\left|D_x\right|^{\frac{2}{3} } t }\langle D_x\rangle^m\langle\frac{1}{D_x}\rangle^\epsilon\left|D_x\right|^{\frac{1}{3}} f\|_{L^2} ^{\frac{3}{2}}     \|e^{a A^{-\frac{1}{3}}\left|D_x\right|^{\frac{2}{3} } t }\langle D_x\rangle^m\langle\frac{1}{D_x}\rangle^\epsilon \partial_x f\|_{L^2} ^{\frac{1}{2}}  \rt.  \\
				& \lt.\quad +  (1 + \sqrt{2}) \|e^{a A^{-\frac{1}{3}}\left|D_x\right|^{\frac{2}{3} } t }\langle D_x\rangle^m\langle\frac{1}{D_x}\rangle^\epsilon f\|_{L^2}  \|e^{a A^{-\frac{1}{3}}\left|D_x\right|^{\frac{2}{3} } t }\langle D_x\rangle^m\langle\frac{1}{D_x}\rangle^\epsilon \px \py \Delta^{-1} \omega \|_{L^2} \rt. \\
				& \lt. \qquad  \times         \|e^{a A^{-\frac{1}{3}}\left|D_x\right|^{\frac{2}{3} } t }\langle D_x\rangle^m\langle\frac{1}{D_x}\rangle^\epsilon \nabla f\|_{L^2} \rt] ,
		\end{aligned}\end{equation}
		where we used the monotonicity property of the Beta function $ B(\epsilon, m )< B(\epsilon, m -1/2)$.

		\underline{\bf Estimate of $u_2 \py f$.}
		Notice that
		\begin{equation}\begin{aligned} \label{eq:est0 of u2 py omega}
				&\quad\left|\Re\left(u_2 \py  f \left\lvert\, \mathcal{M} e^{2 a A^{-\frac{1}{3}}\left|D_x\right|^{\frac{2}{3} }t }\langle D_x\rangle^{2 m}\langle\frac{1}{D_x}\rangle^{2 \epsilon} f\right.\right)\right| \\
				&\leq   (2\pi )^{-\frac12} \int_{\mathbb{R}^2} e^{2 a A^{-\frac{1}{3}}|k|^{\frac{2}{3} } t }\langle k\rangle^{2 m}\langle\frac{1}{k}\rangle^{2 \epsilon}\left| \int_{\mathbb{R}} \mathcal{M}\left(k, D_y\right) f_k(y) l \Delta_l^{-1} \omega_l(y)  \py f_{k-l}(y) d y  \right|d k d l.
		\end{aligned}\end{equation}
		$\bullet$~Case 1:   $\frac{|k-l|}{2} \leq |k| \leq 2|k-l|$. By \eqref{eq:GN2} and \eqref{eq:trick7}, one has
		\begin{align*}  
				&\quad\int_{\frac{|k-l|}{2} \leq|k| \leq 2|k-l|} e^{2 a A^{-\frac{1}{3}}|k|^{\frac{2}{3} } t }\langle k\rangle^{2 m}\langle\frac{1}{k}\rangle^{2 \epsilon}\left|  \int_{\mathbb{R}} \mathcal{M}\left(k, D_y\right) f_k(y)  l \Delta_l^{-1} \omega_l(y)  \py f_{k-l}(y) d y \right| d k d l \\
				&\leq 2^{m+\epsilon}(1 + \Xi)\|e^{a A^{-\frac{1}{3}}|k|^{\frac{2}{3} }t}\langle k\rangle^m\langle\frac{1}{k}\rangle^\epsilon\| f_k\|_{L_y^2}\|_{L_k^2}      \|e^{a A^{-\frac{1}{3}}|l|^{\frac{2}{3} } t }|l|\| \Delta_l^{-1} \omega_l \|_{L_y^{\infty}}\|_{L_l^1} \\
				&\qquad \times\|e^{a A^{-\frac{1}{3}}|k-l|^{\frac{2}{3} } t }\langle k-l\rangle^m\langle\frac{1}{k-l}\rangle^\epsilon\| \partial_y f_{k-l}\|_{L_y^2}\|_{L_{k-l}^2} \\
				&\leq 2^{m+\epsilon}(1 + \Xi)\sqrt{B(\epsilon,m)}  \|e^{a A^{-\frac{1}{3}}\left|D_x\right|^{\frac{2}{3} } t }\langle D_x\rangle^m\langle\frac{1}{D_x}\rangle^\epsilon f\|_{L^2}\|e^{a A^{-\frac{1}{3}}\left|D_x\right|^{\frac{2}{3} } t }\langle D_x\rangle^m\langle\frac{1}{D_x}\rangle^\epsilon \partial_x \nabla \Delta^{-1} \omega\|_{L^2} \\
				&\qquad \times  \|e^{a A^{-\frac{1}{3}}\left|D_x\right|^{\frac{2}{3} } t }\langle D_x\rangle^m\langle\frac{1}{D_x}\rangle^\epsilon \partial_y f\|_{L^2}.\alabel{eq:est1 of u2 py omega}
		\end{align*}

		$\bullet$~Case 2:  $2|k-l|<|k|$.
		Using \eqref{eq:GN2}, \eqref{eq:trick8} and $1 \leq|l|^{{1}/{2}} |k-l|^{-{1}/{2}}$, we obtain
		\begin{equation}\begin{aligned}  \label{eq:est2 of u2 py omega}
				&\quad\int_{2|k-l|<|k|} e^{2 a A^{-\frac{1}{3}}|k|^{\frac{2}{3} } t }\langle k\rangle^{2 m}\langle\frac{1}{k}\rangle^{2 \epsilon}\left|  \int_{\mathbb{R}} \mathcal{M}\left(k, D_y\right) f_k(y)  l \Delta_l^{-1} \omega_l(y)  \py f_{k-l}(y) d y \right| d k d l   \\
				& \leq 2^{m } \lt(\frac32\rt)^\epsilon (1 + \Xi)\|e^{a A^{-\frac{1}{3}}|k|^{\frac{2}{3} } t }\langle k\rangle^m\langle\frac{1}{k}\rangle^\epsilon\| f_k\|_{L_y^2}\|_{L_k^2}             \|e^{a A^{-\frac{1}{3}}|l|^{\frac{2}{3} } t }\langle l\rangle^m\langle\frac{1}{l}\rangle^\epsilon|l|^{\frac{3}{2}}\|  \Delta_l^{-1} \omega_l  \|_{L_y^{\infty}}\|_{L_l^2} \\
				& \quad \times\|e^{a A^{-\frac{1}{3}}|k-l|^{\frac{2}{3} } t }|k-l|^{-\frac{1}{2}}\| \partial_y f_{k-l}\|_{L_y^2}\|_{L_{k-l}^1} \\
				&\leq 2^{m }  \lt(\frac32\rt)^\epsilon (1 + \Xi)\sqrt{B(\epsilon,m)}        \|e^{a A^{-\frac{1}{3}}\left|D_x\right|^{\frac{2}{3} } t }\langle D_x\rangle^m\langle\frac{1}{D_x}\rangle^\epsilon \partial_x \nabla \Delta^{-1} \omega\|_{L^2} \\
				& \quad \times  \|e^{a A^{-\frac{1}{3}}\left|D_x\right|^{\frac{2}{3} } t }\langle D_x\rangle^m\langle\frac{1}{D_x}\rangle^\epsilon f\|_{L^2} \|e^{a A^{-\frac{1}{3}}\left|D_x\right|^{\frac{2}{3} } t }\langle D_x\rangle^m\langle\frac{1}{D_x}\rangle^\epsilon \partial_y f\|_{L^2} .
		\end{aligned}\end{equation}

		$\bullet$~Case 3:  $2|k|<|k-l|$. It follows from \eqref{eq:GN2}, \eqref{eq:trick9} and \eqref{eq:trick4} that
		\begin{equation}\begin{aligned}  \label{eq:est3 of u2 py omega}
				&\quad  \int_{2|k|<|k-l|} e^{2 a A^{-\frac{1}{3}}|k|^{\frac{2}{3} } t }\langle k\rangle^{2 m}\langle\frac{1}{k}\rangle^{2 \epsilon}    \left|   \int_{\mathbb{R}} \mathcal{M}\left(k, D_y\right) f_k(y)  l \Delta_l^{-1} \omega_l(y)  \py f_{k-l}(y) d y \right|  d k d l    \\
				& \leq 2^{\frac12}  \lt(\frac32\rt)^\epsilon(1 + \Xi)\|e^{a A^{-\frac{1}{3}}|k|^{\frac{2}{3} } t }\langle k\rangle^m\langle\frac{1}{k}\rangle^\epsilon\| f_k\|_{L_y^2}\|_{L_k^2}  \|e^{a A^{-\frac{1}{3}}|l|^{\frac{2}{3} } t }\langle l\rangle^m\langle\frac{1}{l}\rangle^\epsilon|l|^{\frac{3}{2}}\| \Delta_l^{-1} \omega_l \|_{L_y^{\infty}}\|_{L_l^2} \\
				& \quad \times   \|e^{a A^{-\frac{1}{3}}|k-l|^{\frac{2}{3} } t }\langle k-l\rangle^m\langle\frac{1}{k-l}\rangle^\epsilon\| \partial_y f_{k-l}\|_{L_y^2}\|_{L_{k-l}^2} \\
                &\quad \times \|\left( \langle\frac{1}{k}\rangle^\epsilon \chi_{\{\frac14\leq \epsilon<\frac12 \}} +\langle\frac{1}{k}\rangle^{\frac{1}{2}-\epsilon} \chi_{\{0<\epsilon<\frac14\}}\right)\langle k\rangle^{-m}\|_{L_k^2}\\
				& \leq 2^{\frac12}  \lt(\frac32\rt)^\epsilon (1 + \Xi) \max\lt\{ \sqrt{B(\frac12 - \epsilon, m -\frac12)} ,\sqrt{ B(\epsilon, m -\frac12)}\rt\} \|e^{a A^{-\frac{1}{3}}\left|D_x\right|^{\frac{2}{3} } t }\langle D_x\rangle^m\langle\frac{1}{D_x}\rangle^\epsilon f\|_{L^2}   \\
				& \quad \times  \|e^{a A^{-\frac{1}{3}}\left|D_x\right|^{\frac{2}{3} } t }\langle D_x\rangle^m\langle\frac{1}{D_x}\rangle^\epsilon \partial_x \nabla \Delta^{-1} \omega \|_{L^2} \|e^{a A^{-\frac{1}{3}}\left|D_x\right|^{\frac{2}{3} } t }\langle D_x\rangle^m\langle\frac{1}{D_x}\rangle^\epsilon \partial_y f\|_{L^2}
		\end{aligned}\end{equation}

		Adding \eqref{eq:est0 of u2 py omega}--\eqref{eq:est3 of u2 py omega} together, we arrive
		\begin{align*} 
				&\quad\left|\Re\left(u_2 \py  f \left\lvert\, \mathcal{M} e^{2 a A^{-\frac{1}{3}}\left|D_x\right|^{\frac{2}{3} }t }\langle D_x\rangle^{2 m}\langle\frac{1}{D_x}\rangle^{2 \epsilon} f\right.\right)\right| \\
				&\leq (2\pi)^{-\frac12}\cdot  2^{ m} \lt(\frac32\rt)^\epsilon  (2 + \sqrt{\frac43}) (1 + \Xi) \max\lt\{\sqrt{ B(\frac12 - \epsilon, m -\frac12)} , \sqrt{B(\epsilon, m -\frac12)}\rt\}    \\
				&\qquad \times \|e^{a A^{-\frac{1}{3}}\left|D_x\right|^{\frac{2}{3} } t }\langle D_x\rangle^m\langle\frac{1}{D_x}\rangle^\epsilon \partial_x \nabla \Delta^{-1} \omega\|_{L^2}  \|e^{a A^{-\frac{1}{3}}\left|D_x\right|^{\frac{2}{3} } t }\langle D_x\rangle^m\langle\frac{1}{D_x}\rangle^\epsilon \partial_y f\|_{L^2}\\
                &\qquad \times \|e^{a A^{-\frac{1}{3}}\left|D_x\right|^{\frac{2}{3} } t }\langle D_x\rangle^m\langle\frac{1}{D_x}\rangle^\epsilon f\|_{L^2}.\alabel{eq:est of u2 py omega}
		\end{align*}
		By integrating \eqref{eq:est of u1 px omega} and \eqref{eq:est of u2 py omega} over $(0, t)$, we have
		\begin{equation*}\begin{aligned}
				&\frac1A \int_{0}^{t} \left|\Re\left(u \cdot \nabla f \left\lvert\, \mathcal{M} e^{2 a A^{-\frac{1}{3}}\left|D_x\right|^{\frac{2}{3} }t }\langle D_x\rangle^{2 m}\langle\frac{1}{D_x}\rangle^{2 \epsilon} f\right.\right)\right| dt  \\
				\leq & (2\pi)^{-\frac12}  (1 + \Xi)\frac{2^{ m} }{A } \lt(\frac32\rt)^\epsilon  \max\lt\{\sqrt{ B(\frac12 - \epsilon, m -\frac12)}, \sqrt{B(\epsilon, m -\frac12)}\rt\}\\
                &\times \lt[   2^{\frac13}\sqrt{\frac43}\|e^{a A^{-\frac{1}{3}}\left|D_x\right|^{\frac{2}{3} } t }\langle D_x\rangle^m\langle\frac{1}{D_x}\rangle^\epsilon \omega\|_{L^\infty L^2} \|e^{a A^{-\frac{1}{3}}\left|D_x\right|^{\frac{2}{3} } t }\langle D_x\rangle^m\langle\frac{1}{D_x}\rangle^\epsilon\left|D_x\right|^{\frac{1}{3}} f\|_{L^2 L^2} ^{\frac{3}{2}}\rt.  \\
				&\qquad\times   \lt.         \|e^{a A^{-\frac{1}{3}}\left|D_x\right|^{\frac{2}{3} } t }\langle D_x\rangle^m\langle\frac{1}{D_x}\rangle^\epsilon \px f\|_{L^2 L^2} ^{\frac{1}{2}}  \rt. \\
				& \lt.\quad +   \|e^{a A^{-\frac{1}{3}}\left|D_x\right|^{\frac{2}{3} } t }\langle D_x\rangle^m\langle\frac{1}{D_x}\rangle^\epsilon f\|_{L^\infty L^2}  \|e^{a A^{-\frac{1}{3}}\left|D_x\right|^{\frac{2}{3} } t }\langle D_x\rangle^m\langle\frac{1}{D_x}\rangle^\epsilon \px \nabla \Delta^{-1} \omega \|_{L^2 L^2} \rt.  \\
				&  \lt. \qquad \times  \lt(   (1+\sqrt{2})     \|e^{a A^{-\frac{1}{3}}\left|D_x\right|^{\frac{2}{3} } t }\langle D_x\rangle^m\langle\frac{1}{D_x}\rangle^\epsilon \nabla f\|_{L^2 L^2} \rt.\rt.\\
                &\lt.\lt.\qquad \qquad+ (2 + \sqrt{\frac43}) \|e^{a A^{-\frac{1}{3}}\left|D_x\right|^{\frac{2}{3} } t }\langle D_x\rangle^m\langle\frac{1}{D_x}\rangle^\epsilon \py f\|_{L^2 L^2} \rt) \rt] \\
				\leq  &  C_{st} 	\frac{2^{ m} }{A^\frac12}  \lt(\frac32\rt)^\epsilon \max\lt\{ \sqrt{B(\frac12 - \epsilon, m -\frac12)} , \sqrt{B(\epsilon, m -\frac12)}\rt\}   \\
                &\quad \times \|  \langle D_x\rangle^m  \langle \frac{1}{D_x}\rangle^\epsilon   \omega \|_{X}   \|  \langle D_x\rangle^m  \langle \frac{1}{D_x}\rangle^\epsilon   f \|_{X}^2  ,
		\end{aligned}\end{equation*}
		where $$C_{st}:= (2\pi)^{-\frac12} (1 + \Xi)\lt[  2^\frac13 \sqrt{\frac43} \lt(\frac{\theta_1\Xi}{2\pi} - 2a \rt)^{-\frac34}  2^{-\frac14}               +   (3 + \sqrt{2} + \sqrt{\frac43})  \left(\frac{\Xi}{2\pi}\right)^{-\frac12}    \left(2-\frac{\theta_2\Xi}{2\pi}\right)^{-\frac12}   \rt].$$
		The proof is complete.
	\end{proof}

	\subsection{Bootstrap argument} 
	
	In this paper, we use the standard bootstrap argument to prove the main theorems.

	\subsubsection{Without coupling the fluid equations} \label{subsubsec}

	Let us define $T$ to be the end-point of the largest interval $[0, T]$ such that the following hypotheses hold for all $0 \leq t \leq T$ :
	\begin{equation}\begin{aligned}  \label{eq:hp}
			&\|  \langle D_x\rangle^{m}  \langle \frac{1}{D_x}\rangle^\epsilon    n \|_{X} \leq 1.01 Q,  \\
			&E_\infty (t):=  \| n\|_{L^\infty L^\infty} \leq 1.01 Q_\infty, 
	\end{aligned}\end{equation}
	where $Q\geq 1$ and $Q_\infty$ will be determined in the proof. We shall show the following proposition to improve the hypotheses above.
	\begin{Prop} \label{main prop3}
		Under the same assumptions of Theorem \ref{thm:PKS}, there exists a positive constant $\Lambda$ depending on $\epsilon, m,\alpha, \theta_1, \theta_2, \Xi$ and $a$ 
        such that if  $A>\Lambda\left( \|n_{\rm in}\|_{Y_{m, \epsilon}}^2+1\right)^{9/2} $, there hold
        $$
		\|  \langle D_x\rangle^{m}  \langle \frac{1}{D_x}\rangle^\epsilon  n \|_{X} \leq Q \quad \text{and} \quad E_\infty(t) \leq Q_\infty
		$$
        for all $0<t<T$.
        
      When $a=  (2000\pi)^{-1}$ and $\Xi=\theta_1 = \theta_2=1$.
        \begin{itemize}
        \item \textbf{Case of $\alpha>0$}: one takes $\epsilon=1/4, m=9/{10}$
        and has
        \begin{equation*}
            \begin{aligned}
                \Lambda \leq  \max\lt\{389256 \alpha^{-\frac12},~ 239197\alpha^{-\frac34}\rt\}.
            \end{aligned}
        \end{equation*}
        \item \textbf{Case of $\alpha=0$}: one takes $\epsilon=5/{12}, m=7/{10}$ and has
        \begin{equation*}
            \begin{aligned}
                \Lambda \leq  2058614.
            \end{aligned}
        \end{equation*}
    \end{itemize}
	\end{Prop}
	\begin{proof}[Proof of Theorem \ref{thm:PKS}]
		Proposition \ref{main prop3} with the local well-posedness (for example, see \cite{Wei2018} for $L^p$ initial data) of the system \eqref{eq:main2} implies that $T=+\infty$, and thus completes the proof.
	\end{proof}

	\subsubsection{With coupling the fluid equations}

	Denote the energy functional:
	\begin{equation*}  
		\begin{aligned}
			E(t):= \sqrt { \|  \langle D_x\rangle^m  \langle \frac{1}{D_x}\rangle^\epsilon   \omega \|_{X}^2     +   \|  \langle D_x\rangle^{m}  \langle \frac{1}{D_x}\rangle^\epsilon    n \|_{X}^2  + \|  \langle D_x\rangle^{m}  \langle \frac{1}{D_x}\rangle^\epsilon |D_x|^\frac13   n \|_{X}^2 } ,
		\end{aligned}
	\end{equation*}
	where $m,\, \epsilon >0$. Let us define $T$ to be the end-point of the largest interval $[0, T]$ such that the following hypotheses hold for all $0 \leq t \leq T$ :
	\begin{equation}\begin{aligned} \label{eq:bootstap}
			E(t) \leq 1.01 K,  \quad
			E_\infty (t) \leq 1.01 K_\infty, 
	\end{aligned}\end{equation}
	where $K\geq 1$ and $K_\infty$ will be determined in the proof. Note that the $T$ here may differ from the one in Subsection \ref{subsubsec}. Without loss of generality, we do not distinguish between them in what follows. We will show the following proposition to improve the above hypotheses.
	\begin{Prop} \label{main prop}
		Under the same assumptions of Theorem \ref{Thm}, there exists a positive constant $\bar{\Lambda}$  depending on $\theta_1$, $\theta_2$, $\Xi$, $a$, $\epsilon$, $m$, and $\alpha$
		such that if $$A>\bar{\Lambda} \left(\|(\omega_{\rm in}, n_{\rm in}, |D_x|^{1/3} n_{\rm in})\|_{Y_{m, \epsilon}}^2+1\right)^{\frac92},$$
        there hold
		$$
		E(t) \leq K, \quad E_\infty(t) \leq K_\infty
		$$
		for all $0 \leq  t \leq T$.
	\end{Prop}	
	\begin{proof}[Proof of Theorem \ref{Thm}]
		Proposition \ref{main prop}, along with the local well-posedness of the system \eqref{eq:main}, implies that $T=+\infty$. The proof of Theorem \ref{Thm} is complete.
	\end{proof}

	\section{Proof of Proposition \ref{main prop3}}	 \label{sec.3}
	In this section, we will estimate $n$ and then complete the proof of Proposition \ref{main prop3}.
	
	\subsection{Estimate of $n$}
	For the PKS system near a large Couette flow \eqref{eq:main2} (without coupling the fluid equations), we have the following estimate for $n$.
	
	\begin{Lem} \label{lem:est of n'}
		Let $\theta_1$, $\theta_2$, $\Xi$ and $a$ satisfy the positivity conditions \eqref{eq:positivity conditions}. Assume that $0\leq t \leq T$. \\
        \textbf{Case of $\alpha>0$:} for $0<\epsilon< 1/2<m$, it holds
		\begin{equation*}\begin{aligned}  
				 \|  \langle D_x\rangle^{m}  \langle \frac{1}{D_x}\rangle^\epsilon   n  \|_{X}^2 
				&\leq   \|  \langle D_x\rangle^{m}  \langle \frac{1}{D_x}\rangle^\epsilon   n_{\mathrm{in}} \|_{L^2}^2 + 2  \bigg[\frac { C_{ch1}}{A^\frac12 } +\frac{C_{ch2}}{A^\frac13} \bigg] \frac { 2^{  m   }  }{\alpha^\frac14 }  \lt(\frac32\rt)^\epsilon \| \langle D_x\rangle^{m  }  \langle \frac{1}{D_x}\rangle^\epsilon   n \|_{X}^3 \\
                &\qquad \times\max\lt\{\sqrt{B(\frac12 - \epsilon , m -\frac12)}   , \sqrt{B(\epsilon , m -\frac12) } \rt\}  ,
		\end{aligned}\end{equation*}
		where 
		\begin{equation*}\begin{aligned}
				C_{ch1}  &= (2\pi)^{-\frac12}  (1 + \Xi)  \cdot 2^{\frac5{12}}\sqrt{\frac43}    \lt(\frac{\theta_1\Xi}{2\pi} - 2a \rt)^{-\frac34}   2^{-\frac14}, \\
				C_{ch2}  &= (2\pi)^{-\frac12}  (1 + \Xi)  \cdot 2^{-\frac14}\lt(\sqrt{\frac43}  + 1 + \sqrt{\frac12}\rt)   \lt(\frac{\theta_1\Xi}{2\pi} - 2a \rt)^{-\frac12}  \lt(2- \frac{\theta_2\Xi}{2\pi}  \rt)^{-\frac12};
		\end{aligned}\end{equation*}
		\textbf{Case of $\alpha=0$:} for $1/3<\epsilon< 1/2<m$, it holds
		\begin{equation*}\begin{aligned}  
				 \|  \langle D_x\rangle^{m}  \langle \frac{1}{D_x}\rangle^\epsilon   n  \|_{X}^2 
				&\leq   \|  \langle D_x\rangle^{m}  \langle \frac{1}{D_x}\rangle^\epsilon   n_{\mathrm{in}} \|_{L^2}^2 + 2^\frac54  \bigg[\frac { C_{ch1}}{A^\frac12 } +\frac{C_{ch3}}{A^\frac13} \bigg]  2^{  m   }    \lt(\frac32\rt)^\epsilon \| \langle D_x\rangle^{m  }  \langle \frac{1}{D_x}\rangle^\epsilon   n \|_{X}^3 \\
				&\qquad \times  \max\lt\{\sqrt{B(\frac12 - \epsilon , m -\frac12)  } , \sqrt{B(\epsilon-\frac13 , m -\frac12)}  \rt\}     ,
		\end{aligned}\end{equation*}
		where
		\begin{equation*}\begin{aligned}
				C_{ch3}  := (2\pi)^{-\frac12}  (1 + \Xi)  \cdot 2^{-\frac14}\lt(\sqrt{\frac43}  + 1 + 1\rt)    \lt(\frac{\theta_1\Xi}{2\pi} - 2a \rt)^{-\frac12}  \lt(2- \frac{\theta_2\Xi}{2\pi}  \rt)^{-\frac12}.
		\end{aligned}\end{equation*}   
	\end{Lem}
	
	\begin{proof}
		Taking the inner product of $\eqref{eq:main2}_1$ with $2 \mathcal{M} e^{2 a A^{-{1}/{3}}\left|D_x\right|^{{2}/{3}}  t } \langle D_x \rangle^{2 m} \langle{D_x}^{-1} \rangle^{2 \epsilon} n$, and using \eqref{eq:crucial M}, Lemma \ref{lem:M12} and \eqref{eq:bound of M}, we have
		\begin{align*}
			&\quad \frac{d}{d t}\|\sqrt{\mathcal{M}} e^{a A^{-\frac{1}{3}}\left|D_x\right|^{\frac{2}{3} }t}\langle D_x\rangle^m\langle\frac{1}{D_x}\rangle^\epsilon n\|_{L^2}^2   + \frac{\Xi}{2\pi} \|e^{a A^{-\frac{1}{3}}\left|D_x\right|^{\frac{2}{3} }t }\langle D_x\rangle^m\langle\frac{1}{D_x}\rangle^\epsilon \partial_x \nabla  \Delta^{-1} n \|_{L^2}^2   \\
			& \quad\quad + \lt(\frac{\theta_1 \Xi}{2\pi}  - 2 a\rt)\frac{1}{ A^{\frac{1}{3}}}\|e^{a A^{-\frac{1}{3}}\left|D_x\right|^{\frac{2}{3} } t }\langle D_x\rangle^m\langle\frac{1}{D_x}\rangle^\epsilon\left|D_x\right|^{\frac{1}{3}} n\|_{L^2}^2 \\
			&\quad\quad + \lt(2 - \frac{\theta_2\Xi}{2\pi}\rt)\frac{1}{A} \|e^{a A^{-\frac{1}{3}}\left|D_x\right|^{\frac{2}{3} }t}\langle D_x\rangle^m\langle\frac{1}{D_x}\rangle^\epsilon \py n\|_{L^2}^2 + \frac{2}{A} \|e^{a A^{-\frac{1}{3}}\left|D_x\right|^{\frac{2}{3} }t}\langle D_x\rangle^m\langle\frac{1}{D_x}\rangle^\epsilon \px n\|_{L^2}^2  \\
			&   \leq \frac2A \left|\Re\left( \nabla\cdot(n\nabla c)\left\lvert\, \mathcal{M} e^{2 a A^{-\frac{1}{3}}\left|D_x\right|^{\frac{2}{3} }t }\langle D_x\rangle^{2 m}\langle\frac{1}{D_x}\rangle^{2 \epsilon} n\right.\right)\right|   .
		\end{align*}
		Next, we estimate the term $\nabla\cdot (n  \nabla c)$. Fourier transform shows
		\begin{equation*}\begin{aligned}   
				&\quad   \left|\Re\left( \nabla\cdot(n\nabla c) \left\lvert\, \mathcal{M} e^{2 a A^{-\frac{1}{3}}\left|D_x\right|^{\frac{2}{3} }t }\langle D_x\rangle^{2 m   }\langle\frac{1}{D_x}\rangle^{2 \epsilon} n\right.\right)\right| \\
				&\leq (2\pi)^{-\frac12}  \int_{\mathbb{R}^2} e^{2 a A^{-\frac{1}{3}}|k|^{\frac{2}{3} } t }\langle k\rangle^{2 m  }\langle\frac{1}{k}\rangle^{2 \epsilon}\left|\int_{\mathbb{R}} \mathcal{M}\left(k, D_y\right) \nabla_k n_k(y) \cdot\nabla_l c_l(y)   n_{k-l}(y) d y\right| d k d l .
		\end{aligned}\end{equation*}
		
		In the following, we employ two different approaches (steps) to estimate $\nabla \cdot (n \nabla c)$. For the $\alpha > 0$ case, we formally transform $\| \nabla^2 c\|_{L^2}$ and $\| \nabla c\|_{L^2}$ to $\| n\|_{L^2}$ by the delicate elliptic estimate. 
		For the $\alpha  = 0$ case, we control $\| \nabla^2 c\|_{L^2}$ by $\| n\|_{L^2}$.

		\underline{\bf Step I: The proof for the $\alpha > 0$ case.} 
		
		$\bullet$~Case 1:   $\frac{|k-l|}{2} \leq |k| \leq 2|k-l|$. By $\eqref{eq:trick7}_1$ and $|k|^\frac23 \leq 2^\frac23 |k-l|^\frac23$, we have
		\begin{equation*}\begin{aligned}  
				&\quad \int_{\frac{|k-l|}{2} \leq|k| \leq 2|k-l|} e^{2 a A^{-\frac{1}{3}}|k|^{\frac{2}{3} } t }\langle k\rangle^{2 m  } \langle\frac{1}{k}\rangle^{2 \epsilon}   \left|   \int_{\mathbb{R}} \mathcal{M}\left(k, D_y\right) k n_k(y)  l c_l(y)   n_{k-l}(y) d y\right| d k d l  \\
				&\leq 2^{m+\epsilon+\frac23 }  (1 + \Xi)  \|e^{a A^{-\frac{1}{3}}|k|^{\frac{2}{3} }t}\langle k\rangle^{m }\langle\frac{1}{k}\rangle^\epsilon |k|^\frac13 \| n_k\|_{L_y^2}\|_{L_k^2}     \|e^{a A^{-\frac{1}{3}}|l|^{\frac{2}{3} } t }|l|\| c_l \|_{L_y^{\infty}}\|_{L_l^1} \\
				&\qquad \times\|e^{a A^{-\frac{1}{3}}|k-l|^{\frac{2}{3} } t }\langle k-l\rangle^{m  } \langle\frac{1}{k-l}\rangle^\epsilon  |k-l|^\frac23\| n_{k-l}\|_{L_y^2}\|_{L_{k-l}^2} \\
				&\leq \frac{2^{m+\epsilon+\frac5{12}}  }{\alpha^\frac14}  (1 + \Xi) \sqrt{B(\epsilon + \frac12,m-\frac12)}  \|e^{a A^{-\frac{1}{3}}\left|D_x\right|^{\frac{2}{3} } t }\langle D_x\rangle^{m  }\langle\frac{1}{D_x}\rangle^\epsilon |D_x|^\frac13  n \|_{L^2}^\frac32        \\
				&\qquad \times  \|e^{a A^{-\frac{1}{3}}\left|D_x\right|^{\frac{2}{3} } t }\langle D_x\rangle^m\langle\frac{1}{D_x}\rangle^\epsilon n\|_{L^2} \|e^{a A^{-\frac{1}{3}}\left|D_x\right|^{\frac{2}{3} } t }\langle D_x\rangle^{m  }\langle\frac{1}{D_x}\rangle^\epsilon \partial_x n\|_{L^2}^\frac12,
		\end{aligned}\end{equation*}
		where we used Lemma \ref{lem:Elliptic estimate} to get
		\begin{equation} \label{eq:temp0.341}
			\begin{aligned}
				& \quad \|e^{a A^{-\frac{1}{3}}|l|^{\frac{2}{3} } t }  |l|  \| c_l \|_{L_y^{\infty}}\|_{L_l^1} \leq \|e^{a A^{-\frac{1}{3}}|l|^{\frac{2}{3}} t}   \|   |l|  c_l \|_{L_y^2}^\frac12   \|   |l| \py  c_l \|_{L_y^2}^\frac12    \|_{L_l^1} \\
				& \leq \|e^{a A^{-\frac{1}{3}}|l|^{\frac{2}{3} }t}\langle l\rangle^m\langle\frac{1}{l}\rangle^\epsilon  \|   |l|  c_l \|_{L_y^2}^\frac12   \|   |l| \py  c_l \|_{L_y^2}^\frac12  \|_{L_l^2}\|  \langle l\rangle^{-m}\langle\frac{1}{l}\rangle^{-\epsilon} \|_{L_l^2}  \\
				&\leq (\frac{ B^2(\epsilon + \frac12,m -\frac12 )}   {2\alpha})^\frac14 \|e^{a A^{-\frac{1}{3}}\left|D_x\right|^{\frac{2}{3} }t}\langle D_x\rangle^m\langle\frac{1}{D_x}\rangle^\epsilon  n\|_{L^2} .
			\end{aligned}
		\end{equation} 
		Similarly, it follows that
		\begin{align*}  
				&\quad \int_{\frac{|k-l|}{2} \leq|k| \leq 2|k-l|} e^{2 a A^{-\frac{1}{3}}|k|^{\frac{2}{3} } t }\langle k\rangle^{2 m  } \langle\frac{1}{k}\rangle^{2 \epsilon}     \left|\int_{\mathbb{R}} \mathcal{M}\left(k, D_y\right) n_k(y)  \py \lt[n_{k-l}(y) \py c_l(y)\rt]    d y\right| d k d l  \\
				&\leq 2^{m+\epsilon } \cdot (1 + \Xi) \|e^{a A^{-\frac{1}{3}}|k|^{\frac{2}{3} }t}\langle k\rangle^{m }\langle\frac{1}{k}\rangle^\epsilon  \|\py n_k\|_{L_y^2}\|_{L_k^2}     \|e^{a A^{-\frac{1}{3}}|l|^{\frac{2}{3} } t } \| \py c_l \|_{L_y^{\infty}}\|_{L_l^1} \\
				&\qquad \times\|e^{a A^{-\frac{1}{3}}|k-l|^{\frac{2}{3} } t }\langle k-l\rangle^{m  } \langle\frac{1}{k-l}\rangle^\epsilon   \| n_{k-l}\|_{L_y^2}\|_{L_{k-l}^2} \\
				&\leq \frac{ 2^{m+\epsilon - \frac{1}{4}}   }{\alpha^\frac14 } \cdot (1 + \Xi) \sqrt{B(\epsilon+\frac16, m -\frac16)}   \|e^{a A^{-\frac{1}{3}}\left|D_x\right|^{\frac{2}{3} } t }\langle D_x\rangle^{m }\langle\frac{1}{D_x}\rangle^\epsilon \py n \|_{L^2}         \\
				&\qquad \times\|e^{a A^{-\frac{1}{3}}\left|D_x\right|^{\frac{2}{3} } t }\langle D_x\rangle^m\langle\frac{1}{D_x}\rangle^\epsilon |D_x|^\frac13  n\|_{L^2}  \|e^{a A^{-\frac{1}{3}}\left|D_x\right|^{\frac{2}{3} } t }\langle D_x\rangle^{m  }\langle\frac{1}{D_x}\rangle^\epsilon   n\|_{L^2} ,
		\end{align*}
		by using
		\begin{equation} \label{eq:temp0.34}
			\begin{aligned}
				&\quad \|e^{a A^{-\frac{1}{3}}|l|^{\frac{2}{3} } t }     \| \py c_l \|_{L_y^{\infty}}\|_{L_l^1} \leq \|e^{a A^{-\frac{1}{3}}|l|^{\frac{2}{3}} t}  \|\py c_l \|_{L_y^2}^{\frac12}\|\py^{2}c_{l}\|_{L_y^{2}}^{\frac12}\|_{L_l^1} \\
				& \leq \|e^{a A^{-\frac{1}{3}}|l|^{\frac{2}{3} }t}\langle l\rangle^m\langle\frac{1}{l}\rangle^\epsilon |l|^\frac13 \|  \py c_l \|_{L_y^2}\|_{L_l^2}^{\frac12}\|e^{a A^{-\frac{1}{3}}|l|^{\frac{2}{3} }t}\langle l\rangle^m\langle\frac{1}{l}\rangle^\epsilon |l|^\frac13 \|  \py^{2} c_l \|_{L_y^2}\|_{L_l^2}^{\frac12}   \||l|^{-\frac{1}{3}}\langle l\rangle^{-m}\langle\frac{1}{l}\rangle^{-\epsilon}\|_{L_l^2} \\
				&\leq (\frac1{2\alpha})^\frac14 \sqrt{B(\epsilon+\frac16, m -\frac16)} \|e^{a A^{-\frac{1}{3}}\left|D_x\right|^{\frac{2}{3} }t}\langle D_x\rangle^m\langle\frac{1}{D_x}\rangle^\epsilon |D_x|^\frac13  n\|_{L^2} .
			\end{aligned}
		\end{equation}

		$\bullet$~Case 2:   $2|k-l|<|k|$.
		By $\eqref{eq:trick8}_1$, $1\leq |l|^\frac13 |k-l|^{-\frac13}$ and 
		\begin{equation} \label{eq:temp1.01}
			\begin{aligned}
				&\quad \|e^{a A^{-\frac{1}{3}}|l|^{\frac{2}{3} } t } \langle l\rangle^{m   }\langle\frac{1}{l}\rangle^\epsilon |l|^\frac{1}{3} \| \nabla_l c_l \|_{L_y^{\infty}}\|_{L_l^2} \leq \|e^{a A^{-\frac{1}{3}}|l|^{\frac{2}{3} } t } \langle l\rangle^{m   }\langle\frac{1}{l}\rangle^\epsilon  |l|^\frac13\| \partial_y\nabla_l    c_l \|_{L_y^2}^{\frac12}\|\nabla_l c_{l}\|_{L_y^{2}}^{\frac12}\|_{L_l^2} \\
				& \leq \|e^{a A^{-\frac{1}{3}}|l|^{\frac{2}{3} }t}\langle l\rangle^m\langle\frac{1}{l}\rangle^\epsilon |l|^\frac13 \|\nabla_l c_l \|_{L_y^2}\|_{L_l^2}^{\frac12}\|e^{a A^{-\frac{1}{3}}|l|^{\frac{2}{3} }t}\langle l\rangle^m\langle\frac{1}{l}\rangle^\epsilon |l|^\frac13 \| \py\nabla_l  c_l \|_{L_y^2}\|_{L_l^2}^{\frac12}    \\
				&\leq (\frac1{2\alpha})^\frac14 \|e^{a A^{-\frac{1}{3}}\left|D_x\right|^{\frac{2}{3} }t}\langle D_x\rangle^m\langle\frac{1}{D_x}\rangle^\epsilon |D_x|^\frac13  n\|_{L^2} ,
			\end{aligned}
		\end{equation} 
		we get
		\begin{equation*}\begin{aligned}   
				&\quad\int_{2|k-l|<|k|} e^{2 a A^{-\frac{1}{3}}|k|^{\frac{2}{3} } t }\langle k\rangle^{2 m  } \langle\frac{1}{k}\rangle^{2 \epsilon}   \left|\int_{\mathbb{R}} \mathcal{M}\left(k, D_y\right) \nabla_k n_k(y) \cdot \nabla_l c_l(y)   n_{k-l}(y) d y  \right|d k d l \\
				&\leq 2^{m  }  \lt(\frac32\rt)^\epsilon (1 + \Xi)  \|e^{a A^{-\frac{1}{3}}|k|^{\frac{2}{3} }t}\langle k\rangle^{m }\langle\frac{1}{k}\rangle^\epsilon  \| \nabla_k n_k\|_{L_y^2}\|_{L_k^2}     \|e^{a A^{-\frac{1}{3}}|l|^{\frac{2}{3} } t } \langle l\rangle^{m   }\langle\frac{1}{l}\rangle^\epsilon |l|^\frac{1}{3} \| \nabla_l c_l \|_{L_y^{\infty}}\|_{L_l^2} \\
				&\qquad \times\|e^{a A^{-\frac{1}{3}}|k-l|^{\frac{2}{3} } t }  \langle k-l\rangle^{m}\langle\frac{1}{k-l}\rangle^{\epsilon} \| n_{k-l}\|_{L_y^2}\|_{L_{k-l}^2} \||k-l|^{-\frac{1}{3}}\langle k-l\rangle^{-m}\langle\frac{1}{k-l}\rangle^{-\epsilon}\|_{L_{k-l}^2}\\
				&\leq \frac{ 2^{m  -\frac1{4}}  }{\alpha^\frac14 }  \lt(\frac32\rt)^\epsilon (1 + \Xi)  \sqrt{B(\epsilon+\frac16, m -\frac16)}    \|e^{a A^{-\frac{1}{3}}\left|D_x\right|^{\frac{2}{3} } t }\langle D_x\rangle^{m  }\langle\frac{1}{D_x}\rangle^\epsilon \nabla  n \|_{L^2}        \\
				&\qquad \times   \|e^{a A^{-\frac{1}{3}}\left|D_x\right|^{\frac{2}{3} } t }\langle D_x\rangle^{m }\langle\frac{1}{D_x}\rangle^\epsilon |D_x|^\frac13 n\|_{L^2} \|e^{a A^{-\frac{1}{3}}\left|D_x\right|^{\frac{2}{3} } t }\langle D_x\rangle^{m  }\langle\frac{1}{D_x}\rangle^\epsilon  n\|_{L^2}.
		\end{aligned}\end{equation*}

		$\bullet$~Case 3:   $2|k|<|k-l|$.
		Using $\eqref{eq:trick9}_1$, 
		\begin{equation}\begin{aligned} \label{eq:trick10}
				1\leq  \langle \frac{1}{k-l} \rangle^\frac13 |k-l|^{\frac13} \leq \lt(\frac32\rt)^\epsilon \langle\frac{1}{l}\rangle^\epsilon\langle\frac{1}{k-l}\rangle^\epsilon\left(\chi_{\{\frac16\leq\epsilon<\frac12 \}}+\langle\frac{1}{k}\rangle^{\frac{1}{3}-2 \epsilon}\chi_{\{0<\epsilon<\frac16\}}\right)|k-l|^{\frac13},
		\end{aligned}\end{equation} 
		and 	
		\begin{equation} \label{eq:temp1.002}
			\begin{aligned}
				&\quad \|e^{a A^{-\frac{1}{3}}|l|^{\frac{2}{3} } t }  \langle l\rangle^{m   }\langle\frac{1}{l}\rangle^\epsilon   \| \nabla_l c_l \|_{L_y^{\infty}}\|_{L_l^2} \leq \|e^{a A^{-\frac{1}{3}}|l|^{\frac{2}{3}} t} \langle l\rangle^{m   }\langle\frac{1}{l}\rangle^\epsilon    \| \nabla_l c_l \|_{L_y^2}^{\frac12}\| \nabla\py c_l \|_{L_y^2}^{\frac12}\|_{L_l^2} \\
				& \leq (\frac{1}{2\alpha})^\frac14 \|e^{a A^{-\frac{1}{3}}\left|D_x\right|^{\frac{2}{3} }t}\langle D_x\rangle^{m }\langle\frac{1}{D_x}\rangle^\epsilon      n\|_{L^2},
			\end{aligned}
		\end{equation} 
		we obtain
		\begin{equation*}\begin{aligned}  
				&\quad\int_{2|k|<|k-l|} e^{2 a A^{-\frac{1}{3}}|k|^{\frac{2}{3} } t }\langle k\rangle^{2 m} \langle\frac{1}{k}\rangle^{2 \epsilon} \left|   \int_{\mathbb{R}} \mathcal{M}\left(k, D_y\right) \nabla_k n_k(y) \cdot \nabla_l c_l(y)   n_{k-l}(y) d y \right|d k d l \\
				& \leq \lt(\frac32\rt)^\epsilon (1 + \Xi) \|e^{a A^{-\frac{1}{3}}|k|^{\frac{2}{3} } t }\langle k\rangle^{m }\langle\frac{1}{k}\rangle^\epsilon   \|\nabla_k  n_k\|_{L_y^2}\|_{L_k^2}  \|e^{a A^{-\frac{1}{3}}|l|^{\frac{2}{3} } t }\langle l\rangle^{m   }\langle\frac{1}{l}\rangle^\epsilon    \|\nabla_l  c_l\|_{L_y^\infty}\|_{L_l^2} \\
				& \quad \times   \|e^{a A^{-\frac{1}{3}}|k-l|^{\frac{2}{3} } t }\langle k-l\rangle^{m  }\langle\frac{1}{k-l}\rangle^\epsilon|k-l|^{\frac{1}{3}}\| n_{k-l}\|_{L_y^{2}}\|_{L_{k-l}^2} \\
                &\quad \times \|\left(\langle\frac{1}{k}\rangle^\epsilon \chi_{\{\frac16\leq\epsilon<\frac12 \}} +\langle\frac{1}{k}\rangle^{\frac{1}{3}-\epsilon} \chi_{\{0<\epsilon<\frac16\}} \right)\langle k\rangle^{-m}\|_{L_k^2}\\
				& \leq \frac{ 2^{-\frac14}  }{\alpha^\frac14} \lt(\frac32\rt)^\epsilon   (1 + \Xi) \max\lt\{\sqrt{B(\frac12 - \epsilon , m -\frac12)},\sqrt{ B(\epsilon+\frac16, m -\frac12)  }\rt\}    \\
				&\quad \times \|e^{a A^{-\frac{1}{3}}\left|D_x\right|^{\frac{2}{3} } t }\langle D_x\rangle^{m  }\langle\frac{1}{D_x}\rangle^\epsilon n  \|_{L^2}   \|e^{a A^{-\frac{1}{3}}\left|D_x\right|^{\frac{2}{3} } t }\langle D_x\rangle^{m  }\langle\frac{1}{D_x}\rangle^\epsilon |D_x|^\frac13 n\|_{L^2} \\
                &\quad \times \|e^{a A^{- \frac{1}{3}}\left|D_x\right|^{\frac{2}{3} } t }\langle D_x\rangle^m\langle\frac{1}{D_x}\rangle^\epsilon \nabla n \|_{L^2}.
		\end{aligned}\end{equation*}
		
		Collecting all the estimates above, we arrive at
		\begin{equation*}\begin{aligned}
				&\quad\left|\Re\left(\nabla \cdot(n\nabla c) \left\lvert\, \mathcal{M} e^{2 a A^{-\frac{1}{3}}\left|D_x\right|^{\frac{2}{3} }t }\langle D_x\rangle^{2 m  }\langle\frac{1}{D_x}\rangle^{2 \epsilon} n\right.\right)\right| \\
				&\leq(2\pi)^{-\frac12} \cdot \frac {2^{m }  }{\alpha^\frac14} \lt(\frac32\rt)^\epsilon    (1 + \Xi)  \max\lt\{\sqrt{B(\frac12 - \epsilon , m -\frac12) }  ,\sqrt{ B(\epsilon , m -\frac12) } \rt\}       \\
				&  \quad \times \lt[ 2^{\frac5{12}}\sqrt{\frac43} \|e^{a A^{-\frac{1}{3}}\left|D_x\right|^{\frac{2}{3} } t }\langle D_x\rangle^{m  }\langle\frac{1}{D_x}\rangle^\epsilon n\|_{L^2} \|e^{a A^{-\frac{1}{3}}\left|D_x\right|^{\frac{2}{3} } t }\langle D_x\rangle^{m  }\langle\frac{1}{D_x}\rangle^\epsilon \partial_x n\|_{L^2}^\frac12 \rt. \\
				&\lt.\qquad \quad\times \|e^{a A^{-\frac{1}{3}}\left|D_x\right|^{\frac{2}{3} } t }\langle D_x\rangle^{m  }\langle\frac{1}{D_x}\rangle^\epsilon |D_x|^\frac13  n \|_{L^2}^\frac32 \rt.\\
				&\lt.\qquad  +   2^{-\frac14}\lt(\sqrt{\frac43}  + 1 + \sqrt{\frac12}\rt)          \|e^{a A^{-\frac{1}{3}}\left|D_x\right|^{\frac{2}{3} } t }\langle D_x\rangle^{m   }\langle\frac{1}{D_x}\rangle^\epsilon |D_x|^\frac13  n\|_{L^2} \rt.\\
				&\lt.\qquad\quad \times \|e^{a A^{-\frac{1}{3}}\left|D_x\right|^{\frac{2}{3} } t }\langle D_x\rangle^{m  }\langle\frac{1}{D_x}\rangle^\epsilon \nabla  n \|_{L^2} \|e^{a A^{-\frac{1}{3}}\left|D_x\right|^{\frac{2}{3} } t }\langle D_x\rangle^{m  }\langle\frac{1}{D_x}\rangle^\epsilon   n\|_{L^2}\rt] , 
		\end{aligned}\end{equation*}
		which implies (by \eqref{eq:X norm})
		\begin{equation}\begin{aligned}  \label{eq:nabla n nabla c}
				&\quad \frac1A \int_{0}^{t} \left|\Re\left(\nabla \cdot(n\nabla c) \left\lvert\, \mathcal{M} e^{2 a A^{-\frac{1}{3}}\left|D_x\right|^{\frac{2}{3} }t }\langle D_x\rangle^{2 m  }\langle\frac{1}{D_x}\rangle^{2 \epsilon} n\right.\right)\right| dt \\
				&\leq  \big[\frac {C_{ch1}}{A^\frac12 } +\frac{C_{ch2}}{A^\frac13} \big] \frac{ 2^{m } } {\alpha^\frac14 } \lt(\frac32\rt)^\epsilon \max\lt\{\sqrt{B(\frac12 - \epsilon , m -\frac12) }  , \sqrt{B(\epsilon , m -\frac12) } \rt\}  \| \langle D_x\rangle^{m   }  \langle \frac{1}{D_x}\rangle^\epsilon   n \|_{X}^3,
		\end{aligned}\end{equation}
		where
		\begin{equation*}\begin{aligned}
				C_{ch1}  &:= (2\pi)^{-\frac12}  (1 + \Xi)  \cdot 2^{\frac5{12}}\sqrt{\frac43}   \lt(\frac{\theta_1\Xi}{2\pi} - 2a \rt)^{-\frac34}  2^{-\frac14}, \\
				C_{ch2}  &:= (2\pi)^{-\frac12}  (1 + \Xi)  \cdot 2^{-\frac14}\lt(\sqrt{\frac43}  + 1 + \sqrt{\frac12}\rt)   \lt(\frac{\theta_1\Xi}{2\pi} - 2a \rt)^{-\frac12}  \lt(2- \frac{\theta_2\Xi}{2\pi}  \rt)^{-\frac12}.
		\end{aligned}\end{equation*}

		\underline{\bf Step II: The proof for the $\alpha = 0$ case.} 
		When $\alpha = 0$, $\epsilon$ should additionally satisfy $ \epsilon > 1/3$. In fact, for example, we rewrite \eqref{eq:temp0.341} as
		\begin{align*}
				& \quad \|e^{a A^{-\frac{1}{3}}|l|^{\frac{2}{3} } t }  |l|  \| c_l \|_{L_y^{\infty}}\|_{L_l^1} \leq \|e^{a A^{-\frac{1}{3}}|l|^{\frac{2}{3}} t}   |l|^\frac12  \|   \nabla_l  c_l \|_{L_y^2}    \|_{L_l^1} \nonumber\\
				& \leq \|e^{a A^{-\frac{1}{3}}|l|^{\frac{2}{3} }t}\langle l\rangle^m\langle\frac{1}{l}\rangle^\epsilon   |l|^\frac13  \|   |l| \nabla_l c_l \|_{L_y^2}   \|_{L_l^2}  \||l|^{-\frac56}  \langle l\rangle^{-m}\langle\frac{1}{l}\rangle^{-\epsilon}\|_{L_l^2}  \nonumber\\
				&\leq  \sqrt{ B(\epsilon - \frac13,m +\frac13 ) }   \|e^{a A^{-\frac{1}{3}}\left|D_x\right|^{\frac{2}{3} }t}\langle D_x\rangle^m\langle\frac{1}{D_x}\rangle^\epsilon |D_x|^{\frac{1}{3}} n\|_{L^2}, \alabel{eq:pks13}
		\end{align*}
		\eqref{eq:temp1.01} as
		\begin{equation*}  
			\begin{aligned}
				\|e^{a A^{-\frac{1}{3}}|l|^{\frac{2}{3} } t } \langle l\rangle^{m   }\langle\frac{1}{l}\rangle^\epsilon |l|^\frac{5}{6} \| \nabla_l c_l \|_{L_y^{\infty}}\|_{L_l^2} &\leq \|e^{a A^{-\frac{1}{3}}|l|^{\frac{2}{3} } t } \langle l\rangle^{m   }\langle\frac{1}{l}\rangle^\epsilon  |l|^\frac13\|\nabla_l^2    c_l \|_{L_y^2} \|_{L_l^2} \\
				&\leq \|e^{a A^{-\frac{1}{3}}\left|D_x\right|^{\frac{2}{3} }t}\langle D_x\rangle^m\langle\frac{1}{D_x}\rangle^\epsilon |D_x|^\frac13  n\|_{L^2} ,
			\end{aligned}
		\end{equation*} 
		\eqref{eq:trick10} as
		\begin{equation*}\begin{aligned} 
				1\leq  \langle \frac{1}{k-l} \rangle^\frac56 |k-l|^{\frac56} \leq \lt(\frac32\rt)^\epsilon \langle\frac{1}{l}\rangle^\epsilon\langle\frac{1}{k-l}\rangle^\epsilon\left(\chi_{\{\frac5{12}\leq\epsilon<\frac12 \}}+\langle\frac{1}{k}\rangle^{\frac56-2 \epsilon}\chi_{\{\frac13<\epsilon<\frac5{12}\}}\right)|k-l|^{\frac56},
		\end{aligned}\end{equation*}
		and \eqref{eq:temp1.002} as
		\begin{equation*} 
			\begin{aligned}
				&\quad \|e^{a A^{-\frac{1}{3}}|l|^{\frac{2}{3} } t }  \langle l\rangle^{m} \langle\frac{1}{l}\rangle^\epsilon  |l|^\frac12 \| \nabla_l c_l \|_{L_y^{\infty}}\|_{L_l^2} \leq \|e^{a A^{-\frac{1}{3}}|l|^{\frac{2}{3}} t} \langle l\rangle^{m   }\langle\frac{1}{l}\rangle^\epsilon    \| \nabla_l^2 c_l \|_{L_y^2} \|_{L_l^2} \\
				& \leq  \|e^{a A^{-\frac{1}{3}}\left|D_x\right|^{\frac{2}{3} }t}\langle D_x\rangle^{m }\langle\frac{1}{D_x}\rangle^\epsilon      n\|_{L^2}.
			\end{aligned}
		\end{equation*}
		Hence, we get 
		\begin{align*}
				&\quad\left|\Re\left(\nabla \cdot(n\nabla c) \left\lvert\, \mathcal{M} e^{2 a A^{-\frac{1}{3}}\left|D_x\right|^{\frac{2}{3} }t }\langle D_x\rangle^{2 m  }\langle\frac{1}{D_x}\rangle^{2 \epsilon} n\right.\right)\right| \\
				&\leq(2\pi)^{-\frac12} \cdot 2^\frac14 \cdot {2^{m }  }  \lt(\frac32\rt)^\epsilon    (1 + \Xi)  \max\lt\{\sqrt{B(\frac12 - \epsilon , m -\frac12) }  , \sqrt{B(\epsilon-\frac13 , m -\frac12)  }\rt\}          \\
				&  \quad \times \lt[ 2^{\frac5{12}}\sqrt{\frac43} \|e^{a A^{-\frac{1}{3}}\left|D_x\right|^{\frac{2}{3} } t }\langle D_x\rangle^{m  }\langle\frac{1}{D_x}\rangle^\epsilon n\|_{L^2} \|e^{a A^{-\frac{1}{3}}\left|D_x\right|^{\frac{2}{3} } t }\langle D_x\rangle^{m  }\langle\frac{1}{D_x}\rangle^\epsilon \partial_x n\|_{L^2}^\frac12 \rt. \\
				&\lt.\qquad\quad \times \|e^{a A^{-\frac{1}{3}}\left|D_x\right|^{\frac{2}{3} } t }\langle D_x\rangle^{m  }\langle\frac{1}{D_x}\rangle^\epsilon |D_x|^\frac13  n \|_{L^2}^\frac32 \rt.\\
				&\lt.\qquad  +   2^{-\frac14}\lt(\sqrt{\frac43}  + 1 + 1\rt)          \|e^{a A^{-\frac{1}{3}}\left|D_x\right|^{\frac{2}{3} } t }\langle D_x\rangle^{m   }\langle\frac{1}{D_x}\rangle^\epsilon |D_x|^\frac13  n\|_{L^2} \rt.\\
				&\lt.\qquad\quad \times \|e^{a A^{-\frac{1}{3}}\left|D_x\right|^{\frac{2}{3} } t }\langle D_x\rangle^{m  }\langle\frac{1}{D_x}\rangle^\epsilon \nabla  n \|_{L^2} \|e^{a A^{-\frac{1}{3}}\left|D_x\right|^{\frac{2}{3} } t }\langle D_x\rangle^{m  }\langle\frac{1}{D_x}\rangle^\epsilon   n\|_{L^2}\rt] , 
		\end{align*}
		which implies 
		\ben\label{eq:alpha0 bound} 
				&&\quad \frac1A \int_{0}^{t} \left|\Re\left(\nabla \cdot(n\nabla c) \left\lvert\, \mathcal{M} e^{2 a A^{-\frac{1}{3}}\left|D_x\right|^{\frac{2}{3} }t }\langle D_x\rangle^{2 m  }\langle\frac{1}{D_x}\rangle^{2 \epsilon} n\right.\right)\right| dt \nonumber\\
				&\leq&  \lt[\frac {C_{ch1}}{A^\frac12 } +\frac{C_{ch3}}{A^\frac13} \rt] 2^\frac14 \cdot { 2^{m } } \lt(\frac32\rt)^\epsilon  \| \langle D_x\rangle^{m   }  \langle \frac{1}{D_x}\rangle^\epsilon   n \|_{X}^3\nonumber\\
                &&\quad\times \max\lt\{\sqrt{B(\frac12 - \epsilon , m -\frac12) }  , \sqrt{B(\epsilon -\frac13 , m -\frac12)  }\rt\} ,
		\een
		where
		\begin{equation*}\begin{aligned}
				C_{ch3}  := (2\pi)^{-\frac12}  (1 + \Xi)  \cdot 2^{-\frac14}\lt(\sqrt{\frac43}  + 1 + 1\rt)   \lt(\frac{\theta_1\Xi}{2\pi} - 2a \rt)^{-\frac12}  \lt(2- \frac{\theta_2\Xi}{2\pi}  \rt)^{-\frac12}.
		\end{aligned}\end{equation*}
        The proof is complete.
	\end{proof}

	To improve the hypotheses \eqref{eq:hp} and \eqref{eq:bootstap}, we are left to complete the $L^\infty L^\infty$ estimate of $n$, which will be achieved by virtue of the Moser--Alikakos iteration.
	\begin{Lem} \label{lem:n L infty L infty}
		For $0\leq t \leq T$, it holds that
		\begin{equation*}\begin{aligned}
				\| n \|_{L_t^\infty L^\infty } \leq 2^7 ( \|\nabla c\|_{L_t^\infty L^4}^2 + 1 ) \left( \|n\|_{ L_t^\infty L^2} +  M + \| n_{\mathrm{in}} \|_{L^\infty}  +1 \right).
		\end{aligned}\end{equation*}
	\end{Lem}
	\begin{proof}
		For $p=2^j$ with $j \geq 1$, multiplying $\eqref{eq:main}_1$ (or $\eqref{eq:main2}_1$) by $2 p n^{2 p-1}$, and integrating the resulting equation over $\mathbb{R}^2 $, we have
		\begin{align*}
			& \frac{d}{d t}\|n^p\|_{L^2}^2+\frac{2(2 p-1)}{p A}\|\nabla n^p\|_{L^2}^2 \\
			= & \frac{2(2 p-1)}{A} \int_{  \mathbb{R}^2 } n^p \nabla c \cdot \nabla n^p d x d y 
			\leq  \frac{2(2 p-1)}{A}\|n^p \nabla c\|_{L^2}\|\nabla n^p\|_{L^2} \\
			\leq & \frac{2 p-1}{p A}\|\nabla n^p\|_{L^2}^2+\frac{(2 p-1) p}{A}\|n^p \nabla c\|_{L^2}^2 \\
			\leq & \frac{2 p-1}{p A}\|\nabla n^p\|_{L^2}^2+\frac{2(2 p-1) p}{A} \|n^p\|_{L^2}\|\nabla n^p\|_{L^2}  \|\nabla c\|_{L^4}^2 \\
			\leq &   \frac{5(2 p-1)}{4p A}\|\nabla n^p\|_{L^2}^2+\frac{4(2 p-1) p^3}{A}\|n^p\|_{L^2}^2\|\nabla c\|_{L^4}^4 ,
		\end{align*}
		which implies that
		\begin{equation}\begin{aligned} \label{eq:est of np0}
				\frac{d}{d t}\|n^p\|_{L^2}^2 \leq -\frac{1}{2 A}\|\nabla n^p\|_{L^2}^2 + \frac{8 p^4}{A}\|n^p\|_{L^2}^2 \|\nabla c\|_{L^4}^4 .
		\end{aligned}\end{equation}
		Note that the Nash inequality
		$$
		\|n^p\|_{L^2} \leq  \|n^p\|_{L^1}^{\frac{1}{2}}\|\nabla n^p\|_{L^2}^{\frac{1}{2}}
		$$
		gives
		$$
		-\|\nabla n^p\|_{L^2}^2 \leq-\frac{\|n^p\|_{L^2}^4}{\|n^p\|_{L^1}^2}.
		$$
		Thus, we deduce from \eqref{eq:est of np0} and Lemma \ref{lem:Elliptic estimate} that
		\begin{equation}\begin{aligned} \label{eq:temp2}
				\frac{d}{d t}\|n^p\|_{L^2}^2 &\leq   - \frac{\|n^p\|_{L^2}^4}{ 2A \|n^p\|_{L^1}^2}      +\frac{8 p^4}{A}\|n^p\|_{L^2}^2 \|\nabla c\|_{L^4}^4 \\
				& \leq - \frac{\|n^p\|_{L^2}^4}{ 4A \|n^p\|_{L^1}^2}      +\frac{64 p^8}{A}\|n^p\|_{L^1}^2 \|\nabla c\|_{L^4}^8 \\
				&\leq - \frac{1}{ 4A \|n^p\|_{L^1}^2}  \lt( \|n^p\|_{L^2}^4  - 256p^8 \|n^p\|_{L^1}^4 \|\nabla c\|_{L^4}^8  \rt).
		\end{aligned}\end{equation}
		
		Now, we claim that for all $0\leq t\leq T$, it holds that
		\begin{equation}\begin{aligned} \label{eq:claim}
				\sup_{s\in [0, t]} \|n^p\|_{L^2}^2 \leq \max\lt\{ 2\|n_{\mathrm{in}}^p \|_{L^2}^2, \sup_{s\in [0, t]}\lt( 32p^4 \|n^p\|_{L^1}^2 \|\nabla c\|_{L^4}^4 \rt),1 \rt\}.
		\end{aligned}\end{equation}
		Otherwise, denote the first time $t^* \in (0, t]$ such that 
		\begin{equation}\begin{aligned} \label{eq:temp3}
				\|n^p (t^*)\|_{L^2}^2 = \max\lt\{ 2\|n_{\mathrm{in}}^p \|_{L^2}^2, \, \sup_{s\in [0, t]}\lt( 32p^4 \|n^p\|_{L^1}^2 \|\nabla c\|_{L^4}^4 \rt),1 \rt\},
		\end{aligned}\end{equation}
		then 
		\begin{equation}\begin{aligned} \label{eq:temp4}
				\lt. \frac{d}{d s}\|n^p\|_{L^2}^2 \rt|_{s=t^*}  \geq 0.
		\end{aligned}\end{equation}
		However, we infer from \eqref{eq:temp2} and \eqref{eq:temp3} that 
		\begin{align*}
				&\lt. \quad\frac{d}{d s}\|n^p\|_{L^2}^2 \rt|_{s=t^*} \\
				&\leq  -\frac{1}{ 4A \|n^p(t^*)\|_{L^1}^2} \lt(  \|n^p(t^*)\|_{L^2}^4  - 256p^8 \|n^p(t^*)\|_{L^1}^4 \|\nabla c(t^*)\|_{L^4}^8 \rt) \\
				&\leq -\frac{1}{ 4A \|n^p(t^*)\|_{L^1}^2} \\
				&\qquad \times \lt[ \max\lt\{ 2\|n_{\mathrm{in}}^p \|_{L^2}^2, \, \sup_{s\in [0, t]}\lt( 32p^4 \|n^p\|_{L^1}^2 \|\nabla c\|_{L^4}^4 \rt), 1 \rt\}^2  - 256p^8 \|n^p(t^*)\|_{L^1}^4 \|\nabla c(t^*)\|_{L^4}^8  \rt] \\
				&< 0,
		\end{align*}
		which contradicts \eqref{eq:temp4}. Hence, by \eqref{eq:claim}, we get
		\begin{equation*}\begin{aligned}
				\sup_{s\in [0, t]} \|n^p\|_{L^2}^2 \leq \max\lt\{ 2\|n_{\mathrm{in}}^p \|_{L^2}^2, \,  32p^4 \|\nabla c\|_{L_t^\infty L^4}^4   \sup_{s\in [0, t]} \|n^p\|_{ L^1}^2 ,1 \rt\},
		\end{aligned}\end{equation*}
		which implies
		\begin{equation}\begin{aligned} \label{eq:temp5}
				\sup_{s\in [0, t]} \int_{  \mathbb{R}^2 } n^{2^{j+1}}  dxdy  \leq \max\lt\{ 2\int_{  \mathbb{R}^2 } n_{\mathrm{in}}^{2^{j+1}}  dxdy  ,   32p^4  \|\nabla c\|_{L_t^\infty L^4}^4  \lt( \sup_{s\in [0, t]}   \int_{  \mathbb{R}^2 } n^{2^{j}}  dxdy \rt)^2, 1 \rt\}.
		\end{aligned}\end{equation}

		Note that for some $0\leq \theta_j \leq 1$, it holds that
		\begin{equation*}\begin{aligned}
				\|n_{\mathrm{in}} \|_{L^{2^{j+1}}} \leq \| n_{\mathrm{in}} \|_{L^1}^{\theta_j} \| n_{\mathrm{in}} \|_{L^\infty}^{1 - \theta_j} \leq  M + \| n_{\mathrm{in}} \|_{L^\infty}   =: H,
		\end{aligned}\end{equation*}
		which implies
		\begin{equation*}\begin{aligned} 
				\int_{  \mathbb{R}^2 } n_{\mathrm{in}}^{2^{j+1}}  dxdy \leq H^{2^{j+1}}.
		\end{aligned}\end{equation*}
		Therefore, we rewrite \eqref{eq:temp5} by 
		\begin{equation*}\begin{aligned}
				\sup_{s\in [0, t]} \int_{  \mathbb{R}^2 } n^{2^{j+1}}  dx dy \leq \max\lt\{ 2H^{2^{j+1}}  ,   32 \|\nabla c\|_{L_t^\infty L^4}^4   16^j \lt( \sup_{s\in [0, t]}   \int_{  \mathbb{R}^2 } n^{2^{j}}  dxdy \rt)^2, 1 \rt\}.
		\end{aligned}\end{equation*}
		Denote $B := H + \| n \|_{L_t^\infty L^2} + 1$. For $j =1$, we have
		\begin{equation*}\begin{aligned}
				\sup_{s\in [0, t]} \int_{  \mathbb{R}^2 } n^{2^{1+1}}  dx dy &\leq \max\lt\{ 2H^{2^{1+1}}  ,   32  \|\nabla c\|_{L_t^\infty L^4}^4   16 \lt( \sup_{s\in [0, t]}   \int_{  \mathbb{R}^2 } n^{2}  dxdy \rt)^2, 1 \rt\} \\
				&\leq \max\lt\{ 2B^{2^{1+1}}  ,  \lt(  32  \|\nabla c\|_{L_t^\infty L^4}^4 +1 \rt)  16 B^4 \rt\} \\
				&\leq \lt( 32  \|\nabla c\|_{L_t^\infty L^4}^4    + 1\rt) 16  B^4.
		\end{aligned}\end{equation*}
		Then, for $j =2$, we have
		\begin{equation*}\begin{aligned}
				\sup_{s\in [0, t]} \int_{  \mathbb{R}^2 } n^{2^{2+1}}  dx dy &\leq \max\lt\{ 2H^{2^{2+1}}  ,   32\|\nabla c\|_{L_t^\infty L^4}^4   16^2 \lt( \sup_{s\in [0, t]}   \int_{  \mathbb{R}^2 } n^{2^2}  dxdy \rt)^2, 1 \rt\} \\
				&\leq \max\lt\{ 2B^{2^{2+1}}  ,   ( 32  \|\nabla c\|_{L_t^\infty L^4}^4 +1 )  16^2   \lt[ (32  \|\nabla c\|_{L_t^\infty L^4}^4    + 1) 16  B^4 \rt] ^2 \rt\} \\
				&\leq (32  \|\nabla c\|_{L_t^\infty L^4}^4    + 1)^3 16^4 B^8.
		\end{aligned}\end{equation*}
		By iteration, for $j=k$ we get
		\begin{equation}\begin{aligned} \label{eq:temp6}
				\sup_{s\in [0, t]} \int_{  \mathbb{R}^2 } n^{2^{k+1}}  dx dy \leq   \lt( 32 \|\nabla c\|_{L_t^\infty L^4}^4   + 1 \rt)^{a_{k}} 16^{b_{k}}  B^{2^{k+1}},
		\end{aligned}\end{equation}
		where $a_k =  2^{k  } -1 $ and $b_k = 2^{k +1 } - k -2$.
		Thus, it follows from \eqref{eq:temp6} that
		\begin{equation*}\begin{aligned}
				\sup_{s\in [0, t]}  \lt(  \int_{  \mathbb{R}^2 } n^{2^{k+1}}  dx dy  \rt)^{\frac{1}{2^{k+1}}} \leq \lt( 32 \|\nabla c\|_{L_t^\infty L^4}^4   +1  \rt)^{\frac{2^{k} -1 }{2^{k+1}}} 16^{\frac{  2^{k + 1} - k - 2}{2^{k+1}}}  B.
		\end{aligned}\end{equation*}
		Hence,  by letting $k \rightarrow +\infty$, one obtains
		\begin{equation*}\begin{aligned}
				\| n \|_{L_t^\infty L^\infty } \leq 2^7\lt( \|\nabla c\|_{L_t^\infty L^4}^2 + 1 \rt) \left( \|n\|_{ L_t^\infty L^2} +  M + \| n_{\mathrm{in}} \|_{L^\infty}  +1 \right),
		\end{aligned}\end{equation*}
		which completes the proof.
	\end{proof}

	\subsection{Closing the energy}
	With a priori estimates for $n$ at hand, one can let $A$ sufficiently large so that the energy can be closed as follows.

    \begin{proof}[Proof of Proposition \ref{main prop3}] The proof is divided into two cases.
    
		{\bf \underline{ Case of $\alpha>0$.}}
		By \eqref{eq:hp}, and Lemma \ref{lem:est of n'}, we have for $0<t\leq T$ that
		\begin{equation}\begin{aligned}  \label{eq:est of n233}
				\|  \langle D_x\rangle^{m}  \langle \frac{1}{D_x}\rangle^\epsilon   n  \|_{X}^2 
				\leq  \|  \langle D_x\rangle^{m}  \langle \frac{1}{D_x}\rangle^\epsilon   n_{\mathrm{in}} \|_{L^2}^2   + 1.01^3 \cdot 2 r_1 \lt(  \frac{ C_{ch1}}{A^\frac12 } +\frac{C_{ch2}}{A^\frac13} \rt)   Q^3,
		\end{aligned}\end{equation}
		where $$r_1 : =  \frac { 2^{  m   }  }{\alpha^\frac14 } \lt(\frac32\rt)^\epsilon \max\lt\{\sqrt{B(\frac12 - \epsilon , m -\frac12) }  , \sqrt{ B(\epsilon , m -\frac12)  }\rt\} \geq 1. $$
		Choose large enough $
		\bar{A}_1 := {\Lambda}_1(\Xi, \theta_1, \theta_2, a, r_1) Q^9 $ such that if $A> \bar{A}_1$, there holds that
		\begin{equation}\begin{aligned} \label{eq:need to be verify}
				1.01^3 \cdot 2r_1   \lt( \frac { C_{ch1}}{A^\frac12 } +\frac{C_{ch2}}{A^\frac13} \rt)  Q^3 \leq 1.
		\end{aligned}\end{equation}
        After fixing the parameters $\Xi$, $\theta_1$, $\theta_2$, $a$, $\epsilon$, and $m$, one can derive an upper bound estimate for $\bar{A}_1$. 
        
        Choosing $a = (2000\pi)^{-1}$ and setting $\Xi = \theta_1 = \theta_2 = 1$, we obtain
        \begin{align*}
            C_{ch1}  &= (2\pi)^{-\frac12}  (1 + \Xi)  \cdot 2^{\frac5{12}}\sqrt{\frac43}    \lt(\frac{\theta_1\Xi}{2\pi} - 2a \rt)^{-\frac34}   2^{-\frac14}\\
        &= (2\pi)^{-\frac12}  2  \cdot 2^{\frac5{12}}\sqrt{\frac43}    \lt(\frac{1}{2\pi} -  (1000\pi)^{-1}\rt)^{-\frac34}   2^{-\frac14}\\
         &=  2^\frac53 \pi^{\frac14}3^{-\frac12}(\frac{1000}{499})^{\frac34}<\frac{16}5 \pi^{\frac14},
        \end{align*}
        where we used $2^\frac53<\frac{16}{5}$ and $3^{-\frac12}(\frac{1000}{499})^{\frac34}<1$.
        By $2^\frac34 \lt(\frac{1000}{499}\rt)^\frac12<\frac{12}{5}$ and $\lt(\frac{1}{4\pi-1}\rt)^\frac12<\frac3{10}$, we estimate
        \begin{align*}
             C_{ch2}  &= (2\pi)^{-\frac12}  (1 + \Xi)  \cdot 2^{-\frac14}\lt(\sqrt{\frac43}  + 1 + \sqrt{\frac12}\rt)   \lt(\frac{\theta_1\Xi}{2\pi} - 2a \rt)^{-\frac12}  \lt(2- \frac{\theta_2\Xi}{2\pi}  \rt)^{-\frac12}\\
        &= (2\pi)^{-\frac12}  2 \cdot 2^{-\frac14}\lt(\sqrt{\frac43}  + 1 + \sqrt{\frac12}\rt)   \lt(\frac{1}{2\pi} - \frac{1}{1000\pi} \rt)^{-\frac12}  \lt(2- \frac{1}{2\pi}  \rt)^{-\frac12}\\
        &<  2^{\frac34}\lt(\sqrt{\frac43}  + 1 + \sqrt{\frac12}\rt)(\frac{1000}{499})^{\frac12}\left(\frac{\pi}{4\pi-1}\right)^\frac12\\
        &<\frac{12}{5}\cdot3\left(\frac{\pi}{4\pi-1}\right)^\frac12<\frac{54}{25}\pi^{\frac12}.
        \end{align*}
        Additionally, letting $\epsilon=\frac14$ and $m=\frac9{10}$, we have
        \begin{align*}
            B(\frac14,\frac25) &= \int_{0}^1 t^{-\frac34} (1-t)^{-\frac35} dt\\
            &\leq 2^{\frac35} \int_0^\frac12 t^{-\frac34} dt + 2^\frac34 \int_{\frac12}^1 (1-t)^{-\frac35} dt\\
            &= 2^\frac{47}{20} + 5\cdot2^{-\frac{13}{20}}
        \end{align*}
        and
        \begin{align*}
             &2^{  m   }  \lt(\frac32\rt)^\epsilon \max\lt\{\sqrt{B(\frac12 - \epsilon , m -\frac12) }  , \sqrt{ B(\epsilon , m -\frac12)  }\rt\} \\
             &= 2^{\frac9{10}} \cdot \lt(\frac{3}{2}\rt)^\frac14 \sqrt{2^\frac{47}{20} + 5\cdot2^{-\frac{13}{20}}} = 3^\frac14 \sqrt{2^\frac{73}{20} + 5 \cdot2^\frac{13}{20}} < 3^\frac14 \sqrt{26}.
        \end{align*}
        Then, one can take $r_1=3^{\frac14}\sqrt{26}\alpha^{-\frac14}$. By solving 
        \begin{align*}
            1.01^3 \cdot 2r_1\cdot\frac{16}{5} \pi^\frac14 \frac{1}{x_1^\frac12} = \frac1{10},
        \end{align*}
        we get
        \begin{align*}
            x_1 = \lt(10\cdot1.01^3 \cdot 2 \cdot 3^\frac14 \sqrt{26} \cdot\frac{16}{5}\pi^\frac14\rt)^2 \alpha^{-\frac12}< 11^2 \cdot 4\cdot 16^2\pi\alpha^{-\frac12} < 389256 \alpha^{-\frac12},
        \end{align*}
where we used $1.01^6\cdot\frac{26}{25}\cdot 3^{\frac12}<1.1^2\pi^{\frac12}$.
        By solving 
        \begin{align*}
            1.01^3 \cdot 2r_1\cdot\frac{54}{25} \pi^\frac12 \frac{1}{x_2^\frac13} = \frac9{10},
        \end{align*}
        we have
        \begin{align*}
            x_2 = \lt(\frac{10}{9}\cdot1.01^3 \cdot 2 \cdot 3^\frac14 \sqrt{26} \cdot\frac{54}{25}\pi^\frac12\rt)^3 \alpha^{-\frac34}< \lt(\frac{11}{9}\rt)^3\cdot8\cdot\frac{54^3}{5^3}\cdot13 \alpha^{-\frac34} < 239197\alpha^{-\frac34},
        \end{align*}
        where we used $3^\frac34 \pi^\frac32<13$ and $1.01^9\cdot\frac{\sqrt{26}^3}{5^3}<1.1^3 $.   
Therefore, define $$\bar{A}_1:=  \max\lt\{389256 \alpha^{-\frac12}, 239197\alpha^{-\frac34}\rt\}Q^9,$$ and then when $A>\bar{A}_1$, \eqref{eq:need to be verify} holds.



		Denote 
        $$
        Q:=    \sqrt{ \|  \langle D_x\rangle^{\frac9{10}  }  \langle \frac{1}{D_x}\rangle^\frac14   n_{\mathrm{in}} \|_{L^2}^2  +  1  }.
        $$  
        By Lemma \ref{lem:n L infty L infty}, Lemma \ref{lem:Elliptic estimate} and \eqref{eq:hp}, we have for $0<t\leq T$ that 
		\begin{align*}
				E_\infty (t) & \leq 2^7  \lt[\lt(\frac1{2\pi \alpha}\rt)^\frac12 \|n\|_{ L_t^\infty L^2}^2+1\rt] \left( \|n\|_{ L_t^\infty L^2} +  M + \| n_{\mathrm{in}} \|_{L^\infty}  +1 \right) \\
				& \leq  2^7  \lt[\lt(\frac1{2\pi \alpha}\rt)^\frac12\|n\|_{X}^2+1\rt] \left( \|n\|_{X} +  M + \| n_{\mathrm{in}} \|_{L^\infty}  +1 \right) \\
				& \leq 2^7  \lt[\lt(\frac1{2\pi\alpha}\rt)^\frac12 \cdot 1.01^2\cdot Q^2+1\rt]				
				\left(1.01 Q +  M + \| n_{\mathrm{in}} \|_{L^\infty}  +1 \right) =: Q_{\infty}.
		\end{align*}

		{\bf \underline{ Case of $\alpha = 0$.}} Rewrite \eqref{eq:est of n233} as 
		\begin{equation*}\begin{aligned}  
				\|  \langle D_x\rangle^{m}  \langle \frac{1}{D_x}\rangle^\epsilon   n  \|_{X}^2 
				\leq  \|  \langle D_x\rangle^{m}  \langle \frac{1}{D_x}\rangle^\epsilon   n_{\mathrm{in}} \|_{L^2}^2   + 1.01^3 \cdot 2^\frac54 \cdot  r_2 \bigg[\frac { C_{ch1}}{A^\frac12 } +\frac{C_{ch3}}{A^\frac13} \bigg]   Q^3,
		\end{aligned}\end{equation*}
        where $$r_2(\epsilon,m):= 2^{ m}  \lt(\frac32\rt)^\epsilon  \max\lt\{\sqrt{ B(\frac12 - \epsilon, m -\frac12)} ,\sqrt{ B(\epsilon - \frac13 , m -\frac12)}\rt\} \geq 1.$$
		Choose large enough $
		\bar{A}_2 = {\Lambda}_2(\Xi, \theta_1, \theta_2, a, r_2) Q^9  $ such that if $A> \bar{A}_2$, it holds that
		\begin{equation}\begin{aligned} \label{eq:need to be verify2}
				1.01^3 \cdot 2^\frac54  \bigg[\frac { C_{ch1}}{A^\frac12 } +\frac{C_{ch3}}{A^\frac13} \bigg] r_2  Q^3 \leq 1.
		\end{aligned}\end{equation}
        Direct calculation gives
        \begin{align*}
            C_{ch3} &=  (2\pi)^{-\frac12}  (1 + \Xi)  \cdot 2^{-\frac14}\lt(\sqrt{\frac43}  + 1 + 1\rt)    \lt(\frac{\theta_1\Xi}{2\pi} - 2a \rt)^{-\frac12}  \lt(2- \frac{\theta_2\Xi}{2\pi}  \rt)^{-\frac12}\\
&= (2\pi)^{-\frac12}  2 \cdot 2^{-\frac14}\lt(\sqrt{\frac43}  + 1 + 1\rt)    \lt(\frac{1}{2\pi} - \frac{1}{1000\pi} \rt)^{-\frac12}  \lt(2- \frac{1}{2\pi}  \rt)^{-\frac12}\\
&=  2^{\frac14}\lt(\sqrt{\frac43}  + 2\rt)(\frac{1000}{499})^{\frac12}\left(\frac{2\pi}{4\pi-1}\right)^\frac12\\
&<   \lt({\frac{5}{4}}  + 2\rt) \cdot\frac{3}{10}\cdot\frac{12}{5}\pi^\frac12 < \frac{12}{5} \pi^\frac12.
        \end{align*}
When $\epsilon=\frac{5}{12}$ and $m=\frac{7}{10}$, we have
\begin{align*}
    B(\frac{1}{12},\frac15) & = \int_0^1 t^{-\frac{11}{12}}  (1-t)^{-\frac45} dt \\
    &\leq 2^\frac45 \int_0^\frac12 t^{-\frac{11}{12}} dt + 2^\frac{11}{12} \int_{\frac12}^{1} (1-t)^{-\frac45} dt =2^\frac{43}{60}\cdot17,
\end{align*}
        and then $r_2=3^{\frac{5}{12}}2^\frac{77}{120}\sqrt{17}$. Solving 
        \begin{align*}
            1.01^3 \cdot 2^\frac54 r_2\cdot\frac{16}{5} \pi^\frac14 \frac{1}{x_3^\frac12} = \frac1{10},
       \end{align*}
        we have
        \begin{align*}
            x_3 = \lt(10\cdot1.01^3 \cdot 2^\frac54 \cdot 3^\frac{5}{12} 2^\frac{77}{120} \sqrt{17} \cdot\frac{16}{5}\pi^\frac14\rt)^2 < 10^2 \cdot1.1 \cdot 2^3\cdot 17 \cdot\frac{16^2}{5^2} \cdot 9 < 1378714,
        \end{align*}
        where $3^\frac56\pi^\frac12<\frac92.$ By solving 
        \begin{align*}
            1.01^3 \cdot 2^\frac54r_2\cdot\frac{12}{5} \pi^\frac12 \frac{1}{x_4^\frac13} = \frac9{10},
        \end{align*}
        we get
        \begin{align*}
            x_4 &= \lt(\frac{10}{9}\cdot1.01^3 \cdot 2^\frac54 \cdot 3^\frac5{12} 2^\frac{77}{120} \sqrt{17} \cdot\frac{12}{5}\pi^\frac12\rt)^3 \\
            &
            < \lt(\frac{10}{9}\rt)^3\cdot1.1\cdot2^6 \cdot17\sqrt{17}\cdot \frac{12^3}{5^3} \cdot 3^\frac54\pi^\frac32<2058614,
        \end{align*}
        where we used $3^\frac54 \pi^\frac32<22.$ Therefore, let $\bar{A}_2:=  2058614Q^9$. When $A>\bar{A}_2$, \eqref{eq:need to be verify2} holds.

		Denote by 
        $$
        Q:=    \sqrt{ \|  \langle D_x\rangle^{\frac7{10}  }  \langle \frac{1}{D_x}\rangle^\frac5{12}  n_{\mathrm{in}} \|_{L^2}^2  +  1  }.
        $$  
        Using Lemma \ref{lem:n L infty L infty}, Lemma \ref{lem:Elliptic estimate} and \eqref{eq:hp}, we obtain for $0<t\leq T$ that 
		\begin{equation*}\begin{aligned}
				E_\infty (t) & \leq 2^7  \lt[ \frac4{\pi^2} (3 + 2\sqrt{2})^2 ( \|n\|_{ L_t^\infty L^2} + M )^2   +1\rt] \left( \|n\|_{ L_t^\infty L^2} +  M + \| n_{\mathrm{in}} \|_{L^\infty}  +1 \right) \\
				& \leq 2^7  \lt[ \frac4{\pi^2} (3 + 2\sqrt{2})^2 ( 1.01Q + M )^2   +1\rt] \left( 1.01 Q +  M + \| n_{\mathrm{in}} \|_{L^\infty}  +1 \right) =: Q_{\infty}.
		\end{aligned}\end{equation*}
		Combining the two cases and denoting $\Lambda:=\max\lt\{\Lambda_1, \Lambda_2\rt\}$, we complete the proof.
	\end{proof}

\section{Proof of Proposition \ref{main prop}} \label{sec.4}
In this section, we will estimate $n$ and $\omega$ and then complete the proof of Proposition \ref{main prop}.
	
	\subsection{Estimate of $n$ and $\omega$}
    For the PKSNS system near a large Couette flow \eqref{eq:main}, we have the following estimate.
	\begin{Lem} \label{lem:est of n}
		Let $\theta_1$, $\theta_2$, $\Xi$ and $a$ satisfy the positivity conditions \eqref{eq:positivity conditions}. Assume that $0\leq t \leq T$. \\
        \textbf{Case of $\alpha>0$:} for $0<\epsilon< 1/2<m$, it holds
		\begin{align} \label{eq:temp0.33}
				&\no\quad \|  \langle D_x\rangle^{m }  \langle \frac{1}{D_x}\rangle^\epsilon   n  \|_{X}^2  + \|  \langle D_x\rangle^{m }  \langle \frac{1}{D_x}\rangle^\epsilon |D_x|^\frac13   n  \|_{X}^2 \\
				&\no\leq \|  \langle D_x\rangle^{m }  \langle \frac{1}{D_x}\rangle^\epsilon   n_{\mathrm{in}} \|_{L^2}^2  +   \|  \langle D_x\rangle^{m }  \langle \frac{1}{D_x}\rangle^\epsilon  |D_x|^\frac13  n_{\mathrm{in}} \|_{L^2}^2  \\
				&\no\quad + \frac{   2^{  m}  }{A^\frac12}  \lt(\frac32\rt)^\epsilon \max\lt\{ \sqrt{B(\frac12 - \epsilon, m -\frac12) }, \sqrt{B(\epsilon, m -\frac12)}\rt\}\|  \langle D_x\rangle^m  \langle \frac{1}{D_x}\rangle^\epsilon   \omega \|_{X}  \\
				&\no\qquad\times \left[ 2C_{st}   \|  \langle D_x\rangle^m  \langle \frac{1}{D_x}\rangle^\epsilon      n \|_{X}^2  + (2C_{st}  + 2C_{hl}) \|  \langle D_x\rangle^m  \langle \frac{1}{D_x}\rangle^\epsilon  |D_x|^\frac13    n \|_{X}^2 \right]  \\
				&\no\quad + 2\big[\frac {C_{ch1}}{A^\frac12 } +\frac{C_{ch2}}{A^\frac13} \big] \frac{ 2^{m } } {\alpha^\frac14 } \lt(\frac32\rt)^\epsilon \max\lt\{\sqrt{B(\frac12 - \epsilon , m -\frac12) } ,\sqrt{ B(\epsilon , m -\frac12) } \rt\}  \| \langle D_x\rangle^{m   }  \langle \frac{1}{D_x}\rangle^\epsilon   n \|_{X}^3 \\
				&\no\quad +  2\frac{ 2^{m } } {\alpha^\frac14 } \lt(\frac32\rt)^\epsilon \max\lt\{\sqrt{B(\frac12 - \epsilon , m -\frac12)  } , \sqrt{B(\epsilon , m -\frac12) } \rt\} \lt[ \frac{C_{fl1}}{A^\frac12} \| \langle D_x\rangle^{m  }  \langle \frac{1}{D_x}\rangle^\epsilon  |D_x|^\frac13  n \|_{X}^3\rt.\\
				&\lt.\qquad  +  \frac{C_{fl2}}{A^\frac13}\lt( \| \langle D_x\rangle^{m  }  \langle \frac{1}{D_x}\rangle^\epsilon    n \|_{X}^2 + \| \langle D_x\rangle^{m  }  \langle \frac{1}{D_x}\rangle^\epsilon |D_x|^\frac13   n \|_{X}^2 \rt)^\frac32 \rt],
		\end{align}
		where 
		\begin{equation*}\begin{aligned}
				C_{fl1} &= (2\pi)^{-\frac12}   (1+ \Xi)\cdot 2^{\frac34}\sqrt{\frac43}  (\frac{\theta_1\Xi}{2\pi} - 2a)^{-\frac34}  \cdot 2^{-\frac14}  ,\\
				C_{fl2} &= (2\pi)^{-\frac12}   (1+ \Xi) \lt( \frac{2^{\frac{25}{12}} + 2^{-\frac1{12}}}{3\sqrt{3}} \rt)  (\frac{\theta_1\Xi}{2\pi} - 2a)^{-\frac12}  (2- \frac{\theta_2\Xi}{2\pi})^{-\frac12} ;
		\end{aligned}\end{equation*}
        \textbf{Case of $\alpha=0$:} for $1/3<\epsilon< 1/2<m$, it holds
		\begin{align} \label{eq:temp0.3332}
			&\quad \|  \langle D_x\rangle^{m }  \langle \frac{1}{D_x}\rangle^\epsilon   n  \|_{X}^2  + \|  \langle D_x\rangle^{m }  \langle \frac{1}{D_x}\rangle^\epsilon |D_x|^\frac13   n  \|_{X}^2 \no\\
			&\leq \|  \langle D_x\rangle^{m }  \langle \frac{1}{D_x}\rangle^\epsilon   n_{\mathrm{in}} \|_{L^2}^2  +   \|  \langle D_x\rangle^{m }  \langle \frac{1}{D_x}\rangle^\epsilon  |D_x|^\frac13  n_{\mathrm{in}} \|_{L^2}^2  \no\\
			&\quad + \frac{   2^{  m}  }{A^\frac12}  \lt(\frac32\rt)^\epsilon \max\lt\{ \sqrt{B(\frac12 - \epsilon, m -\frac12)} , \sqrt{B(\epsilon, m -\frac12)}\rt\}\|  \langle D_x\rangle^m  \langle \frac{1}{D_x}\rangle^\epsilon   \omega \|_{X}  \no\\
			&\qquad\times \left[ 2C_{st}   \|  \langle D_x\rangle^m  \langle \frac{1}{D_x}\rangle^\epsilon      n \|_{X}^2  + (2C_{st}  + 2C_{hl}) \|  \langle D_x\rangle^m  \langle \frac{1}{D_x}\rangle^\epsilon  |D_x|^\frac13    n \|_{X}^2 \right] \no \\
			&\quad + 2^\frac54 \cdot { 2^{m } }  \lt(\frac32\rt)^\epsilon \big[\frac {C_{ch1}}{A^\frac12 } +\frac{C_{ch3}}{A^\frac13} \big]      \| \langle D_x\rangle^{m   }  \langle \frac{1}{D_x}\rangle^\epsilon   n \|_{X}^3 \no\\
            &\qquad \times \max\lt\{\sqrt{B(\frac12 - \epsilon , m -\frac12) }  , \sqrt{B(\epsilon - \frac13 , m -\frac12) } \rt\}\no\\
			&\quad +  2^\frac54 \cdot { 2^{m } } \lt(\frac32\rt)^\epsilon \max\lt\{\sqrt{B(\frac12 - \epsilon , m -\frac12) }  , \sqrt{B(\epsilon -\frac13 , m -\frac12) } \rt\}  \no\\
			& \qquad \times \lt[ \frac{C_{fl1}}{A^\frac12} \| \langle D_x\rangle^{m  }  \langle \frac{1}{D_x}\rangle^\epsilon  |D_x|^\frac13  n \|_{X}^3\rt. \no\\
            &\lt.\qquad\quad+  \frac{C_{fl3}}{A^\frac13}\lt( \| \langle D_x\rangle^{m  }  \langle \frac{1}{D_x}\rangle^\epsilon    n \|_{X}^2 + \| \langle D_x\rangle^{m  }  \langle \frac{1}{D_x}\rangle^\epsilon |D_x|^\frac13   n \|_{X}^2 \rt)^\frac32 \rt],
		\end{align}
		where
		\begin{equation*}\begin{aligned}
				C_{fl3} = (2\pi)^{-\frac12}  (1+ \Xi)\cdot \lt( \frac{2^{\frac{25}{12}} + 2^{\frac5{12}}}{3\sqrt{3}} \rt)  (\frac{\theta_1\Xi}{2\pi} - 2a)^{-\frac12}  (2- \frac{\theta_2\Xi}{2\pi})^{-\frac12}  .
		\end{aligned}\end{equation*}
	\end{Lem}

	\begin{proof}
		One applies $|D_x|^{1/3}$ to $\eqref{eq:main}_1$ and has
		\begin{equation}\begin{aligned} \label{eq:dx 13 n}
				\partial_t |D_x|^\frac13  n+y \partial_x |D_x|^\frac13  n-\frac{1}{A} \Delta |D_x|^\frac13  n=-\frac{1}{A} |D_x|^\frac13  (\nabla \cdot(n \nabla c) ) - \frac1A |D_x|^\frac13  (u \cdot \nabla n).
		\end{aligned}\end{equation}
		Since 
		\begin{align*}
			\left|\big(|D_x|^\frac13 (fg) \big| |D_x|^\frac13 h\big)\right|
			&\leq \int_{\mathbb{R}^3}  |k-l|^\frac13|\hat{f}(k-l) \hat{g}(l)| |k| ^\frac13 |\hat{h}(k)| dkdldy \\
			&\quad  +  \int_{\mathbb{R}^3}   |l|^ \frac13|\hat{f}(k-l) \hat{g}(l)| |k| ^\frac13 |\hat{h}(k)| dkdldy
		\end{align*}
		holds for certain functions $f$, $g$ and $h$, we formally decompose $ |D_x|^{1/3}  (u \cdot \nabla n) $ into $u \cdot \nabla |D_x|^{1/3}  n $ and $| D_x |^{1/3}  u \cdot \nabla n$ in the $L^2$ energy estimate. Then, applying Proposition \ref{lem:est of f} to \eqref{eq:dx 13 n}, it holds that
		\begin{equation}\begin{aligned} \label{eq:temp001}
				&\quad \|  \langle D_x\rangle^m  \langle \frac{1}{D_x}\rangle^\epsilon   |D_x|^\frac13  n  \|_{X}^2 \\ 
				&\leq  \|  \langle D_x\rangle^m  \langle \frac{1}{D_x}\rangle^\epsilon   |D_x|^\frac13  n_{\mathrm{in}} \|_{L^2}^2 + 	 \frac{ 2C_{st} \cdot   2^{  m }   }{A^\frac12} \lt(\frac32\rt)^\epsilon   \max\lt\{ \sqrt{B(\frac12 - \epsilon, m -\frac12)} , \sqrt{B(\epsilon, m -\frac12)}\rt\} \\
				&\qquad\times \|  \langle D_x\rangle^m  \langle \frac{1}{D_x}\rangle^\epsilon   \omega \|_{X}   \|  \langle D_x\rangle^m  \langle \frac{1}{D_x}\rangle^\epsilon   |D_x|^\frac13  n \|_{X}^2 \\
				&\quad + \frac{2}{A}   \int_{0}^{t} \left|\Re\left(  | D_x |^\frac13  u \cdot \nabla n \left\lvert\, \mathcal{M} e^{2 a A^{-\frac{1}{3}}\left|D_x\right|^{\frac{2}{3} }t }\langle D_x\rangle^{2 m}\langle\frac{1}{D_x}\rangle^{2 \epsilon} |D_x|^\frac13  n \right.\right)\right| dt \\
				&\quad + \frac{2}{A} \int_{0}^{t} \left|\Re\left( |D_x|^\frac13  (\nabla \cdot(n \nabla c) ) \left\lvert\, \mathcal{M} e^{2 a A^{-\frac{1}{3}}\left|D_x\right|^{\frac{2}{3} }t }\langle D_x\rangle^{2 m}\langle\frac{1}{D_x}\rangle^{2 \epsilon} |D_x|^\frac13  n\right.\right)\right| dt.
		\end{aligned}\end{equation}
		Using Lemma \ref{lem:est of dx u}, we get
		\begin{equation}\begin{aligned}   \label{eq:temp002} 
				&\quad \frac{2}{A} \int_{0}^{t} \left|\Re\left(|D_x|^{\frac13}u \cdot \nabla n \left\lvert\, \mathcal{M} e^{2 a A^{-\frac{1}{3}}\left|D_x\right|^{\frac{2}{3} }t }\langle D_x\rangle^{2 m   }\langle\frac{1}{D_x}\rangle^{2 \epsilon} |D_x|^\frac13  n  \right.\right)\right|   dt  \\
				&\leq \frac{2 C_{hl}\cdot 2^{m   }  }{A^\frac12} \lt(\frac32\rt)^\epsilon  \max\lt\{ \sqrt{B(\frac12 - \epsilon, m -\frac12) }, \sqrt{B(\epsilon, m -\frac12)}\rt\} \\
				&\qquad\times \| \langle D_x\rangle^m\langle\frac{1}{D_x}\rangle^\epsilon \omega \|_{X}    \| \langle D_x\rangle^m\langle\frac{1}{D_x}\rangle^\epsilon |D_x|^\frac13 n \|_{X}^2.
		\end{aligned}\end{equation}
		Similarly, applying Proposition \ref{lem:est of f} to $\eqref{eq:main}_1$, we have
		\begin{align}  \label{eq:temp003}
				&\no\quad \|  \langle D_x\rangle^m  \langle \frac{1}{D_x}\rangle^\epsilon     n  \|_{X}^2 \\ 
				&\no\leq    \|  \langle D_x\rangle^m  \langle \frac{1}{D_x}\rangle^\epsilon     n_{\mathrm{in}} \|_{L^2}^2 + 	 \frac{ 2C_{st} \cdot   2^{  m}  }{A^\frac12}  \lt(\frac32\rt)^\epsilon  \max\lt\{\sqrt{ B(\frac12 - \epsilon, m -\frac12) }, \sqrt{B(\epsilon, m -\frac12)}\rt\} \\
				&\no\qquad\times  \|  \langle D_x\rangle^m  \langle \frac{1}{D_x}\rangle^\epsilon   \omega \|_{X}   \|  \langle D_x\rangle^m  \langle \frac{1}{D_x}\rangle^\epsilon      n \|_{X}^2 \\
				&\quad + \frac{2}{A}  \int_{0}^{t} \left|\Re\left(  \nabla \cdot(n \nabla c)   \left\lvert\, \mathcal{M} e^{2 a A^{-\frac{1}{3}}\left|D_x\right|^{\frac{2}{3} }t }\langle D_x\rangle^{2 m}\langle\frac{1}{D_x}\rangle^{2 \epsilon}   n\right.\right)\right| dt.
		\end{align}

		In the remaining part of this proof, we only provide the estimate for $|D_x|^{1/3} (\nabla \cdot(n \nabla c) )$, since the estimate for $\nabla \cdot(n \nabla c) $ has already been established in Lemmas \ref{lem:est of n'}. There holds 
		\begin{equation}\begin{aligned}  \label{eq:temp004}
				&\quad   \frac{1}{A}   \left|\Re\left( |D_x|^{\frac13} ( \nabla\cdot(n\nabla c) ) \left\lvert\, \mathcal{M} e^{2 a A^{-\frac{1}{3}}\left|D_x\right|^{\frac{2}{3} }t }\langle D_x\rangle^{2 m  }\langle\frac{1}{D_x}\rangle^{2 \epsilon} |D_x|^{\frac13} n\right.\right)\right|\\
				&\leq (2\pi)^{-\frac12}  \frac{1}{A}    \int_{\mathbb{R}^2} e^{2 a A^{-\frac{1}{3}}|k|^{\frac{2}{3} } t }\langle k\rangle^{2 m  }\langle\frac{1}{k}\rangle^{2 \epsilon}|k|^{\frac23}\left|\int_{\mathbb{R}} \mathcal{M}\left(k, D_y\right) \nabla_k n_k(y) \cdot\nabla_l c_l(y)   n_{k-l}(y) d y\right| d k d l .
		\end{aligned}\end{equation}

		\underline{\bf Step I: The proof of \eqref{eq:temp0.33}  for  the $\alpha > 0$ case.}

		$\bullet$~Case 1:   $\tfrac{|k-l|}{2} \leq |k| \leq 2|k-l|$. By \eqref{eq:trick7} and \eqref{eq:temp0.34}, we obtain 
		\begin{equation*}\begin{aligned}  
				&\quad \int_{\frac{|k-l|}{2} \leq|k| \leq 2|k-l|} e^{2 a A^{-\frac{1}{3}}|k|^{\frac{2}{3} } t }\langle k\rangle^{2 m} \langle\frac{1}{k}\rangle^{2 \epsilon}|k|^{\frac23}   \left|   \int_{\mathbb{R}} \mathcal{M}\left(k, D_y\right) k n_k(y)  l c_l(y)   n_{k-l}(y) d y\right| d k d l  \\
				&\leq 2^{m+\epsilon+1}   (1 + \Xi)  \|e^{a A^{-\frac{1}{3}}|k|^{\frac{2}{3} }t}\langle k\rangle^{m}\langle\frac{1}{k}\rangle^\epsilon |k|^\frac23 \| n_k\|_{L_y^2}\|_{L_k^2}     \|e^{a A^{-\frac{1}{3}}|l|^{\frac{2}{3} } t }|l|\| c_l \|_{L_y^{\infty}}\|_{L_l^1} \\
				&\qquad \times\|e^{a A^{-\frac{1}{3}}|k-l|^{\frac{2}{3} } t }\langle k-l\rangle^{m} \langle\frac{1}{k-l}\rangle^\epsilon  |k-l|\| n_{k-l}\|_{L_y^2}\|_{L_{k-l}^2} \\
				&\leq \frac{2^{m+\epsilon+\frac34}  }{\alpha^\frac14}  (1 + \Xi) \sqrt{B(\epsilon + \frac16,m-\frac16)}  \|e^{a A^{-\frac{1}{3}}\left|D_x\right|^{\frac{2}{3} } t }\langle D_x\rangle^{m}\langle\frac{1}{D_x}\rangle^\epsilon |D_x|^\frac23  n \|_{L^2}^\frac32        \\
				&\qquad \times  \|e^{a A^{-\frac{1}{3}}\left|D_x\right|^{\frac{2}{3} } t }\langle D_x\rangle^m\langle\frac{1}{D_x}\rangle^\epsilon |D_x|^{\frac13}n\|_{L^2} \|e^{a A^{-\frac{1}{3}}\left|D_x\right|^{\frac{2}{3} } t }\langle D_x\rangle^{m}\langle\frac{1}{D_x}\rangle^\epsilon |D_x|^{\frac13}\partial_x n\|_{L^2}^\frac12
		\end{aligned}\end{equation*}
		and
		\begin{align*}  
			&\quad \int_{\frac{|k-l|}{2} \leq|k| \leq 2|k-l|} e^{2 a A^{-\frac{1}{3}}|k|^{\frac{2}{3} } t }\langle k\rangle^{2 m} \langle\frac{1}{k}\rangle^{2 \epsilon} |k|^{\frac23}    \left|\int_{\mathbb{R}} \mathcal{M}\left(k, D_y\right) n_k(y)  \py \lt[n_{k-l}(y) \py c_l(y)\rt]    d y\right| d k d l  \\
			&\leq 2^{m+\epsilon+\frac13} (1 + \Xi) \|e^{a A^{-\frac{1}{3}}|k|^{\frac{2}{3} }t}\langle k\rangle^{m}\langle\frac{1}{k}\rangle^\epsilon |k|^{\frac13} \|\py n_k\|_{L_y^2}\|_{L_k^2}     \|e^{a A^{-\frac{1}{3}}|l|^{\frac{2}{3} } t } \| \py c_l \|_{L_y^{\infty}}\|_{L_l^1} \\
			&\qquad \times\|e^{a A^{-\frac{1}{3}}|k-l|^{\frac{2}{3} } t }\langle k-l\rangle^{m } \langle\frac{1}{k-l}\rangle^\epsilon   |k-l|^{\frac13}\| n_{k-l}\|_{L_y^2}\|_{L_{k-l}^2} \\
			&\leq \frac{ 2^{m+\epsilon+\frac{1}{12}}   }{\alpha^\frac14 }   (1 + \Xi) \sqrt{B(\epsilon+\frac16, m -\frac16)}   \|e^{a A^{-\frac{1}{3}}\left|D_x\right|^{\frac{2}{3} } t }\langle D_x\rangle^{m}\langle\frac{1}{D_x}\rangle^\epsilon \py|D_x|^{\frac13} n \|_{L^2}         \\
			&\qquad \times\|e^{a A^{-\frac{1}{3}}\left|D_x\right|^{\frac{2}{3} } t }\langle D_x\rangle^m\langle\frac{1}{D_x}\rangle^\epsilon |D_x|^\frac13  n\|_{L^2}^2 .
		\end{align*}

		$\bullet$~Case 2:   $2|k-l|<|k|$.
		By $\eqref{eq:trick8}_1$ and $|k|^\frac13 \leq 2^\frac13 |l|^\frac23 |k-l|^{-\frac13}$, we have
		\begin{align}  \label{eq:temp0.001}
				&\no\quad\int_{2|k-l|<|k|} e^{2 a A^{-\frac{1}{3}}|k|^{\frac{2}{3} } t }\langle k\rangle^{2 m  } \langle\frac{1}{k}\rangle^{2 \epsilon} |k|^\frac23 \left|  \int_{\mathbb{R}} \mathcal{M}\left(k, D_y\right) \nabla_k n_k(y) \cdot \nabla_l c_l(y)   n_{k-l}(y) d y \right| d k d l \\
				&\no\leq 2^{m  +\frac13} \lt(\frac32\rt)^\epsilon(1 + \Xi)  \|e^{a A^{-\frac{1}{3}}|k|^{\frac{2}{3} }t}\langle k\rangle^{m }\langle\frac{1}{k}\rangle^\epsilon |k|^\frac13  \| \nabla_k n_k\|_{L_y^2}\|_{L_k^2}      \||k-l|^{-\frac{1}{3}}\langle k-l\rangle^{-m}\langle\frac{1}{k-l}\rangle^{-\epsilon}\|_{L_{k-l}^2}\\
				&\no\qquad \times\|e^{a A^{-\frac{1}{3}}|k-l|^{\frac{2}{3} } t }  \langle k-l\rangle^{m}\langle\frac{1}{k-l}\rangle^{\epsilon} \| n_{k-l}\|_{L_y^2}\|_{L_{k-l}^2} \|e^{a A^{-\frac{1}{3}}|l|^{\frac{2}{3} } t } \langle l\rangle^{m  }\langle\frac{1}{l}\rangle^\epsilon |l|^\frac{2}{3} \| \nabla_l c_l \|_{L_y^{\infty}}\|_{L_l^2}\\
				&\no\leq \frac{ 2^{m +\frac1{12}}  }{\alpha^\frac14 }  \lt(\frac32\rt)^\epsilon   (1 + \Xi)  \sqrt{B(\epsilon+\frac16, m -\frac16)}    \|e^{a A^{-\frac{1}{3}}\left|D_x\right|^{\frac{2}{3} } t }\langle D_x\rangle^{m }\langle\frac{1}{D_x}\rangle^\epsilon \nabla  |D_x|^{\frac13}  n \|_{L^2}        \\
				&\qquad \times   \|e^{a A^{-\frac{1}{3}}\left|D_x\right|^{\frac{2}{3} } t }\langle D_x\rangle^{m }\langle\frac{1}{D_x}\rangle^\epsilon |D_x|^\frac23 n\|_{L^2} \|e^{a A^{-\frac{1}{3}}\left|D_x\right|^{\frac{2}{3} } t }\langle D_x\rangle^{m  }\langle\frac{1}{D_x}\rangle^\epsilon  n\|_{L^2},
		\end{align}
		where we applied 
		\begin{equation*}  
			\begin{aligned}
				&\quad \|e^{a A^{-\frac{1}{3}}|l|^{\frac{2}{3} } t } \langle l\rangle^{m   }\langle\frac{1}{l}\rangle^\epsilon |l|^\frac{2}{3} \| \nabla_l c_l \|_{L_y^{\infty}}\|_{L_l^2} \leq \|e^{a A^{-\frac{1}{3}}|l|^{\frac{2}{3} } t } \langle l\rangle^{m   }\langle\frac{1}{l}\rangle^\epsilon  |l|^\frac23\| \partial_y\nabla_l    c_l \|_{L_y^2}^{\frac12}\|\nabla_l c_{l}\|_{L_y^{2}}^{\frac12}\|_{L_l^2} \\
				& \leq \|e^{a A^{-\frac{1}{3}}|l|^{\frac{2}{3} }t}\langle l\rangle^m\langle\frac{1}{l}\rangle^\epsilon |l|^\frac23 \|\nabla_l c_l \|_{L_y^2}\|_{L_l^2}^{\frac12}\|e^{a A^{-\frac{1}{3}}|l|^{\frac{2}{3} }t}\langle l\rangle^m\langle\frac{1}{l}\rangle^\epsilon |l|^\frac23 \| \py\nabla_l  c_l \|_{L_y^2}\|_{L_l^2}^{\frac12}    \\
				&\leq (\frac1{2\alpha})^\frac14 \|e^{a A^{-\frac{1}{3}}\left|D_x\right|^{\frac{2}{3} }t}\langle D_x\rangle^m\langle\frac{1}{D_x}\rangle^\epsilon |D_x|^\frac23  n\|_{L^2} 
			\end{aligned}
		\end{equation*}
		to the second line of \eqref{eq:temp0.001}.

		$\bullet$~Case 3:   $2|k|<|k-l|$.
		Using \eqref{eq:trick10},
		\begin{equation*}  
			\begin{aligned}
				&\quad \|e^{a A^{-\frac{1}{3}}|l|^{\frac{2}{3} } t }  \langle l\rangle^{m   }\langle\frac{1}{l}\rangle^\epsilon |l|^\frac13  \| \nabla_l c_l \|_{L_y^{\infty}}\|_{L_l^2} \leq \|e^{a A^{-\frac{1}{3}}|l|^{\frac{2}{3}} t} \langle l\rangle^{m   }\langle\frac{1}{l}\rangle^\epsilon  |l|^\frac13   \| \nabla_l c_l \|_{L_y^2}^{\frac12}\| \nabla\py c_l \|_{L_y^2}^{\frac12}\|_{L_l^2} \\
				& \leq (\frac{1}{2\alpha})^\frac14 \|e^{a A^{-\frac{1}{3}}\left|D_x\right|^{\frac{2}{3} }t}\langle D_x\rangle^{m }\langle\frac{1}{D_x}\rangle^\epsilon  |D_x|^\frac13    n\|_{L^2},
			\end{aligned}
		\end{equation*}
		and
		\begin{equation*} 
			\langle k \rangle^{2m  } |k|^\frac23 \leq 2^{-\frac13} \langle l \rangle^{m    } |l|^\frac13 \langle k-l \rangle^{m   } |k-l|^\frac13,
		\end{equation*}
		we have
		\begin{align*}  
				&\quad\int_{2|k|<|k-l|} e^{2 a A^{-\frac{1}{3}}|k|^{\frac{2}{3} } t }\langle k\rangle^{2 m  } \langle\frac{1}{k}\rangle^{2 \epsilon}  |k|^\frac23 \left|  \int_{\mathbb{R}} \mathcal{M}\left(k, D_y\right) \nabla_k n_k(y) \cdot \nabla_l c_l(y)   n_{k-l}(y) d y \right| d k d l \\
				& \leq 2^{ - \frac13} \lt(\frac32\rt)^\epsilon (1 + \Xi) \|e^{a A^{-\frac{1}{3}}|k|^{\frac{2}{3} } t }\langle k\rangle^{m }\langle\frac{1}{k}\rangle^\epsilon    \|\nabla_k  n_k\|_{L_y^2}\|_{L_k^2}  \|e^{a A^{-\frac{1}{3}}|l|^{\frac{2}{3} } t }\langle l\rangle^{m   }\langle\frac{1}{l}\rangle^\epsilon  |l|^\frac13  \|\nabla_l  c_l\|_{L_y^\infty}\|_{L_l^2} \\
				& \quad \times   \|e^{a A^{-\frac{1}{3}}|k-l|^{\frac{2}{3} } t }\langle k-l\rangle^{m   }\langle\frac{1}{k-l}\rangle^\epsilon|k-l|^{\frac{2}{3}}\| n_{k-l}\|_{L_y^{2}}\|_{L_{k-l}^2} \\
                &\quad \times \|\left(\langle\frac{1}{k}\rangle^\epsilon \chi_{\{\frac16\leq\epsilon<\frac12 \}} +\langle\frac{1}{k}\rangle^{\frac{1}{3}-\epsilon} \chi_{\{0<\epsilon<\frac16\}} \right)\langle k\rangle^{-m}\|_{L_k^2}\\
				& \leq \frac{2^{ -\frac{7}{12}}    }{\alpha^\frac14} \lt(\frac32\rt)^\epsilon   (1 + \Xi) \max\lt\{\sqrt{B(\frac12 - \epsilon , m -\frac12)}   , \sqrt{B(\epsilon+\frac16, m -\frac12)}  \rt\}    \\
				&\quad \times \|e^{a A^{-\frac{1}{3}}\left|D_x\right|^{\frac{2}{3} } t }\langle D_x\rangle^{m  }\langle\frac{1}{D_x}\rangle^\epsilon  |D_x|^\frac13 n  \|_{L^2}   \|e^{a A^{-\frac{1}{3}}\left|D_x\right|^{\frac{2}{3} } t }\langle D_x\rangle^{m  }\langle\frac{1}{D_x}\rangle^\epsilon |D_x|^\frac23 n\|_{L^2}\\
                &\quad\times \|e^{a A^{- \frac{1}{3}}\left|D_x\right|^{\frac{2}{3} } t }\langle D_x\rangle^m\langle\frac{1}{D_x}\rangle^\epsilon \nabla  n \|_{L^2}
		\end{align*}
		provided that   $0<\epsilon < 1/2 <m$.
		Adding all the estimates above together, we obtain
		\begin{equation*}\begin{aligned}
				&\quad\left|\Re\left(|D_x|^\frac13 \nabla \cdot(n\nabla c) \left\lvert\, \mathcal{M} e^{2 a A^{-\frac{1}{3}}\left|D_x\right|^{\frac{2}{3} }t }\langle D_x\rangle^{2 m  }\langle\frac{1}{D_x}\rangle^{2 \epsilon} |D_x|^\frac13 n\right.\right)\right| \\
				&\leq (2\pi)^{-\frac12} \frac {2^{m }  }{\alpha^\frac14} \lt(\frac32\rt)^\epsilon   (1 + \Xi)  \max\lt\{\sqrt{B(\frac12 - \epsilon , m -\frac12)},   \sqrt{B(\epsilon , m -\frac12) } \rt\} \\\
                &\quad\times \lt( 2^{\frac34}\sqrt{\frac43}\|e^{a A^{-\frac{1}{3}}\left|D_x\right|^{\frac{2}{3} } t }\langle D_x\rangle^{m  }\langle\frac{1}{D_x}\rangle^\epsilon |D_x|^\frac23  n \|_{L^2}^\frac32 \|e^{a A^{-\frac{1}{3}}\left|D_x\right|^{\frac{2}{3} } t }\langle D_x\rangle^{m  }\langle\frac{1}{D_x}\rangle^\epsilon |D_x|^\frac13 n\|_{L^2} \rt.   \\
				&\lt.  \qquad \quad\times  \|e^{a A^{-\frac{1}{3}}\left|D_x\right|^{\frac{2}{3} } t }\langle D_x\rangle^{m }\langle\frac{1}{D_x}\rangle^\epsilon \partial_x |D_x|^\frac13 n\|_{L^2}^\frac12 \rt. \\
				&\lt. \qquad  +  2^\frac1{12}   \|e^{a A^{-\frac{1}{3}}\left|D_x\right|^{\frac{2}{3} } t }\langle D_x\rangle^{m  }\langle\frac{1}{D_x}\rangle^\epsilon \nabla |D_x|^\frac13 n \|_{L^2}  \|e^{a A^{-\frac{1}{3}}\left|D_x\right|^{\frac{2}{3} } t }\langle D_x\rangle^{m  }\langle\frac{1}{D_x}\rangle^\epsilon |D_x|^\frac23  n\|_{L^2} \rt. \\
				& \lt. \qqquad \times \|e^{a A^{-\frac{1}{3}}\left|D_x\right|^{\frac{2}{3} } t }\langle D_x\rangle^{m }\langle\frac{1}{D_x}\rangle^\epsilon   n\|_{L^2}   \rt.\\
				&\lt. \qquad +2^{-\frac{13}{12}}\|e^{a A^{-\frac{1}{3}}\left|D_x\right|^{\frac{2}{3} } t }\langle D_x\rangle^{m  }\langle\frac{1}{D_x}\rangle^\epsilon \nabla  n \|_{L^2} \|e^{a A^{-\frac{1}{3}}\left|D_x\right|^{\frac{2}{3} } t }\langle D_x\rangle^{m }\langle\frac{1}{D_x}\rangle^\epsilon   |D_x|^\frac13 n\|_{L^2} \rt.\\
				&\lt. \qqquad \times\|e^{a A^{-\frac{1}{3}}\left|D_x\right|^{\frac{2}{3} } t }\langle D_x\rangle^{m }\langle\frac{1}{D_x}\rangle^\epsilon   |D_x|^\frac23 n\|_{L^2}\rt. \\
				&\lt. \qquad + 2^\frac1{12} \|e^{a A^{-\frac{1}{3}}\left|D_x\right|^{\frac{2}{3} } t }\langle D_x\rangle^m\langle\frac{1}{D_x}\rangle^\epsilon |D_x|^\frac13  n\|_{L^2}^2\|e^{a A^{-\frac{1}{3}}\left|D_x\right|^{\frac{2}{3} } t }\langle D_x\rangle^{m  }\langle\frac{1}{D_x}\rangle^\epsilon \nabla |D_x|^\frac13 n \|_{L^2} \rt) , 
		\end{aligned}\end{equation*}
		which implies (by noting that $\lambda_1^2 \lambda_2 \leq \frac{2}{3\sqrt{3}} (\lambda_1^2 + \lambda_2^2)^\frac32$ for all $\lambda_1,\lambda_2>0$)
		\begin{equation}\begin{aligned} \label{eq:temp0.135}
				&\quad \frac1A \int_{0}^{t} \left|\Re\left(\nabla \cdot(n\nabla c) \left\lvert\, \mathcal{M} e^{2 a A^{-\frac{1}{3}}\left|D_x\right|^{\frac{2}{3} }t }\langle D_x\rangle^{2 m +\frac23}\langle\frac{1}{D_x}\rangle^{2 \epsilon} n\right.\right)\right| dt \\
				&\leq   \frac{ 2^{m } } {\alpha^\frac14 } \lt(\frac32\rt)^\epsilon \max\lt\{\sqrt{B(\frac12 - \epsilon , m -\frac12)}   , \sqrt{B(\epsilon , m -\frac12)}  \rt\} \lt[ \frac{C_{fl1}}{A^\frac12} \| \langle D_x\rangle^{m  }  \langle \frac{1}{D_x}\rangle^\epsilon  |D_x|^\frac13  n \|_{X}^3\rt.\\
				&\lt.\qquad  +  \frac{C_{fl2}}{A^\frac13}\lt( \| \langle D_x\rangle^{m  }  \langle \frac{1}{D_x}\rangle^\epsilon    n \|_{X}^2 + \| \langle D_x\rangle^{m  }  \langle \frac{1}{D_x}\rangle^\epsilon |D_x|^\frac13   n \|_{X}^2 \rt)^\frac32 \rt] ,
		\end{aligned}\end{equation}
		where 
		\begin{equation*}\begin{aligned}
				C_{fl1} &:= (2\pi)^{-\frac12}   (1+ \Xi)\cdot 2^{\frac34}\sqrt{\frac43} (\frac{\theta_1\Xi}{2\pi} - 2a)^{-\frac34} \cdot 2^{-\frac14}  ,\\
				C_{fl2} &:= (2\pi)^{-\frac12}   (1+ \Xi)\cdot \lt( \frac{2^{\frac{25}{12}} + 2^{-\frac1{12}}}{3\sqrt{3}} \rt)  (\frac{\theta_1\Xi}{2\pi} - 2a)^{-\frac12}   (2- \frac{\theta_2\Xi}{2\pi})^{-\frac12}  .
		\end{aligned}\end{equation*}
		Therefore, \eqref{eq:temp0.33} follows from \eqref{eq:nabla n nabla c}, \eqref{eq:temp001}--\eqref{eq:temp004} and \eqref{eq:temp0.135}.

		\underline{\bf Step II: The proof of \eqref{eq:temp0.3332}  for  the $\alpha = 0$ case.}

		Since the computations follow same arguments as the second step in the proof of Lemma \ref{lem:est of n'}, we omit it and present the final result:
		\ben\label{eq:alpha0 bound2}
				&&\quad \frac1A \int_{0}^{t} \left|\Re\left(\nabla \cdot(n\nabla c) \left\lvert\, \mathcal{M} e^{2 a A^{-\frac{1}{3}}\left|D_x\right|^{\frac{2}{3} }t }\langle D_x\rangle^{2 m +\frac23}\langle\frac{1}{D_x}\rangle^{2 \epsilon} n\right.\right)\right| dt \nonumber\\
				&\leq&   2^{\frac14} \cdot { 2^{m } }  \lt(\frac32\rt)^\epsilon \max\lt\{\sqrt{B(\frac12 - \epsilon , m -\frac12)}   , \sqrt{B(\epsilon -\frac13 , m -\frac12)}  \rt\} \lt[ \frac{C_{fl1}}{A^\frac12} \| \langle D_x\rangle^{m  }  \langle \frac{1}{D_x}\rangle^\epsilon  |D_x|^\frac13  n \|_{X}^3\rt.\nonumber\\
				&&\lt.\qquad  +  \frac{C_{fl3}}{A^\frac13}\lt( \| \langle D_x\rangle^{m  }  \langle \frac{1}{D_x}\rangle^\epsilon    n \|_{X}^2 + \| \langle D_x\rangle^{m  }  \langle \frac{1}{D_x}\rangle^\epsilon |D_x|^\frac13   n \|_{X}^2 \rt)^\frac32 \rt] ,
		\een
        where
		\begin{equation*}\begin{aligned}
				C_{fl3} = (2\pi)^{-\frac12}  (1+ \Xi)\cdot \lt( \frac{2^{\frac{25}{12}} + 2^{\frac5{12}}}{3\sqrt{3}} \rt)  (\frac{\theta_1\Xi}{2\pi} - 2a)^{-\frac12}  (2- \frac{\theta_2\Xi}{2\pi})^{-\frac12}  .
		\end{aligned}\end{equation*}
		The proof is complete.
	\end{proof}

    Then, we give the estimate of $\omega$.
	
	\begin{Lem} \label{lem:est of omega}
		Let $\theta_1$, $\theta_2$, $\Xi$ and $a$ satisfy the positivity conditions \eqref{eq:positivity conditions}. For any $0<\epsilon<1/2<m$ and $0\leq t \leq T$, it holds that 
		\begin{equation*}\begin{aligned}
				&\quad\|  \langle D_x\rangle^m  \langle \frac{1}{D_x}\rangle^\epsilon   \omega  \|_{X}^2 \\
				&\leq   \|  \langle D_x\rangle^m  \langle \frac{1}{D_x}\rangle^\epsilon   \omega_{\mathrm{in}} \|_{L^2}^2   +  \frac{2 C_l}{ A^\frac23}   \|  \langle D_x\rangle^m  \langle \frac{1}{D_x}\rangle^\epsilon   \omega \|_{X}   \|  \langle D_x\rangle^{m }  \langle \frac{1}{D_x}\rangle^\epsilon |D_x|^\frac13  n \|_{X}  \\
				&\quad+   \frac{2C_{st}\cdot2^{  m}  }{A^\frac12}   \lt(\frac32\rt)^\epsilon \max\lt\{ \sqrt{B(\frac12 - \epsilon, m -\frac12)}, \sqrt{B(\epsilon, m -\frac12)}\rt\}   \|  \langle D_x\rangle^m  \langle \frac{1}{D_x}\rangle^\epsilon   \omega \|_{X}^3,
		\end{aligned}\end{equation*}
		where
		\begin{equation*}\begin{aligned}
				C_l = (1 + \Xi)  \left(\frac{\theta_1 \Xi}{2\pi} -2a\right)^{-1} .
		\end{aligned}\end{equation*}
	\end{Lem}
	\begin{proof}
		Applying Proposition \ref{lem:est of f} to $\eqref{eq:main}_3$, we obtain
		\begin{equation}\begin{aligned} \label{eq:est of omega1}
				&\quad\|  \langle D_x\rangle^m  \langle \frac{1}{D_x}\rangle^\epsilon   \omega  \|_{X}^2 	\\
				&\leq   \|  \langle D_x\rangle^m  \langle \frac{1}{D_x}\rangle^\epsilon   \omega_{\mathrm{in}} \|_{L^2}^2   \\
				&\quad + \frac{2C_{st}\cdot 2^{  m}    }{A^\frac12}  \lt(\frac32\rt)^\epsilon \max\lt\{ \sqrt{B(\frac12 - \epsilon, m -\frac12)} , \sqrt{B(\epsilon, m -\frac12)}\rt\}   \|  \langle D_x\rangle^m  \langle \frac{1}{D_x}\rangle^\epsilon   \omega \|_{X}^3   \\
				&\quad +   \frac{2}{A}\int_{0}^{t} \Re\left(\px n \left\lvert\, \mathcal{M} e^{2 a A^{-\frac{1}{3}}\left|D_x\right|^{\frac{2}{3} }t }\langle D_x\rangle^{2 m}\langle\frac{1}{D_x}\rangle^{2 \epsilon} \omega\right.\right) dt .
		\end{aligned}\end{equation}
		Moreover, it holds that
		\begin{equation}\begin{aligned} \label{eq:px n omega}
				&\frac2A \int_{0}^{t}  \left|\Re\left(\px n \left\lvert\, \mathcal{M} e^{2 a A^{-\frac{1}{3}}\left|D_x\right|^{\frac{2}{3} }t }\langle D_x\rangle^{2 m}\langle\frac{1}{D_x}\rangle^{2 \epsilon} \omega\right.\right)\right|  dt \\
				\leq & \frac{2 }A \cdot (1 + \Xi) \|e^{a A^{-\frac{1}{3}}\left|D_x\right|^{\frac{2}{3} } t }  \langle D_x\rangle^m\langle\frac{1}{D_x}\rangle^\epsilon |D_x|^\frac13 \omega\|_{L^2 L^2}  \|e^{a A^{-\frac{1}{3}}\left|D_x\right|^{\frac{2}{3} } t }  \langle D_x\rangle^m\langle\frac{1}{D_x}\rangle^\epsilon |D_x|^\frac23 n\|_{L^2 L^2} \\
				\leq & \frac{2 C_l }{ A^\frac23}   \|  \langle D_x\rangle^m  \langle \frac{1}{D_x}\rangle^\epsilon   \omega \|_{X}   \|  \langle D_x\rangle^{m }  \langle \frac{1}{D_x}\rangle^\epsilon |D_x|^\frac13  n \|_{X},
		\end{aligned}\end{equation}
		where $C_l := (1 + \Xi) \left(\frac{\theta_1 \Xi}{2\pi} -2a\right)^{-1} $.
		Combining \eqref{eq:px n omega} with \eqref{eq:est of omega1}, we complete the proof.
	\end{proof}

    \subsection{Closing the energy}
	Based on the established estimates for $n$ and $\omega$, one now chooses $A$ sufficiently large so that the energy can be closed as follows.
	\begin{proof}[Proof of Proposition \ref{main prop}]
		 The proof is divided into two cases.

        \underline{\bf Case of $\alpha>0$.}
		By \eqref{eq:bootstap}, Lemmas \ref{lem:est of n} and \ref{lem:est of omega}, we get for $0<t\leq T$ that
		\begin{align*} 
				&\quad E(t)^2\\
				&\leq \|  \langle D_x\rangle^m  \langle \frac{1}{D_x}\rangle^\epsilon   \omega_{\mathrm{in}} \|_{L^2}^2 + \|  \langle D_x\rangle^{m }  \langle \frac{1}{D_x}\rangle^\epsilon   n_{\mathrm{in}} \|_{L^2}^2  +   \|  \langle D_x\rangle^{m }  \langle \frac{1}{D_x}\rangle^\epsilon  |D_x|^\frac13  n_{\mathrm{in}} \|_{L^2}^2  \\
				&\quad + \frac{2 C_l}{ A^\frac23}   \|  \langle D_x\rangle^m  \langle \frac{1}{D_x}\rangle^\epsilon   \omega \|_{X}   \|  \langle D_x\rangle^{m }  \langle \frac{1}{D_x}\rangle^\epsilon |D_x|^\frac13  n \|_{X}  \\
				&\quad+   \frac{2C_{st}\cdot2^{  m}  }{A^\frac12}  \lt(\frac32\rt)^\epsilon \max\lt\{ \sqrt{B(\frac12 - \epsilon, m -\frac12)} , \sqrt{B(\epsilon, m -\frac12)}\rt\}   \|  \langle D_x\rangle^m  \langle \frac{1}{D_x}\rangle^\epsilon   \omega \|_{X}^3\\
				&\quad + \frac{   2^{  m}  }{A^\frac12}  \lt(\frac32\rt)^\epsilon \max\lt\{ \sqrt{B(\frac12 - \epsilon, m -\frac12)} , \sqrt{B(\epsilon, m -\frac12)}\rt\}\|  \langle D_x\rangle^m  \langle \frac{1}{D_x}\rangle^\epsilon   \omega \|_{X}  \\
				&\qquad\times \left[ 2C_{st}   \|  \langle D_x\rangle^m  \langle \frac{1}{D_x}\rangle^\epsilon      n \|_{X}^2  + (2C_{st}  + 2C_{hl}) \|  \langle D_x\rangle^m  \langle \frac{1}{D_x}\rangle^\epsilon  |D_x|^\frac13    n \|_{X}^2 \right]  \\
				&\quad + 2\big[\frac {C_{ch1}}{A^\frac12 } +\frac{C_{ch2}}{A^\frac13} \big] \frac{ 2^{m } } {\alpha^\frac14 } \lt(\frac32\rt)^\epsilon \max\lt\{\sqrt{B(\frac12 - \epsilon , m -\frac12)}   , \sqrt{B(\epsilon , m -\frac12)} \rt\}  \| \langle D_x\rangle^{m   }  \langle \frac{1}{D_x}\rangle^\epsilon   n \|_{X}^3 \\
				&\quad +  2\frac{ 2^{m } } {\alpha^\frac14 } \lt(\frac32\rt)^\epsilon \max\lt\{B(\frac12 - \epsilon , m -\frac12)  , B(\epsilon , m -\frac12)  \rt\} \lt[ \frac{C_{fl1}}{A^\frac12} \| \langle D_x\rangle^{m  }  \langle \frac{1}{D_x}\rangle^\epsilon  |D_x|^\frac13  n \|_{X}^3\rt.\\
				&\lt.\qquad  +  \frac{C_{fl2}}{A^\frac13}\lt( \| \langle D_x\rangle^{m  }  \langle \frac{1}{D_x}\rangle^\epsilon    n \|_{X}^2 + \| \langle D_x\rangle^{m  }  \langle \frac{1}{D_x}\rangle^\epsilon |D_x|^\frac13   n \|_{X}^2 \rt)^\frac32 \rt]\\
				&\leq \|  \langle D_x\rangle^m  \langle \frac{1}{D_x}\rangle^\epsilon   \omega_{\mathrm{in}} \|_{L^2}^2 + \|  \langle D_x\rangle^{m }  \langle \frac{1}{D_x}\rangle^\epsilon   n_{\mathrm{in}} \|_{L^2}^2  +   \|  \langle D_x\rangle^{m }  \langle \frac{1}{D_x}\rangle^\epsilon  |D_x|^\frac13  n_{\mathrm{in}} \|_{L^2}^2  \\
				&\quad + 2\cdot 1.01^3\cdot r_3 \left[ \frac{  C_l  K^2}{1.01A^\frac23}  +  \frac{(2C_{st} + C_{hl}  +  C_{ch1} + C_{fl1} )K^3}{A^\frac12}  + \frac{(C_{ch2} + C_{fl2})K^3}{A^\frac13}\right], \alabel{eq:temp1}
		\end{align*}
		where $$r_3(\epsilon,m,\alpha):=  2^{  m}  \lt(\frac32\rt)^\epsilon \max\{1, \alpha^{-\frac14}\} \max\lt\{ \sqrt{B(\frac12 - \epsilon, m -\frac12)} , \sqrt{B(\epsilon, m -\frac12)}\rt\}.$$
		Choose large enough $
		\bar{A}_3 = \bar{\Lambda}_1(\theta_1, \theta_2, \Xi, a, r_3) K^9 $ such that if $A> \bar{A}_3$,
		\begin{equation}\begin{aligned} \label{eq:yz1}
				2\cdot 1.01^3\cdot r_3 \left[ \frac{  C_l  K^2}{1.01A^\frac23}  +  \frac{(2C_{st} + C_{hl}  +  C_{ch1} + C_{fl1} )K^3}{A^\frac12}  + \frac{(C_{ch2} + C_{fl2})K^3}{A^\frac13}\right] \leq 1.
		\end{aligned}\end{equation}


		Now we denote $$K:=  \sqrt{ \|  \langle D_x\rangle^m  \langle \frac{1}{D_x}\rangle^\epsilon   \omega_{\mathrm{in}} \|_{L^2}^2  +  \|  \langle D_x\rangle^{m  }  \langle \frac{1}{D_x}\rangle^\epsilon   n_{\mathrm{in}} \|_{L^2}^2  + \|  \langle D_x\rangle^{m  }  \langle \frac{1}{D_x}\rangle^\epsilon   |D_x|^\frac13 n_{\mathrm{in}} \|_{L^2}^2 +1 }.$$  Then, by Lemma \ref{lem:n L infty L infty}, Lemma \ref{lem:Elliptic estimate} and \eqref{eq:temp1}, we have for $0<t\leq T$ that 
		\begin{equation*}\begin{aligned}
				E_\infty (t) & \leq 2^7  \lt[\lt(\frac1{2\pi \alpha}\rt)^\frac12 \|n\|_{ L_t^\infty L^2}^2+1\rt] \left( \|n\|_{ L_t^\infty L^2} +  M + \| n_{\mathrm{in}} \|_{L^\infty}  +1 \right) \\
				& \leq  2^7  \lt[\lt(\frac1{2\pi \alpha}\rt)^\frac12\|n\|_{X}^2+1\rt] \left( \|n\|_{X} +  M + \| n_{\mathrm{in}} \|_{L^\infty}  +1 \right) \\
				& \leq 2^7  \lt[\lt(\frac1{2\pi\alpha}\rt)^\frac12 \cdot 1.01^2\cdot K^2+1\rt]				
				\left(1.01 K +  M + \| n_{\mathrm{in}} \|_{L^\infty}  +1 \right) =: K_{\infty}.
		\end{aligned}\end{equation*}

		\underline{\bf Case of $\alpha = 0$.} One rewrites \eqref{eq:temp1} as
		\begin{equation*}\begin{aligned} 
				&\quad E(t)^2 \\ 
				&\leq \|  \langle D_x\rangle^m  \langle \frac{1}{D_x}\rangle^\epsilon   \omega_{\mathrm{in}} \|_{L^2}^2 + \|  \langle D_x\rangle^{m }  \langle \frac{1}{D_x}\rangle^\epsilon   n_{\mathrm{in}} \|_{L^2}^2  +   \|  \langle D_x\rangle^{m }  \langle \frac{1}{D_x}\rangle^\epsilon  |D_x|^\frac13  n_{\mathrm{in}} \|_{L^2}^2  \\
				&\quad + 2\cdot 1.01^3\cdot r_2 \left[ \frac{  C_l  K^2}{1.01A^\frac23}  +  \frac{(2C_{st} + C_{hl}  +  2^\frac14 C_{ch1} + 2^\frac14 C_{fl1} )K^3}{A^\frac12}  + \frac{(2^\frac14 C_{ch3} + 2^\frac14  C_{fl3})K^3}{A^\frac13}\right].
		\end{aligned}\end{equation*}
		Choose large enough $
		\bar{A}_4 = \bar{\Lambda}_2(\theta_1, \theta_2, \Xi, a, r_2) K^9 $ such that if $A> \bar{A}_4$,
		\begin{equation}\begin{aligned} \label{eq:yz2}
				2\cdot 1.01^3\cdot r_2 \left[ \frac{  C_l  K^2}{1.01A^\frac23}  +  \frac{(2C_{st} + C_{hl}  +  2^\frac14 C_{ch1} + 2^\frac14 C_{fl1} )K^3}{A^\frac12}  + \frac{(2^\frac14 C_{ch3} + 2^\frac14 C_{fl3})K^3}{A^\frac13}\right] \leq 1.
		\end{aligned}\end{equation}
        

		Therefore, denote $$K:=   \sqrt{ \|  \langle D_x\rangle^m  \langle \frac{1}{D_x}\rangle^\epsilon   \omega_{\mathrm{in}} \|_{L^2}^2  +  \|  \langle D_x\rangle^{m  }  \langle \frac{1}{D_x}\rangle^\epsilon   n_{\mathrm{in}} \|_{L^2}^2 + \|  \langle D_x\rangle^{m  }  \langle \frac{1}{D_x}\rangle^\epsilon |D_x|^\frac13  n_{\mathrm{in}} \|_{L^2}^2 +1  }.$$  Then, by \eqref{eq:temp1}, Lemmas \ref{lem:Elliptic estimate} and \ref{lem:n L infty L infty}, we have for $0<t\leq T$ that 
		\begin{equation*}\begin{aligned}
				E_\infty (t) & \leq 2^7  \lt[ \frac4{\pi^2} (3 + 2\sqrt{2})^2 ( \|n\|_{ L_t^\infty L^2} + M )^2   +1\rt] \left( \|n\|_{ L_t^\infty L^2} +  M + \| n_{\mathrm{in}} \|_{L^\infty}  +1 \right) \\
				& \leq 2^7  \lt[ \frac4{\pi^2} (3 + 2\sqrt{2})^2 ( 1.01K + M )^2   +1\rt] \left( 1.01 K +  M + \| n_{\mathrm{in}} \|_{L^\infty}  +1 \right) =: K_{\infty}.
		\end{aligned}\end{equation*}
		Combining the two cases and denoting $\bar{\Lambda} :=\max\lt\{\bar{\Lambda}_1, \bar{\Lambda}_2 \rt\}$, the proof is complete.
	\end{proof}

	\appendix 
	\section{Some useful lemmas}
	\label{appA}
	
	First, we recall some sharp estimates for the Riesz operators.
	\begin{Lem}[Theorem 1.1 in \cite{GT1996}] \label{lem:R}
		For each $1<p<\infty$ and $\ell=1,2, \ldots, n$ we have
		$$
		\left\|\mathbf{R}_{\ell}\right\|_p = \left\{\begin{array}{lll}
			\tan (\pi / 2 p) & \text { if } & 1<p \leq 2, \\
			\cot (\pi / 2 p) & \text { if } & 2 \leq p<\infty,
		\end{array}\right.
		$$
		where
		$$
		\mathbf{R}_{\ell} f(x) :=c_n \int_{\mathbb{R}^n} \frac{\left(x_{\ell}-y_{\ell}\right) f(y)}{|x-y|^{n+1}} d y \quad\text{with}\quad c_n=\pi^{-\frac{n+1}{2}} \Gamma\left(\frac{n+1}{2}\right) 
		$$
		for certain smooth enough $f$ and
		$$
		\left\|\mathbf{R}_{\ell}\right\|_p
		= \sup_{f \ne 0} \frac{\|\mathbf{R}_{\ell} f\|_{L^p(\mathbb{R}^n)}}{\|f\|_{L^p(\mathbb{R}^n)}} .
		$$
	\end{Lem}

        Second, we introduce the following embedding inequalities.
	\begin{Lem}
		For $f_l= f_l(y) \in H^1(\mathbb{R})$ (defined in \eqref{eq:def of fk}), it holds that
		\begin{align}
			\| f_l(\cdot)\|_{L^\infty} &\leq   \|   f_l(\cdot)\|_{L^2}^\frac12  \|   \py f_l(\cdot)\|_{L^2}^\frac12, \label{eq:GN1}\\
			\| f_l(\cdot)\|_{L^\infty} &\leq  |l|^{-\frac12} \| \nabla_l f_l(\cdot)\|_{L^2}. \label{eq:GN2}
		\end{align}
	\end{Lem}
	\begin{proof}
		The inequality \eqref{eq:GN1} follows from 1D Gagliardo--Nirenberg inequality in \cite{LW2024} and then by Young's inequality \eqref{eq:GN2} holds immediately.
	\end{proof}
	
	Finally, we give a direct estimate for nonlinear terms of the form $|D_x|^{1/3}u \cdot \nabla f$, which appears in the proof of Lemma \ref{lem:est of n}.

	\begin{Lem} \label{lem:est of dx u}
		Let $u$ be determined by $\eqref{eq:main}_4$ and $\theta_1$, $\theta_2$, $\Xi$ and $a$ satisfy the positivity conditions \eqref{eq:positivity conditions}. For $0<\epsilon<1/2<m$, $0\leq t\leq T$ and certain smooth enough $f$, it holds that
		\begin{equation*}\begin{aligned}    
				&\quad\frac1A \int_{0}^{t} \left|\Re\left(|D_x|^{\frac13}u \cdot \nabla f \left\lvert\, \mathcal{M} e^{2 a A^{-\frac{1}{3}}\left|D_x\right|^{\frac{2}{3} }t }\langle D_x\rangle^{2 m   }\langle\frac{1}{D_x}\rangle^{2 \epsilon} |D_x|^\frac13  f  \right.\right)\right|   dt\\
				&\leq  C_{hl}\frac{  2^{m } }{A^\frac12}\lt(\frac32\rt)^{\epsilon}\max\lt\{ \sqrt{B(\frac12 - \epsilon, m -\frac12)}, \sqrt{B(\epsilon, m -\frac12)}\rt\}  \\
				&\qquad\times \| \langle D_x\rangle^m\langle\frac{1}{D_x}\rangle^\epsilon \omega \|_{X}    \| \langle D_x\rangle^m\langle\frac{1}{D_x}\rangle^\epsilon |D_x|^\frac13 f \|_{X}^2   ,
		\end{aligned}\end{equation*}
		where
		\begin{equation*}\begin{aligned}
				C_{hl}  &= (2\pi)^{-\frac12}  (1 + \Xi) \lt[3^\frac13\sqrt{\frac43}   \lt( \frac{\theta_1\Xi}{2\pi}  - 2a\rt)^{-\frac34} \cdot 2^{-\frac14} \rt.\\
                            &\lt. \qquad\qquad+  \lt(2^{\frac13}+2^{\frac16} + 3^{\frac13}\sqrt{\frac43}+1+2^{-\frac13} \rt) \left(\frac{\Xi}{2\pi}\right)^{-\frac12} \left(2 - \frac{\theta_2\Xi}{2\pi}\right)^{-\frac12}  \rt] .
		\end{aligned}\end{equation*}
	\end{Lem}
	\begin{proof}
		Due to 
        $$
        |D_x|^{\frac13}u \cdot \nabla f = |D_x|^{\frac13}u_1 \px   f + |D_x|^{\frac13}u_2 \py  f,
        $$
		the following estimate is divided into two parts.

		\underline{\bf Estimate of $ |D_x|^{1/3} u_1 \px f$.}
		Fourier transform yields
		\begin{equation}\begin{aligned}   \label{eq:est of u1 px f 01'}
				&\quad(2\pi)^{\frac12} \left|\Re\left( |D_x|^\frac13 u_1 \px  f \left\lvert\, \mathcal{M} e^{2 a A^{-\frac{1}{3}}\left|D_x\right|^{\frac{2}{3} }t }\langle D_x\rangle^{2 m   }\langle\frac{1}{D_x}\rangle^{2 \epsilon} |D_x|^\frac13 f\right.\right)\right| \\
				&\leq  \int_{\mathbb{R}^2} e^{2 a A^{-\frac{1}{3}}|k|^{\frac{2}{3} } t }\langle k\rangle^{2 m   }\langle\frac{1}{k}\rangle^{2 \epsilon}   |k|^\frac13  |l|^\frac13 \left|  \int_{\mathbb{R}} \mathcal{M}\left(k, D_y\right) f_k(y) \partial_y \Delta_l^{-1} \omega_l(y)  (k-l) f_{k-l}(y) d y \right|  d k d l.
		\end{aligned}\end{equation}
		$\bullet$~Case 1:   $\frac{|k-l|}{2} \leq |k| \leq 2|k-l|$.
		Similarly to \eqref{eq:est1 of u1 px omega}, it follows from \eqref{eq:trick7} and $|l|^\frac13 \leq 3^{\frac{1}{3}} |k|^{\frac{1}{3}}   $ that
		\begin{align}  \label{eq:est of u1 px c 11'}
				&\no\quad \int_{\frac{|k-l|}{2} \leq|k| \leq 2|k-l|} e^{2 a A^{-\frac{1}{3}}|k|^{\frac{2}{3} }t }\langle k\rangle^{2 m   }\langle\frac{1}{k}\rangle^{2 \epsilon}  |k|^\frac13 |l|^\frac13 \\
				&\no  \qquad  \qquad  \times \left|  \int_{\mathbb{R}} \mathcal{M}\left(k, D_y\right) f_k (y) \partial_y \Delta^{-1}_l \omega_l(y)  (k-l) f_{k-l}(y) d y \right| d k d l \\
				&\no \leq 3^{\frac13}\cdot2^{m + \epsilon } \cdot(1 + \Xi) \|e^{a A^{-\frac{1}{3}}|k|^{\frac{2}{3} } t }\langle k\rangle^{m  }\langle\frac{1}{k}\rangle^\epsilon|k|^\frac23 \| f_k\|_{L_y^2}\|_{L_k^2}     \|e^{a A^{-\frac{1}{3}}|l|^{\frac{2}{3} } t }\| \partial_y \Delta^{-1}_l \omega_l \|_{L_y^{\infty}}\|_{L_l^1}   \\
				&\no\quad  \times  \|e^{a A^{-\frac{1}{3}}|k-l|^{\frac{2}{3} }t }\langle k-l\rangle^{m   }\langle\frac{1}{k-l}\rangle^\epsilon  |k-l|  \| f_{k-l}\|_{L_y^2}\|_{L_{k-l}^2}\\
				&\no \leq 3^{\frac13} \cdot2^{m + \epsilon}\cdot (1 + \Xi) \sqrt{ B(\epsilon, m)}      \|e^{a A^{-\frac{1}{3}}\left|D_x\right|^{\frac{2}{3} } t }\langle D_x\rangle^{m  }\langle\frac{1}{D_x}\rangle^\epsilon\left|D_x\right|^{\frac{2}{3}}   f  \|_{L^2}^\frac32          \\
				&\qquad \times  \|e^{a A^{-\frac{1}{3}}\left|D_x\right|^{\frac{2}{3} } t }\langle D_x\rangle^{m  }\langle\frac{1}{D_x}\rangle^\epsilon \px |D_x|^\frac13    f\|_{L^2} ^\frac12 \|e^{a A^{-\frac{1}{3}}\left|D_x\right|^{\frac{2}{3} } t }\langle D_x\rangle^m\langle\frac{1}{D_x}\rangle^\epsilon   \omega\|_{L^2}.
		\end{align}
		
		$\bullet$~Case 2:   $2|k-l|<|k|$.
		Using \eqref{eq:trick8}, $|k-l|^\frac23 |k|^\frac13 \leq 2^{\frac13}|l|$ and following a similar estimate for \eqref{eq:est2 of u1 px omega}, 
		one deduces that
		\begin{align*}   
				& \quad\int_{2|k-l|<|k|} e^{2 a A^{-\frac{1}{3}}|k|^{\frac{2}{3} }t }\langle k\rangle^{2 m  }\langle\frac{1}{k}\rangle^{2 \epsilon} |k|^\frac13  |l|^\frac13  \\
				&  \qquad  \qquad  \times \left| \int_{\mathbb{R}} \mathcal{M}\left(k, D_y\right) f_k(y)  \partial_y \Delta^{-1}_l \omega_l(y)  (k-l)  f_{k-l}(y) d y  \right| d k d l \\
				& \leq 2^{m +\frac13  }\lt(\frac32\rt)^{\epsilon}(1 + \Xi) \|e^{a A^{-\frac{1}{3}}|k|^{\frac{2}{3} } t }\langle k\rangle^{m   }\langle\frac{1}{k}\rangle^\epsilon   |k|^\frac13  \| f_k\|_{L_y^2}\|_{L_k^2}           \|e^{a A^{-\frac{1}{3}}|l|^{\frac{2}{3} } t }\langle l\rangle^{m   }\langle\frac{1}{l}\rangle^\epsilon |l| \| \partial_y \Delta^{-1}_l \omega_l \|_{L_y^2}\|_{L_l^2} \\
				& \qquad \times \|e^{a A^{-\frac{1}{3}}|k-l|^{\frac{2}{3} } t }|k-l|^\frac13 \| f_{k-l}\|_{L_y^{\infty}}\|_{L_{k-l}^1} \\
				& \leq  2^{m +\frac13 } \lt(\frac32\rt)^{\epsilon}(1 + \Xi) \sqrt{ B(\epsilon, m)} \|e^{a A^{-\frac{1}{3}}\left|D_x\right|^{\frac{2}{3} } t }\langle D_x\rangle^{m   }\langle\frac{1}{D_x}\rangle^\epsilon |D_x|^\frac13  f \|_{L^2}     \\
				&\qquad \times \|e^{a A^{-\frac{1}{3}}\left|D_x\right|^{\frac{2}{3} } t }\langle D_x\rangle^{m   }\langle\frac{1}{D_x}\rangle^\epsilon \px  \py \Delta^{-1} \omega \|_{L^2}\|e^{a A^{-\frac{1}{3}}\left|D_x\right|^{\frac{2}{3} } t }\langle D_x\rangle^m\langle\frac{1}{D_x}\rangle^\epsilon \nabla  |D_x|^\frac13     f \|_{L^2}.  \alabel{eq:est of u1 px c 21'}
		\end{align*}

		$\bullet$~Case 3:   $2|k|<|k-l|$.
		From \eqref{eq:trick9}, \eqref{eq:trick4} and $|k-l|^{\frac23}\leq 2^\frac{2}{3} |l|^{\frac{2}{3}}$, we obtain
		\begin{align*}   
				& \quad\int_{2|k|<|k-l|} e^{2 a A^{-\frac{1}{3}}|k|^{\frac{2}{3} } t }\langle k\rangle^{2 m   }\langle\frac{1}{k}\rangle^{2 \epsilon}  |k|^\frac13 |l|^\frac13   \\
				&  \qquad \qquad   \qquad  \times \left|  \int_{\mathbb{R}} \mathcal{M}\left(k, D_y\right) f_k(y)  \partial_y \Delta^{-1}_l \omega_l(y) (k-l) f_{k-l}(y) d y \right| d k d l \\
				& \leq   2^{\frac23}\lt(\frac32\rt)^{\epsilon} (1 + \Xi) \|e^{a A^{-\frac{1}{3}}|k|^{\frac{2}{3} } t }\langle k\rangle^m\langle\frac{1}{k}\rangle^\epsilon   |k|^\frac13   \| f_k\|_{L_y^2}\|_{L_k^2}      \|e^{a A^{-\frac{1}{3}}|l|^{\frac{2}{3} } t }\langle l\rangle^{m   }\langle\frac{1}{l}\rangle^\epsilon|l|\| \partial_y \Delta^{-1}_l \omega_l\|_{L_y^2}\|_{L_l^2} \\
				& \quad \times \|e^{a A^{-\frac{1}{3}}|k-l|^{\frac{2}{3} } t }\langle k-l\rangle^{m   }\langle\frac{1}{k-l}\rangle^\epsilon  |k-l|^{\frac{5}{6}}  \| f_{k-l}\|_{L_y^{\infty}}\|_{L_{k-l}^2} \\
                &\quad \times \|\left(\langle\frac{1}{k}\rangle^\epsilon \chi_{\{\frac14\leq \epsilon<\frac12 \}} +\langle\frac{1}{k}\rangle^{\frac{1}{2}-\epsilon} \chi_{\{0<\epsilon<\frac14\}}\right)\langle k\rangle^{-m}\|_{L_k^2}\\
				& \leq  2^{\frac23}\lt(\frac32\rt)^{\epsilon} (1 + \Xi) \max\lt\{ \sqrt{B(\frac12 - \epsilon, m -\frac12)}, \sqrt{B(\epsilon, m -\frac12)}\rt\}  \\
				&\qquad \times  \|e^{a A^{-\frac{1}{3}}\left|D_x\right|^{\frac{2}{3} } t }\langle D_x\rangle^{m   }\langle\frac{1}{D_x}\rangle^\epsilon \px \py \Delta^{-1} \omega \|_{L^2} \|e^{a A^{-\frac{1}{3}}\left|D_x\right|^{\frac{2}{3} } t }\langle D_x\rangle^{m   }\langle\frac{1}{D_x}\rangle^\epsilon \nabla  |D_x|^\frac13  f\|_{L^2}\\
                &\qquad \times  \|e^{a A^{- \frac{1}{3}}\left|D_x\right|^{\frac{2}{3} } t }\langle D_x\rangle^m\langle\frac{1}{D_x}\rangle^\epsilon |D_x|^\frac13   f \|_{L^2} .  \alabel{eq:est of u1 px c 31'}
		\end{align*}

		Combining \eqref{eq:est of u1 px c 11'}--\eqref{eq:est of u1 px c 31'} with \eqref{eq:est of u1 px f 01'}, we deduce that
		\begin{equation}\begin{aligned} \label{eq:dx13 u1pxf}
				&\quad\left|\Re\left( |D_x|^\frac13 u_1 \px  f \left\lvert\, \mathcal{M} e^{2 a A^{-\frac{1}{3}}\left|D_x\right|^{\frac{2}{3} }t }\langle D_x\rangle^{2 m   }\langle\frac{1}{D_x}\rangle^{2 \epsilon} |D_x|^\frac13 f\right.\right)\right| \\
				&\leq (2\pi)^{-\frac12} 2^{m}\lt(\frac32\rt)^{\epsilon} (1 + \Xi) \max\lt\{ \sqrt{B(\frac12 - \epsilon, m -\frac12)}, \sqrt{B(\epsilon, m -\frac12)}\rt\}\\
                &\quad\times \lt( 3^{\frac13}\sqrt{\frac43} \|e^{a A^{-\frac{1}{3}}\left|D_x\right|^{\frac{2}{3} } t }\langle D_x\rangle^{m   }\langle\frac{1}{D_x}\rangle^\epsilon |D_x|^\frac23  f \|_{L^2} ^\frac32\|e^{a A^{-\frac{1}{3}}\left|D_x\right|^{\frac{2}{3} } t }\langle D_x\rangle^{m   }\langle\frac{1}{D_x}\rangle^\epsilon  \omega \|_{L^2}  \rt.  \\
				&\lt. \qqquad \times \|e^{a A^{-\frac{1}{3}}\left|D_x\right|^{\frac{2}{3} } t }\langle D_x\rangle^m\langle\frac{1}{D_x}\rangle^\epsilon \px  |D_x|^\frac13     f \|_{L^2}^\frac12 \rt. \\
				& \lt. \qquad + (2^{\frac13}+2^{\frac16})\|e^{a A^{- \frac{1}{3}}\left|D_x\right|^{\frac{2}{3} } t }\langle D_x\rangle^m\langle\frac{1}{D_x}\rangle^\epsilon |D_x|^\frac13   f \|_{L^2} \|e^{a A^{-\frac{1}{3}}\left|D_x\right|^{\frac{2}{3} } t }\langle D_x\rangle^{m   }\langle\frac{1}{D_x}\rangle^\epsilon \px \py \Delta^{-1} \omega \|_{L^2} \rt.   \\
				&\lt. \qqquad \times \|e^{a A^{-\frac{1}{3}}\left|D_x\right|^{\frac{2}{3} } t }\langle D_x\rangle^{m   }\langle\frac{1}{D_x}\rangle^\epsilon \nabla  |D_x|^\frac13  f\|_{L^2} \rt).
		\end{aligned}\end{equation}

		\underline{\bf Estimate of $|D_x|^{1/3} u_2 \py f$.}
		Note that
		\begin{equation}\begin{aligned}   \label{eq:est of u2 py c 01'}
				&\quad\left|\Re\left(|D_x|^{\frac13}u_2 \py  f\left\lvert\, \mathcal{M} e^{2 a A^{-\frac{1}{3}}\left|D_x\right|^{\frac{2}{3} }t }\langle D_x\rangle^{2 m   }\langle\frac{1}{D_x}\rangle^{2 \epsilon} |D_x|^\frac13  f  \right.\right)\right| \\
				&\leq (2\pi)^{-\frac12}  \int_{\mathbb{R}^2} e^{2 a A^{-\frac{1}{3}}|k|^{\frac{2}{3} } t }\langle k\rangle^{2 m   }\langle\frac{1}{k}\rangle^{2 \epsilon}  |k|^\frac13 |l|^{\frac13}  \left| \int_{\mathbb{R}} \mathcal{M}\left(k, D_y\right) f_k(y) l \Delta_l^{-1} \omega_l(y)  \py f_{k-l}(y) d y\right|  d k d l .
		\end{aligned}\end{equation}
		
		$\bullet$~Case 1:   $\frac{|k-l|}{2} \leq |k| \leq 2|k-l|$. From \eqref{eq:trick7} and $|l|^\frac13 \leq 3^{\frac13}|k-l|^\frac13 $, we infer that
		\begin{equation}\begin{aligned}  \label{eq:est of u2 py c 11'}
				&\quad\int_{\frac{|k-l|}{2} \leq|k| \leq 2|k-l|} e^{2 a A^{-\frac{1}{3}}|k|^{\frac{2}{3} } t }\langle k\rangle^{2 m   }\langle\frac{1}{k}\rangle^{2 \epsilon}  |k|^\frac13 |l|^\frac13    \\
				&  \qquad \qquad   \qquad  \times \left|  \int_{\mathbb{R}} \mathcal{M}\left(k, D_y\right) f_k(y)  l \Delta_l^{-1} \omega_l(y)  \py f_{k-l}(y) d y \right| d k d l \\
				&\leq  3^{\frac13}\cdot2^{m + \epsilon } (1 + \Xi)  \|e^{a A^{-\frac{1}{3}}|k|^{\frac{2}{3} }t}\langle k\rangle^{m  }\langle\frac{1}{k}\rangle^\epsilon  |k|^\frac13   \| f_k\|_{L_y^2}\|_{L_k^2}      \|e^{a A^{-\frac{1}{3}}|l|^{\frac{2}{3} } t }|l|\| \Delta_l^{-1} \omega_l \|_{L_y^{\infty}}\|_{L_l^1} \\
				&\qquad \times\|e^{a A^{-\frac{1}{3}}|k-l|^{\frac{2}{3} } t }\langle k-l\rangle^{m   }\langle\frac{1}{k-l}\rangle^\epsilon  |k-l|^\frac13 \| \partial_y f_{k-l}\|_{L_y^2}\|_{L_{k-l}^2} \\
				&\leq 3^{\frac13}\cdot2^{m + \epsilon} (1 + \Xi) \sqrt{ B(\epsilon, m)}   \|e^{a A^{-\frac{1}{3}}\left|D_x\right|^{\frac{2}{3} } t }\langle D_x\rangle^{m  }\langle\frac{1}{D_x}\rangle^\epsilon |D_x|^\frac13  f\|_{L^2}  \\
				&\qquad \times  \|e^{a A^{-\frac{1}{3}}\left|D_x\right|^{\frac{2}{3} } t }\langle D_x\rangle^m\langle\frac{1}{D_x}\rangle^\epsilon \partial_x \nabla \Delta^{-1} \omega\|_{L^2} \|e^{a A^{-\frac{1}{3}}\left|D_x\right|^{\frac{2}{3} } t }\langle D_x\rangle^{m  }\langle\frac{1}{D_x}\rangle^\epsilon \partial_y |D_x|^\frac13  f\|_{L^2}  .
		\end{aligned}\end{equation}

		$\bullet$~Case 2:  $2|k-l|<|k|$.
		\eqref{eq:trick8} and $1 \leq |l|^{\frac{1}{6}} |k-l|^{-\frac{1}{6}}$ yield that
		\begin{align}   \label{eq:est of u2 py c 21'}
			&\quad\no\int_{2|k-l|<|k|} e^{2 a A^{-\frac{1}{3}}|k|^{\frac{2}{3} } t }\langle k\rangle^{2 m  }\langle\frac{1}{k}\rangle^{2 \epsilon} |k|^\frac13 |l|^\frac13 \\
			&\no \qquad \qquad   \qquad  \times \left|  \int_{\mathbb{R}} \mathcal{M}\left(k, D_y\right) f_k(y)  l \Delta_l^{-1} \omega_l(y)  \py f_{k-l}(y) d y \right|  d k d l  \\
			&\no \leq 2^{m} \lt(\frac32\rt)^{\epsilon}(1 + \Xi)\|e^{a A^{-\frac{1}{3}}|k|^{\frac{2}{3} } t }\langle k\rangle^{m} \langle\frac{1}{k}\rangle^\epsilon   |k|^{\frac13}  \| f_k\|_{L_y^2}\|_{L_k^2}   \| e^{a A^{-\frac{1}{3}}|l|^{\frac{2}{3} } t }\langle l\rangle^{m }\langle\frac{1}{l}\rangle^\epsilon  |l|   \| \nabla_l \Delta_l^{-1} \omega_l \|_{L_y^2}\|_{L_l^2} \\
			&\no \quad \times\|e^{a A^{-\frac{1}{3}}|k-l|^{\frac{2}{3} } t }\langle k-l\rangle^m\langle\frac{1}{k-l}\rangle^\epsilon  |k-l|^\frac13 \| \partial_y f_{k-l}\|_{L_y^2}\|_{L_{k-l}^2}\||k-l|^{-\frac{1}{2}}\langle k-l\rangle^{-m}\langle\frac{1}{k-l}\rangle^{-\epsilon}\|_{L_{k-l}^2} \\
			&\no \leq 2^{m}\lt(\frac32\rt)^{\epsilon}(1 + \Xi)  \sqrt{ B(\epsilon, m)}    \|e^{a A^{-\frac{1}{3}}\left|D_x\right|^{\frac{2}{3} } t }\langle D_x\rangle^{m  } \langle\frac{1}{D_x}\rangle^\epsilon |D_x|^\frac13  f \|_{L^2}      \\
			& \quad \times  \|e^{a A^{-\frac{1}{3}}\left|D_x\right|^{\frac{2}{3} } t }\langle D_x\rangle^{m  }\langle\frac{1}{D_x}\rangle^\epsilon  \px \nabla \Delta^{-1}   \omega\|_{L^2} \|e^{a A^{-\frac{1}{3}}\left|D_x\right|^{\frac{2}{3} } t }\langle D_x\rangle^m\langle\frac{1}{D_x}\rangle^\epsilon \py  |D_x|^\frac13  f\|_{L^2}.
		\end{align}
		
		$\bullet$~Case 3:  $2|k|<|k-l|$. Using \eqref{eq:trick9},  \eqref{eq:trick4} and $|k-l|^{\frac16}\leq 2^{\frac16} |l|^{\frac16}$, one has
		\begin{align*}    
				&\quad   \int_{2|k|<|k-l|} e^{2 a A^{-\frac{1}{3}}|k|^{\frac{2}{3} } t }\langle k\rangle^{2 m   }\langle\frac{1}{k}\rangle^{2 \epsilon}  |k|^\frac13 |l|^\frac13  \\
				&  \qquad \qquad   \qquad  \times \left|   \int_{\mathbb{R}} \mathcal{M}\left(k, D_y\right) f_k(y)  l \Delta_l^{-1} \omega_l(y)  \py f_{k-l}(y) d y \right|  d k d l   \\
				& \leq  2^{\frac16}\lt(\frac32\rt)^{\epsilon} (1 + \Xi) \|e^{a A^{-\frac{1}{3}}|k|^{\frac{2}{3} } t }\langle k\rangle^m\langle\frac{1}{k}\rangle^\epsilon  |k|^\frac13   \| f_k\|_{L_y^2}\|_{L_k^2}  \|e^{a A^{-\frac{1}{3}}|l|^{\frac{2}{3} } t }\langle l\rangle^{ m  }\langle\frac{1}{l}\rangle^\epsilon|l|^{\frac{3}{2}}\| \Delta_l^{-1} \omega_l \|_{L_y^{\infty}}\|_{L_l^2} \\
				& \quad \times  \|e^{a A^{-\frac{1}{3}}|k-l|^{\frac{2}{3} } t }\langle k-l\rangle^{m  }\langle\frac{1}{k-l}\rangle^\epsilon   |k-l|^\frac13   \| \partial_y f_{k-l}\|_{L_y^2}\|_{L_{k-l}^2} \\
                &\quad\times \|\left(\langle\frac{1}{k}\rangle^\epsilon \chi_{\{\frac14\leq \epsilon<\frac12 \}} +\langle\frac{1}{k}\rangle^{\frac{1}{2}-\epsilon} \chi_{\{0<\epsilon<\frac14\}}\right)\langle k\rangle^{-m}\|_{L_k^2}\\
				& \leq  2^{\frac16}\lt(\frac32\rt)^{\epsilon}(1 + \Xi) \max\lt\{ \sqrt{B(\frac12 - \epsilon, m -\frac12)}, \sqrt{B(\epsilon, m -\frac12)}\rt\}      \\
				& \quad \times \|e^{a A^{-\frac{1}{3}}\left|D_x\right|^{\frac{2}{3} } t }\langle D_x\rangle^{m  }\langle\frac{1}{D_x}\rangle^\epsilon \partial_x \nabla \Delta^{-1} \omega \|_{L^2} \|e^{a A^{-\frac{1}{3}}\left|D_x\right|^{\frac{2}{3} } t }\langle D_x\rangle^{m  }\langle\frac{1}{D_x}\rangle^\epsilon \py |D_x|^\frac13  f  \|_{L^2}\\
                &\quad\times \|e^{a A^{-\frac{1}{3}}\left|D_x\right|^{\frac{2}{3} } t }\langle D_x\rangle^m\langle\frac{1}{D_x}\rangle^\epsilon |D_x|^\frac13 f \|_{L^2}. \alabel{eq:est of u2 py c 31'}
		\end{align*}

		Hence, it follows from \eqref{eq:est of u2 py c 01'}--\eqref{eq:est of u2 py c 31'} that
		\begin{equation*}\begin{aligned}
				&\quad\left|\Re\left(|D_x|^{\frac13}u_2 \py  f \left\lvert\, \mathcal{M} e^{2 a A^{-\frac{1}{3}}\left|D_x\right|^{\frac{2}{3} }t }\langle D_x\rangle^{2 m   }\langle\frac{1}{D_x}\rangle^{2 \epsilon} |D_x|^\frac13 f  \right.\right)\right| \\
				&\leq (2\pi)^{-\frac12} \lt(3^{\frac13}\sqrt{\frac43}+1+2^{-\frac13}\rt) 2^{m}\lt(\frac32\rt)^{\epsilon} (1 + \Xi) \max\lt\{ \sqrt{B(\frac12 - \epsilon, m -\frac12)}, \sqrt{B(\epsilon, m -\frac12)}\rt\}\\
				&\quad\times\|e^{a A^{-\frac{1}{3}}\left|D_x\right|^{\frac{2}{3} } t }\langle D_x\rangle^m\langle\frac{1}{D_x}\rangle^\epsilon |D_x|^\frac13 f \|_{L^2} \|e^{a A^{-\frac{1}{3}}\left|D_x\right|^{\frac{2}{3} } t }\langle D_x\rangle^{m  }\langle\frac{1}{D_x}\rangle^\epsilon \partial_x \nabla \Delta^{-1} \omega \|_{L^2} \\
				&\quad\times\|e^{a A^{-\frac{1}{3}}\left|D_x\right|^{\frac{2}{3} } t }\langle D_x\rangle^{m  }\langle\frac{1}{D_x}\rangle^\epsilon \py |D_x|^\frac13  f  \|_{L^2},
		\end{aligned}\end{equation*}
		which, along with \eqref{eq:dx13 u1pxf}, implies
		\begin{equation*}\begin{aligned}
				&\quad\left|\Re\left( |D_x|^\frac13 u \cdot\nabla  f \left\lvert\, \mathcal{M} e^{2 a A^{-\frac{1}{3}}\left|D_x\right|^{\frac{2}{3} }t }\langle D_x\rangle^{2 m   }\langle\frac{1}{D_x}\rangle^{2 \epsilon} |D_x|^\frac13 f\right.\right)\right| \\
				&\leq (2\pi)^{-\frac12 } \cdot 2^{m}\lt(\frac32\rt)^{\epsilon} (1 + \Xi) \max\lt\{ \sqrt{B(\frac12 - \epsilon, m -\frac12)}, \sqrt{B(\epsilon, m -\frac12)}\rt\} \\
				&\quad\times\lt[ 3^{\frac13}\sqrt{\frac43} \|e^{a A^{-\frac{1}{3}}\left|D_x\right|^{\frac{2}{3} } t }\langle D_x\rangle^{m   }\langle\frac{1}{D_x}\rangle^\epsilon |D_x|^\frac23  f \|_{L^2} ^\frac32  \|e^{a A^{-\frac{1}{3}}\left|D_x\right|^{\frac{2}{3} } t }\langle D_x\rangle^{m   }\langle\frac{1}{D_x}\rangle^\epsilon  \omega \|_{L^2}\rt.\\
				&\lt.\qqquad\times\|e^{a A^{-\frac{1}{3}}\left|D_x\right|^{\frac{2}{3} } t }\langle D_x\rangle^m\langle\frac{1}{D_x}\rangle^\epsilon \px  |D_x|^\frac13     f \|_{L^2}^\frac12 
				\rt.\\
                &\lt.\qquad+\lt(2^{\frac13}+2^{\frac16} + 3^{\frac13}\sqrt{\frac43}+1+2^{-\frac13} \rt)\|e^{a A^{-\frac{1}{3}}\left|D_x\right|^{\frac{2}{3} } t }\langle D_x\rangle^{m   }\langle\frac{1}{D_x}\rangle^\epsilon \nabla  |D_x|^\frac13  f\|_{L^2}\rt.\\
				&\lt.\qqquad\times\|e^{a A^{- \frac{1}{3}}\left|D_x\right|^{\frac{2}{3} } t }\langle D_x\rangle^m\langle\frac{1}{D_x}\rangle^\epsilon |D_x|^\frac13   f \|_{L^2} \|e^{a A^{-\frac{1}{3}}\left|D_x\right|^{\frac{2}{3} } t }\langle D_x\rangle^{m   }\langle\frac{1}{D_x}\rangle^\epsilon \px \py \Delta^{-1} \omega \|_{L^2}  \rt].
		\end{aligned}\end{equation*}
		Hence, denote
		\begin{align*}
				C_{hl}  &:= (2\pi)^{-\frac12}  (1 + \Xi) \lt[3^\frac13\sqrt{\frac43}   \lt( \frac{\theta_1\Xi}{2\pi}  - 2a\rt)^{-\frac34} \cdot 2^{-\frac14} \rt.\\
				&\lt. \qquad\qquad+  \lt(2^{\frac13}+2^{\frac16} + 3^{\frac13}\sqrt{\frac43}+1+2^{-\frac13} \rt) \left(\frac{\Xi}{2\pi}\right)^{-\frac12} \left(2 - \frac{\theta_2\Xi}{2\pi}\right)^{-\frac12}  \rt],
		\end{align*}
		and then the proof is complete.
	\end{proof}

\section{Numerical verification of constants}
\label{appB}
In this appendix, we provide the numerical evaluation of several constants
appearing in Remark~\ref{rem1.2}. The computations were performed in \textsc{Matlab} using double precision arithmetic. All numerical values below are rounded to three decimal places.

Note that when $a=  (2000\pi)^{-1}$ and $\Xi=\theta_1 = \theta_2=1$,
we have
\begin{gather*}
    C_l < 12.592,\quad C_{st}<12.089,\quad C_{hl}<13.052, \quad C_{ch1}<4.111,\\
    C_{fl1}<5.179, \quad C_{ch2} < 3.551, \quad C_{fl2}<1.472, \quad C_{ch3}<3.915, \quad C_{fl3}<1.583.
\end{gather*} 

When $\alpha>0$, let $\epsilon=1/4$ and $m=9/{10}$, then $ r_3<4.977\max\lt\{1,\alpha^{-\frac14}\rt\}$. Using the upper bound above, we consider solving the following equation:
		\begin{equation*}\begin{aligned}
				1.01^3 \cdot 2 \cdot 4.977\left( \frac{  C_l  }{1.01x^\frac23}  +  \frac{2C_{st} + C_{hl}  +  C_{ch1} + C_{fl1}  }{x^\frac12}  + \frac{C_{ch2} + C_{fl2} }{x^\frac13}\right) = 1,
		\end{aligned}\end{equation*}
which shows $ x< 1.013\times 10^6$. Now, when $A>\bar{A}_3:=1.013\times 10^6 \max\lt\{1,\alpha^{-\frac34}\rt\} K^9 $, \eqref{eq:yz1} holds.

When $\alpha=0$, let $\epsilon=5/{12}$ and $m=7/{10}$, then $r_2<7.841$. Similarly, we treat with the following equation:
\begin{equation*}\begin{aligned}
				1.01^3 \cdot 2 \cdot 7.841 \left( \frac{  C_l  }{1.01x^\frac23}  +  \frac{2C_{st} + C_{hl}  +  2^\frac14 C_{ch1} + 2^\frac14 C_{fl1}  }{x^\frac12}  + \frac{2^\frac14 C_{ch3} + 2^\frac14C_{fl3} }{x^\frac13}\right) = 1,
		\end{aligned}\end{equation*}
which implies $ x< 4.673\times 10^6$. If $A>\bar{A}_4:=4.673\times 10^6  K^9 $, \eqref{eq:yz2} holds.

    \qquad\\

	\noindent {\bf Acknowledgments.}
	W. Wang was supported by National Key R\&D Program of China (No. 2023YFA1009200) and NSFC under grant 12471219.\\

	\noindent {\bf Declaration of Interests.}
	The authors report no conflict of interest.\\
	
	\noindent {\bf Data availability.}
	No data was used in this paper.

\end{document}